\documentclass[11pt]{article}
\usepackage{mymacros}
\usepackage{mhequ}
\usepackage{comment}
\usepackage{mathtools}
\usepackage{color}

\mathtoolsset{showonlyrefs}

\usepackage{soul}

\newcommand{\PP}{\bm{P}}

\newcommand{\EE}{\bm{E}}

\newcommand{\RR}{\mathbb{R}}

\newcommand{\GFF}{\textup{GFF}}

\newcommand{\fineGFF}{X_K^f}
\newcommand{\tPhisP}{\tilde \Phi_s}
\newcommand{\sqrtpi}{} 
\newcommand{\ZDM}{{\rm Z}}
\newcommand{\scaling}{\sqrt{8\pi}}
\newcommand{\Max}{\max_{\Omega_\epsilon}}
\newcommand{\PhisKP}{\Phi_s}
\newcommand{\DiscTorusGrid}{\Omega_\epsilon^{\delta}}
\newcommand{\coarsesKGauss}{\tilde X_{s,K}^c}

\newcommand{\limcoarsePhiP}{Z_{s,K}^{c,0}}  
\newcommand{\scalinglog}{\frac{2}{\sqrt{2\pi}}}
\newcommand{\bu}{\mathbf{u}}
\newcommand{\coarsePhiP}{Z_{s,K}^c}  

\newcommand{\wick}[1]{: #1 :}
\newcommand{\besov}[3]{{B_{#1,#2}^{#3}}}

\newcommand{\minimiser}{\bar u}

\newcommand{\vE}{v^E}
\newcommand{\PhiE}{\Phi^{\cP,E}}
\newcommand{\minimiserE}{u^E}

\title{Multiscale coupling and the maximum of $\cP(\phi)_2$ models on the torus}

\author{Nikolay Barashkov \footnote{University of Helsinki, Department of Mathematics and Statistics. E-mail: {\tt nikolay.barashkov@helsinki.fi}}  \and Trishen S. Gunaratnam \footnote{Universit\'e de Gen\`eve,  { Section de math\'ematiques}. E-mail: {\tt trishen.gunaratnam@unige.ch }}
  \and Michael Hofstetter \footnote{University of Cambridge, Statistical Laboratory, DPMMS. E-mail: {\tt mh901@cam.ac.uk}.}  }

\date{}

\begin{document}

\maketitle

\begin{abstract}
We establish a coupling between the $\cP(\phi)_2$ measure and the Gaussian free field on the two-dimensional unit torus at all spatial scales, quantified by probabilistic regularity estimates on the difference field. Our result includes the well-studied $\phi^4_2$ measure.

The proof uses an exact correspondence between the Polchinski renormalisation group ap\-proach, which is used to define the coupling,
and the Bou\'e-Dupuis stochastic control representation for $\cP(\phi)_2$. 
More precisely, we show that the difference field is obtained from a specific minimiser of the variational problem.
This allows to transfer regularity estimates for the small-scales of minimisers, obtained using discrete harmonic analysis tools, to the difference field.  

As an application of the coupling, we prove that the maximum of the $\cP(\phi)_2$ field on the discretised torus with mesh-size $\epsilon > 0$ converges in distribution to a randomly shifted Gumbel distribution as $\epsilon \rightarrow 0$. 
\end{abstract}



\tableofcontents

\section{Introduction}

\subsection{Model and main result}

We study global probabilistic properties of the $\cP(\phi)_2$ Euclidean quan\-tum field theories on the two dimensional unit torus $\Omega=\T^2$.
These objects are measures $\nu^\cP$ on the space of distributions $S'(\Omega)$ that are formally given by
\begin{equs}
\label{eq:cont-pphi-density}
\nu^\cP(d\phi) \propto \exp\Big(- \int_\Omega \cP(\phi(x)) dx\Big) \nu^\GFF(d\phi),
\end{equs}
where $\cP$ is a polynomial of even degree with positive leading  coefficient
and $\nu^\GFF$ is the law of the massive Gaussian free field, i.e.\  the mean zero Gaussian measure with covariance $(-\Delta + m^2)^{-1}$ for some arbitrary mass $m>0$ that is fixed throughout this article.
The most famous example is when $\cP(\phi)$ is a quartic polynomial, in which case $\nu^\cP$ is known as the $\phi_2^4$ measure in finite volume with periodic boundary condition.

The origin of these measures lies in constructive quantum field theory, where they arise as Wick rotations of interacting bosonic quantum field theories in $1+1$ Minkowski space-time, see for instance \cite{Syma69} and \cite{MR0376002}.
Their construction was first achieved in 2D by Nelson in \cite{MR0210416}, see also the books by Simon \cite{MR0489552} and Glimm and Jaffe \cite[Section 8]{MR887102}.
Indeed, fields sampled from $\nu^\GFF$ are almost surely distributions in a Sobolev space of negative regularity, and hence cannot be evaluated pointwise.
Therefore, since there is no canonical definition of nonlinear functions of distributions, the density in \eqref{eq:cont-pphi-density} is ill-defined.
This can be seen concretely, if one imposes a small scale cut-off. Then the cut-off measures are well-defined, but as the cut-off is removed, one encounters so-called ultraviolet divergences.
It is well-known that a suitable renormalisation procedure is needed to remove these divergences. 
 
In order to give meaning to \eqref{eq:cont-pphi-density}, we  view it as a limit of renormalised lattice approximations. 
For $\epsilon>0$ let $\Omega_\epsilon = \Omega \cap (\epsilon \Z^2)$ be the discretised unit torus and denote $X_\epsilon = \R^{\Omega_\epsilon}= \{\varphi\colon \Omega_\epsilon \to \R\}$.
We assume throughout that $1/\epsilon \in \N$.  
Moreover, write $\Delta^\epsilon$ for the discrete Laplacian acting on functions $f \in  X_\epsilon$ by $\Delta^\epsilon f (x) = \epsilon^{-2} \sum_{y\sim x} \big ( f(y) - f(x) \big)$, where $y \sim x$ denotes that $x,y \in \Omega_\epsilon$ are nearest neighbours. 
Let $\nu^{\GFF_\epsilon}$ be the centred Gaussian measure on $X_\epsilon$ with covariance $(-\Delta^\epsilon + m^2)^{-1}$,
i.e.\ the law of the massive discrete Gaussian free field on $\Omega_\epsilon$. 
Note that as $\epsilon\to 0$, we have for all $x\in \Omega_\epsilon$
\begin{equs}
\label{eq:variance-gff}
c_\epsilon \equiv \var\big( \Phi_x^{\GFF_\epsilon} \big) = \frac{1}{2\pi} \log \frac{1}{\epsilon} + O(1), \qquad \Phi^{\GFF_\epsilon} \sim \nu^{\GFF_\epsilon}.
\end{equs}
To compensate this small scale divergence we renormalise the polynomial $\cP$ in \eqref{eq:cont-pphi-density} by interpreting it as Wick-ordered.
For $n\in \N$, we define the Wick powers 
\begin{equation}
\label{eq:wick-monomials}
\wick{\phi^n(x)}_\epsilon  = c_\epsilon^{n/2} H_n(\phi(x)/\sqrt{c_\epsilon}),
\end{equation}
where $H_n$ denotes the $n$-th Hermite polynomial and $c_\epsilon$ is as in \eqref{eq:variance-gff}. For instance, when $n=4$, we have
\begin{equation}
\wick{\phi^4(x)}_\epsilon = \phi^4(x) - 6 c_\epsilon \phi^2(x) + 3 c_\epsilon^2.
\end{equation} 
Note that by the multiplication with $c_\epsilon^n$ we ensure that the leading coefficient of the Wick power equals $1$.

It can be shown that as $\epsilon\to 0$ the Wick ordered monomials for the Gaussian free field converge to well-defined random variables $\wick{\Phi^n}$ with values in the space of distributions  $S'(\Omega)$.
Moreover, the collection $(\wick{\Phi^n})_{n\in \N}$ is orthogonal in $L^2(\mu^\GFF)$ and satisfies 
\begin{equation}
\E[\wick{\Phi^n(x)} \wick{\Phi^n(y)}]= n! \, G^n(x-y), \qquad \Phi \sim \nu^{\GFF},
\end{equation}
where $G = (-\Delta + m^2)^{-1}$ is the Green function of the Laplacian on $\Omega$.
This gives the Wick powers the interpretation of powers of the Gaussian free field.

For general polynomials of the form $\cP \colon \R \to \R, r\mapsto \sum_{k=1}^{N} a_k r^k$ of even degree $N \in 2\N$,  and coefficients satisfying $a_k \in \R$, $1 \leq k < N$, and $a_N>0$, we define the Wick ordering
\begin{equation}
\label{eq:wick-polynomials}
\wick{\cP(\phi(x))}_\epsilon =  \sum_{k=1}^{N} a_k \wick{\phi^k(x)}_\epsilon.
\end{equation}
When it is clear from the context or when there is no conceptual difference between the discrete Wick power and the continuum limit, we drop $\epsilon$ from the notation.
Note that, by a simple scaling argument, for the construction of the measures \eqref{eq:cont-pphi-density} we may assume without loss of generality that $\cP$ has no constant term.

As for the monomials, the Wick ordered polynomials converge to well-defined random variables with values in a space of distributions as $\epsilon \to 0$.
However, as $\epsilon \to 0$ one loses the uniform lower bound of $\wick{\cP}_\epsilon$.
Indeed, due to the logarithmic divergence of $c_\epsilon$ in \eqref{eq:variance-gff}, we have for $\wick{\cP}_\epsilon$ as in \eqref{eq:wick-polynomials}
\begin{equation}
\inf_{x\in \Omega_\epsilon} \wick{\cP}_\epsilon \, \geq - O(|\log \epsilon|^{n}), \qquad \epsilon \to 0
\end{equation}
for some $n\in \N$.
Therefore, the exponential integrability (and hence the construction of \eqref{eq:cont-pphi-density} as a limit) is not immediate.  
Nelson leveraged the polylogarithmic divergence in $d=2$ to give a rigorous construction of the continuum object $\nu^\cP$,
which can be obtained as a weak limit of the regularised measures
$\nu^{\cP_\epsilon}$ on $X_\epsilon$ defined by
\begin{equation}
\label{eq:nu-p-eps}
\nu^{\cP_\epsilon}(d \phi)
\propto
\exp\Big( -  \epsilon^2 \sum_{x\in \Omega_\epsilon} \wick{\cP(\phi(x))}_\epsilon  \Big) \nu^{\GFF_\epsilon}(d\phi).
\end{equation}
Other constructions also exist, for example via stochastic quantisation techniques \cite{MR3825880}, the Bou\'e-Dupuis stochastic control representation \cite{MR4173157}, and martingale methods \cite{2003.12535}. 

Fields sampled from $\nu^\cP$ have similar small scale behaviour to the Gaussian free field on $\Omega$.
Indeed, the two measures are mutually absolutely continuous and, as such, $\nu^\cP$ is supported on a space of distributions, not functions.
One of the central themes of this article is to quantify the difference in short-distance behaviour between these two measures,
and to see whether the similarities on small scales lead to similarities on a global scale, e.g.\ universal behaviour of the maxima. 

Finally, let us remark that similar Wick renormalisation procedures have been successfully used to construct other continuum non-Gaussian Euclidean field theories in 2D,
see for instance \cite{MR4399156}, for the sine-Gordon field and \cite{BarashkovDeVecchi2021Elliptic} for the sinh-Gordon field. 

Our main result is a probabilistic coupling between the $\cP(\phi)_2$ field and the Gaussian free field at all scales,
which allows to write the field of interest as a sum of the Gaussian free field and a (non-independent) regular difference field.
The methods we use rely on a stochastic control formulation developed for the $\phi_2^4$ field in \cite{MR4173157}
and the essentially equivalent non-perturbative Polchinski renormalisation group approach that was used in \cite{MR4399156} for the sine-Gordon field. 
Our point of view in this work resembles that in the latter reference, 
for which we use a similar notation. 
Specifically, we construct the $\epsilon$-regularised $\cP(\phi)_2$ field as the solution to the high dimensional SDE
\begin{equation}
\label{eq:pphi-sde-construction-intro}
d \Phi_t^{\cP_\epsilon,E} 
= - \dot c_t^\epsilon \nabla v_t^{\epsilon,E} (\Phi_t^{\cP_\epsilon,E}) dt+ (\dot c_t^\epsilon)^{1/2} d W_t, \qquad \Phi_\infty^{\cP_\epsilon,E} = 0,
\end{equation}
where $(\dot c_t^\epsilon)_{t\in [0,\infty]}$ is associated to the Pauli-Villars decomposition of the Gaussian free field covariance defined for $t \in (0,\infty)$ by
\begin{equs}
\label{eq:def-pauli-villars-intro}
 c_t^\epsilon
 =
 ( -\Delta^\epsilon + m^2 + 1/t)^{-1}, \qquad \dot c_t^\epsilon = \frac{d}{dt}c_t^\epsilon,
\end{equs}
and with $c_0^\epsilon =0$ and $c_\infty^\epsilon = (-\Delta^\epsilon + m^2)^{-1}$,
$W$ is a Brownian motion in $\Omega_\epsilon$, and $v_t^{\epsilon,E}$ is the renormalised potential to be defined later.
The Pauli-Villars decomposition could be replaced by, for example, a heat-kernel decomposition,
but it is technically convenient to work with the former rather than the latter,
see Remark \ref{rem:pauli-villars} for a more precise discussion on the rigidity concerning the choice of scale decomposition. 

The energy cut-off $E>0$ is sufficient to ensure that the SDE \eqref{eq:pphi-sde-construction-intro} is well-defined and has strong existence, and is removed by taking $E\to \infty$.
The stochastic integral on the right-hand side of \eqref{eq:pphi-sde-construction-intro} corresponds to a scale decomposition of the Gaussian free field,
i.e.\ setting
\begin{equation}
\label{eq:decomposed-gff-intro}
\Phi_t^{\GFF_\epsilon} = \int_t^\infty (\dot c_s^\epsilon)^{1/2} d W_s
\end{equation}
we obtain a Gaussian process $(\Phi_t^{\GFF_\epsilon}\big)_t$ started at $t=\infty$ with $\Phi_\infty^{\GFF_\epsilon} = 0$ and such that $\Phi_0^{\GFF_\epsilon} \sim \nu^{\GFF_\epsilon}$.
Thus, going back to \eqref{eq:pphi-sde-construction-intro} the difference field corresponds to the finite variation term on the right-hand side of the SDE.

We state the main result in the following theorem on the level of regularisations.
For $\alpha \in \R$ let $H^\alpha(\Omega_\epsilon)\equiv H^\alpha$ be the (discrete) Sobolev space of regularity $\alpha$, see Section \ref{ssec:sobolev} for a precise definition.
Moreover, we denote by $C_0([0, \infty), \cS)$ the space of continuous sample paths with values in a metric space $(\cS, \|\cdot\|_\cS)$ that vanish at infinity equipped with the topology of compact convergence.

\begin{theorem}
\label{thm:coupling-pphi-to-gff-eps}
There exists a process $\Phi^{\cP_\epsilon}\in C_0([0,\infty), H^{-\kappa})$ for any $\kappa >0$ such that
  \begin{equation}
    \label{eq:coupling-pphi-to-gff-eps}
    \Phi_t^{\cP_\epsilon}
    = \Phi_t^{\Delta_\epsilon}
      + \Phi_t^{\GFF_\epsilon}, \qquad  \Phi_0^{\cP_\epsilon} \sim \nu^{\cP_\epsilon} ,
    \end{equation}
   where the difference field $\Phi^{\Delta_\epsilon}$ satisfies
\begin{equs}
\label{eq:phi-delta-h1}
 &\sup_{\epsilon >0 } \sup_{t\geq 0} \E [\|\Phi_t^{\Delta_\epsilon}\|_{H^{1}}^2] < \infty,
 \\
\label{eq:phi-delta-h2}
&\sup_{\epsilon>0} \sup_{t\geq t_0} \E [\|\Phi_t^{\Delta_\epsilon}\|_{H^{2}}^2] < \infty
\end{equs}
for any $t_0>0$, and for any $\alpha \in [0,2)$ and some large enough $L\geq 4$
\begin{align} 
\label{eq:phi-delta-1-2}
&\sup_{\epsilon >0 } \sup_{t\geq 0} \E [\|\Phi_t^{\Delta_\epsilon}\|_{H^{\alpha}}^{2/L}] < \infty, \\
\label{eq:phi-delta-to-0}
&\sup_{\epsilon>0} \E[ \| \Phi_t^{\Delta_\epsilon}-\Phi_0^{\Delta_\epsilon}\|_{H^{\alpha}}^{2/L} ]\to 0 \text{~as~} t\to 0.
\end{align}  
Moreover, for any $t>0$, $\Phi^{\GFF_\epsilon}_0-\Phi_t^{\GFF_\epsilon}$ is independent of $\Phi_t^{\cP_\epsilon}$.
\end{theorem}

Let us make two important remarks concerning the optimality of the regularity estimates in Theorem \ref{thm:coupling-pphi-to-gff-eps}.
First, the restriction to $\alpha\in [0,2)$ in \eqref{eq:phi-delta-1-2} and \eqref{eq:phi-delta-to-0} seems to be optimal, at least with respect to our method.
Other reasonable choices of $(\dot c_t^\epsilon)_{t\in [0,\infty]}$, see again Remark \ref{rem:pauli-villars}, should also lead to the same estimates \eqref{eq:phi-delta-h1}, \eqref{eq:phi-delta-1-2} and \eqref{eq:phi-delta-to-0}.
Second, we do not expect that the probabilistic regularity estimates on the difference field $\Phi_t^{\Delta_\epsilon}$, $t\geq 0$ in Theorem \ref{thm:coupling-pphi-to-gff-eps} can be replaced by deterministic ones,
in contradistinction to the case of the sine-Gordon field considered in \cite{MR4399156}.
This is because $\Phi_t^{\Delta_\epsilon}$ is related to $\nabla v_t^{\epsilon,E}(\phi)$, $\phi\in X_\epsilon$ in \eqref{eq:polchinski-pphi-cut-off}, which is, unlike for the sine-Gordon field, unbounded in $\phi$ as $E\to \infty$.
In fact, since the constant $L$ depends linearly on the degree of the polynomial $\cP$, the bounds in Theorem \ref{thm:coupling-pphi-to-gff-eps} become weaker as the degree of $\cP$ increases.
This suggests that the larger the growth of the nonlinearity in \eqref{eq:cont-pphi-density}, the weaker the bounds for the difference field. Note that, in the case of the $\phi^4_2$ field, we can take $L=4$.

As a corollary of Theorem \ref{thm:coupling-pphi-to-gff-eps}, we also obtain the following statement for the continuum, i.e.\ $\epsilon=0$.
We stress that, in the statement below, $H^\alpha = H^\alpha(\Omega)$ refers to the usual, i.e.\ continuum, Sobolev space of regularity $\alpha$ over $\Omega$.

\begin{corollary}
\label{cor:pphi-coupling-continuum}
There exists a process $\Phi^{\cP} \in C_0([0,\infty), H^{-\kappa})$ for every $\kappa>0$ such that
\begin{equation}
\label{eq:coupling-pphi-to-gff-cont}
\Phi_t^{\cP}
= \Phi_t^{\Delta}
+ \Phi_t^{\GFF},
\end{equation}
where $\Phi_0^\cP$ is distributed as the continuum $\cP(\phi)_2$ field and $\Phi_0^\GFF$ is distributed as the continuum Gaussian free field.
For the difference field $\Phi^\Delta$, the analogous estimates as for $\Phi_t^{\Delta_\epsilon}$ in Theorem \ref{thm:coupling-pphi-to-gff-eps}  hold in the continuum Sobolev spaces. 
Finally, for any $t>0$, $\Phi_0^\GFF- \Phi_t^\GFF$ is independent of $\Phi_t^\cP$.
\end{corollary}

As we shall see in the proof of Corollary \ref{cor:pphi-coupling-continuum}, we construct the continuum processes $\Phi^\Delta$ in \eqref{eq:coupling-pphi-to-gff-cont} as a weak limit of $(\Phi^{\Delta_\epsilon})_\epsilon$ as $\epsilon \to 0$ along a subsequence, thereby establishing the existence of $\Phi^\cP$.
While the convergence as processes only holds along a subsequence, we prove that every subsequential limit of $(\Phi^{\Delta_\epsilon})_\epsilon$ and hence of $(\Phi^{\cP_\epsilon})_\epsilon$ has the same law.

The key ingredient in the proof of our main result, Theorem \ref{thm:coupling-pphi-to-gff-eps}, is an exact correspondence between two different stochastic representations of $\cP(\phi)_2$:
a stochastic control representation, called the Bou\'e-Dupuis representation, and the Polchinski renormalisation group approach, see Section \ref{ssec:correspondence}.
More precisely, the difference field is directly related to a special minimiser of the stochastic control problem.
We use this correspondence to transfer fractional moment estimates on the minimiser to the difference field. 

In the case of the sine-Gordon field considered in \cite{MR4399156}, the proof of the analogous estimates on $\Phi^{\Delta_\epsilon}$ relies heavily on deterministic estimates on the gradient of the renormalised potential $v_t^\epsilon$,
which is enabled by a Yukawa gas representation from \cite{MR914427}.
Such a representation is not available in the present case, as the nonlinearity $\cP(\phi)$ is not periodic in the field variables $(\phi_x)_{x\in \Omega_\epsilon}$.
Therefore, our approach is more robust in the sense that we do not need to know the gradient of the renormalised potential,
for which a uniform bound may not be available, 
and we do not need the periodicity requirement on the potential.

It would be of interest to extend our approach to other continuum Euclidean field theories. On the one hand, we believe that our approach can be extended to treat fields with more general nonlinearities in the log-density, such as the sinh-Gordon model. On the other hand, by combining our techniques with paracontrolled ansatz techniques developed in \cite{MR4173157},
it may be possible to analyse the sine-Gordon field up to the same regime as treated in \cite{MR4399156}.
Note that our techniques would a priori yield probabilistic estimates on the difference field, while the precise analysis of the gradient of the renormalised potential in \cite{MR4399156} yields deterministic estimates. 
It would be interesting to see whether one can recover similar deterministic estimates with our approach.
For $\beta < 4\pi$, this seems possible by using methods developed in \cite{Barashkov2022SineGordon}.
A treatment of the full subcritical regime $\beta < 8\pi$ using either method would be of great interest.
Let us mention that in the setting of \textit{fermionic} Euclidean field theories, such couplings have been constructed for a $\phi^{4}$-type model in the full subcritical regime in \cite{DeVecchi2022Stochastic} using a Forward-Backward SDE approach.
Although the objects treated are not the same (recall that we are interested in \textit{bosonic} Euclidean field theories),
on a high level their general framework is closely related to our point of view here and may be of help in analysing sine-Gordon for $\beta < 8\pi$. 

Finally, we remark that similar couplings can be obtained from stochastic quantisation, for instance with the methods of \cite{MR4164267} or \cite{MR3951704},
which allow to write the $\cP(\phi)_2$ field as a sum of a Gaussian free field and a non-Gaussian regular field with values in $H^1(\Omega)$.
However, the couplings obtained in this way do not have the independence property stated in Corollary \ref{cor:pphi-coupling-continuum}.
This property together with the fact that our regularity estimates imply $L^\infty$ bounds for $\Phi_t^{\Delta}$, $t\geq 0$ enables us to study the distribution of the centred maximum of the field as described below.

\subsection{Application to the maximum of the $\cP(\phi)_2$ field}

The strong coupling to the Gaussian free field in Theorem \ref{thm:coupling-pphi-to-gff-eps} is a useful tool to study novel probabilistic aspects of $\cP(\phi)_2$ theories that go beyond the scope of the current literature. 
As an illustration, we investigate the global centred maximum of the regularised fields $\Phi^{\cP_\epsilon}$ defined as
\begin{equs}
M^\epsilon 
= 
\Max \Phi^{\cP_\epsilon} \qquad \Phi^{\cP_\epsilon}\sim \nu^{\cP\epsilon}.
\end{equs}
It is clear that as $\epsilon \to \infty$ this random variable diverges, since the limiting field takes values in a space of distributions. For the (massive) Gaussian free field, i.e.\ the case $\cP(\phi) = 0$, this divergence has been quantified in \cite{MR2846636} by
\begin{equs}
\label{eq:expectation-maximum}
\E[M^\epsilon] 
=
 \mathfrak{m}_\epsilon + O(1), \qquad \mathfrak{m}_\epsilon= \frac{1}{\sqrt{2\pi} } (2\log \frac{1}{\epsilon}-\frac{3}{4}\log \log \frac{1}{\epsilon}).
\end{equs}
Moreover, in this reference it is also shown that the sequence of centred maxima $(M^\epsilon - \mathfrak{m}_\epsilon)_\epsilon$ is tight.
We remark that these results were initially proved for the massless Gaussian free field on a box with Dirichlet boundary condition rather than on the torus, but the arguments are not susceptible to a different choice of boundary.
The minor difference between the prefactor $1/\sqrt{2\pi}$ in \eqref{eq:expectation-maximum} to \cite{MR2846636} and various other references comes from our different scaling of the fields, see \eqref{eq:variance-gff}.

It is clear that the coupling in Theorem \ref{thm:coupling-pphi-to-gff-eps} for $t=0$ together with the properties of $\Phi^\Delta$ imply \eqref{eq:expectation-maximum} and tightness of the centred global maxima for the $\cP(\phi)_2$ field.
Indeed, by the standard Sobolev embedding in $d=2$, the Sobolev norm in \eqref{eq:phi-delta-1-2} can be replaced by the $L^\infty$-norm,
and hence, the maximum of the $\cP(\phi)_2$ field and that of the GFF differ by a random variable with finite fractional moment.
In particular, this implies that \eqref{eq:expectation-maximum} also holds for general even polynomials $\cP$ with positive leading coefficient.

Exploiting also the larger scales $t>0$ of \eqref{eq:coupling-pphi-to-gff-eps} allows to understand the $O(1)$ terms in $M^\epsilon - \mathfrak{m}_\epsilon$ and establish the following convergence in distribution as $\epsilon \to 0$.

\begin{theorem}
  \label{thm:convergence-to-gumbel}
The centred maximum of the $\epsilon$-regularised $\cP(\phi)_2$ field $\Phi^{\cP_\epsilon} \sim \nu^{\cP_\epsilon}$ converges in distri\-bu\-tion  as  $\epsilon \to 0$
  to a randomly shifted Gumbel distribution, i.e.\
  \begin{equation} 
  \label{e:thm-max-convergence}
    \max_{\Omega_\epsilon} \Phi^{\cP_\epsilon} - \mathfrak{m}_\epsilon
    \to \frac{1}{\sqrt{8\pi}} X + \frac{1}{\sqrt{8\pi}}\log \ZDM^{\cP} + b,
  \end{equation}
  where $\ZDM^{\cP} $ is a nontrivial positive random variable,
  $X$ is an independent standard Gumbel random variable,
  and $b$ is a deterministic constant.
\end{theorem}

The analogous result for the Gaussian free field was proved in \cite{MR3433630} and later generalised to log-correlated Gaussian fields in \cite{MR3729618},
and to the (non-Gaussian) continuum sine-Gordon field in \cite{MR4399156}. 
The proof in the latter reference relies on a coupling between the Gaussian free field and the sine-Gordon field, 
in essence similar to Theorem \ref{thm:coupling-pphi-to-gff-eps}, and a generalisation of all key result in \cite{MR3433630} to a non-Gaussian regime.
Here we follow a similar strategy to establish Theorem \ref{thm:convergence-to-gumbel} verifying that the technical results in \cite[Section 4]{MR4399156} also hold under the weaker assumptions on the term $\Phi^\Delta$ in Theorem \ref{thm:coupling-pphi-to-gff-eps}.

It is believed that the limiting law of the centred maximum is universal in the sense that if $\Phi^\epsilon$ belongs to a larger class of Gaussian or non-Gaussian log-correlated field,
then the fluctuations of the centred maximum are of order $1$ and moreover, there is a sequence $a_\epsilon$, such that
\begin{equation}
\label{eq:universality-maximum}
\P(\Max \Phi^\epsilon - a_\epsilon \leq x) \to  \E[ \exp(-C\ZDM e^{-cx})], \qquad x\in \R
\end{equation}
for some positive constants $c,C>0$ and a positive random variable $\ZDM$; see for instance \cite{PhysRevE.63.026110}.
The expectation value in \eqref{eq:universality-maximum} is the distribution function of a randomly shifted Gumbel distribution,
which is obtained from the deterministic Gumbel distribution function and averaging over the random shift $\log \ZDM$.
In particular, the weak convergence for $\max_{\Omega_\epsilon}\Phi^{\cP_\epsilon}$ in Theorem \ref{thm:convergence-to-gumbel} can be equivalently stated as in \eqref{eq:universality-maximum} by setting $a_\epsilon = \mathfrak{m}_\epsilon$, $c = \sqrt{8\pi}$, $C= e^{\sqrt{8\pi}b}$ and $\ZDM = \ZDM^{\cP}$.
In recent years there has been substantial progress on the extremal behaviour of log-correlated fields and related models, thus confirming the conjectured behaviour of the maximum.
For Gaussian fields, in particular the discrete Gaussian free field, the the vast majority of questions centred around the maximum have been answered thanks to the works \cite{MR3433630}, \cite{MR3729618}, \cite{MR4089492} as well as \cite{MR3509015}, \cite{MR3787554}.
As various key methods in the proof of these results only apply to Gaussian fields, the picture in the non-Gaussian regime is less complete.
Important recent works on non-Gaussian models include \cite{MR4399156}, \cite{MR4164451},\cite{MR4332224}, \cite{SchweigerZeitouni2022GFFinRE} and \cite{MR3933043}.
There is also a surprising relation between log-correlated processes and the extreme values of characteristic polynomials of certain random matrix ensembles as well as the maximum of the Riemann zeta function in a typical short interval on the critical line,
which was first described and investigated in \cite{108.170601} and \cite{MR3151088}.
Subsequent works in this direction include \cite{MR3594368}, \cite{MR3848391} and \cite{MR3848227}.
For a survey on recent developments we refer to \cite{MR4407477}.

The random variable $\ZDM^{\cP}$ is believed to be a multiple of the critical multiplicative chaos of the field,
also known as the derivative martingale, which can be obtained as the weak limit of 
\begin{equation}
\label{eq:derivative-martingale}
\ZDM^{\cP_\epsilon} = \epsilon^2 \sum_{\Omega_\epsilon} (\frac{2}{\sqrt{2\pi}}\log\frac{1}{\epsilon} -  \Phi^{\cP_\epsilon}) e^{-2\log \frac{1}{\epsilon} + \sqrt{8\pi}\Phi^{\cP_\epsilon}}.
\end{equation}
However, this has been established rigorously only for the massless Gaussian free field thanks to its conformal invariance, see \cite{MR4082182}.
Even though the exact characterisation of $\ZDM^\cP$ does not come out of the proof of Theorem \ref{thm:convergence-to-gumbel}, it does show that $\ZDM^\cP$ is obtained as the limit of prototypical derivative martingales defined similarly to \eqref{eq:derivative-martingale}.

Nonetheless, the coupling in Theorem \ref{thm:coupling-pphi-to-gff-eps} and Corollary \ref{cor:pphi-coupling-continuum} immediately gives a construction of the limit of the measure associated to \eqref{eq:derivative-martingale} as $\epsilon \to 0$ through
\begin{equation}
\lim_{\epsilon \to 0} \ZDM^{\cP_\epsilon}(dx) = e^{\sqrt{8\pi} \Phi_0^{\Delta} (x)}\ZDM^\GFF(x) 
\end{equation}
where $\ZDM^\GFF(dx)$ is the Gaussian multiplicative chaos associated with $\Phi_0^\GFF$.
We expect that this relation allows to establish further properties for the multiplicative chaos of the $\cP(\phi)_2$ field.

Finally, we believe that the coupling in Theorem \ref{thm:coupling-pphi-to-gff-eps} can also be used to prove finer results on the extreme values of the $\cP(\phi)_2$ field,
in particular on the locations and heights of local maxima.
The main reference for the local and full extremal process for the Gaussian free field is the work by Biskup and Louidor, see \cite{MR3509015}, \cite{MR3787554} and \cite{MR4082182}.
Using the coupling in Theorem \ref{thm:coupling-pphi-to-gff-eps} and the generalisation of the key results in these references to the non-Gaussian setting analogously to \cite{Hofstetter2021Extremal}, 
it seems plausible that the local extremal process of the $\cP(\phi)_2$ field converges to a Poisson point process on $\Omega\times \R$ with random intensity measure $\ZDM(dx)\otimes e^{-\sqrt{8\pi} h}dh$.
With additional work, it should also be possible to prove convergence of the full extremal process thanks to the continuity of the difference field.

\subsection{Notation}

For a covariance $c^\epsilon$ on $X_\epsilon$, we denote by $\EE_{c^\epsilon}$ the expectation with respect to the centred Gaussian measure with covariance $c^\epsilon$. 
Moreover, we use the standard Landau big-$O$ and little-$o$ notation and write $f \lesssim g$ to denote that $f\leq O(g)$, and $f \simeq g$ if $f \lesssim g$ and $f\lesssim g$.

When an estimate holds true up to a constant factor depending on a parameter $\alpha$ say, i.e.\ if $f \leq C g$ where the constant $C=C(\alpha)$ depends only on $\alpha$, then we write $f\lesssim_{\alpha} g$.

For fields $f \in X_\epsilon$, we either write $f(x)$ or $f_x$ for the evaluation of $f$ at $x\in \Omega_\epsilon$.
When the field already comes with an index, say $t$, then we write $f_t(x)$ instead. We also use the word function and field interchangeably for elements in $X_\epsilon$.

To simplify notation in what follows we will write
\begin{equation}
\int_{\Omega_\epsilon} f(x) dx \equiv
\epsilon^2 \sum_{x \in \Omega_\epsilon} f(x)
\end{equation}
for the discrete integral over $\Omega_\epsilon$.
For two functions $f,g \in X_\epsilon^\C \equiv  \{\varphi\colon \Omega_\epsilon \to \C\}$ we define the inner product
\begin{equation} \label{e:innerprod}
  \avg{f,g} \equiv \avg{f,g}_{\Omega_\epsilon} \equiv \epsilon^2 \sum_{x \in \Omega_\epsilon} f(x) \overline{ g(x)}.
\end{equation}
For $p\in [1,\infty)$  we also define the $L^p$-norm for $f\in X_\epsilon^\C$ by 
\begin{equation}
\|f\|_{L^p(\Omega_\epsilon)}^p \equiv \| f \|_{L^p}^p = \int_{\Omega_\epsilon} |f(x)|^p dx
\end{equation}
and for $p=\infty$ we define 
\begin{equation}
\|f\|_{L^\infty(\Omega_\epsilon)} \equiv \|f\|_{L^\infty} \equiv \|f\|_{\infty} = \max_{x\in \Omega_\epsilon}|f(x)|.
\end{equation}
We write $L^p(\Omega_\epsilon) = (X_\epsilon, \|\cdot\|_{L^p})$, $1\leq p\leq\infty$ for the discrete $L^p$ space. 
Finally, for $F\colon X_\epsilon \to \R$, $\varphi \mapsto F(\varphi)$ define $\nabla F$ as the unique function $X_\epsilon\to X_\epsilon$ that satisfies
\begin{equation}
\avg{\nabla F(\varphi), g} = D_\varphi F(g) 
\end{equation}
for all $\varphi,g \in X_\epsilon$, where $D_\varphi F  \colon X_\epsilon \to \R$ is the Fr\'echet derivative of $F$ at $\varphi$, i.e.\ the unique bounded linear map satisfying
\begin{equs}
F(\varphi + g) 
= 
F(\varphi) + D_\varphi F(g) + o(g)	
\end{equs}
for all $g \in X_\epsilon$ and where $o(g) = o(\|g\|)$ for some norm on $X_\epsilon$.
Note that our convention for $\nabla F$ differs from the usual gradient by a factor $\epsilon^{-2}$ due to the normalised inner product.

\section{Discrete Besov spaces and the regularity of Wick powers}
\label{sec:discrete-calculus}

In our analysis we require regularity information on various types of random fields on $X_\epsilon$ that are uniform in the regularisation parameter. 
Particularly important examples of such random fields are given by discrete analogues of Wick powers for the GFF. 
Besov spaces allow for a convenient description of the regularity for the latter objects. 
As such, in this section, we first recall basic notions from discrete Fourier analysis and define discrete analogues of Besov spaces. All inequalities in this section will be uniform in $\epsilon$ unless otherwise stated.

\subsection{Discrete Fourier series and trigonometric embedding}
\label{ssec:fourier}

Any function in $X_\epsilon$ admits a Fourier representation thanks to the periodic boundary conditions. Let $\Omega_\epsilon^*=\{ (k_1,k_2) \in 2\pi \Z^2 \colon  -\pi/\epsilon < k_i \leq \pi/\epsilon \}$ and $\Omega^*= 2\pi \Z^2$ be the Fourier dual spaces of $\Omega_\epsilon$ and $\Omega$. Then, for $f\in X_\epsilon$ and $x\in \Omega_\epsilon$,
\begin{equs}
\label{eq:fourier-series-f}
f(x)
=
\sum_{k \in \Omega_\epsilon^*} \hat f(k) e^{ik\cdot x},
\end{equs}
where $\hat f(k)\in \C$ is the $k$-th Fourier coefficient of $f$ given by
\begin{equs}
\label{eq:fourier-coefficient-f}
\hat f (k)
=
\avg{f,e^{ik\cdot} }_{\Omega_\epsilon}
= \epsilon^2 \sum_{x\in \Omega_\epsilon} f(x) e^{-ik\cdot x}.
\end{equs}
Since $f$ is real-valued, we have that $\hat f(k) = \overline{\hat f(-k)}$.
For a given $f\in X_\epsilon$, the map $\cF_\epsilon(f)\colon \Omega_\epsilon^* \to \C$, $k\mapsto \hat f(k)$ is called the (discrete) Fourier transform of $f$.
Similarly, for any function $\hat f \in X_\epsilon^* := \{ g \colon \Omega_\epsilon^* \to \C \}$, we define the inverse Fourier transform $\cF_\epsilon^{-1}$ by
\begin{equs}
\cF_\epsilon^{-1} (\hat f) (x)
=
\sum_{k \in \Omega_\epsilon^*} \hat f(k) e^{ik\cdot x}.
\end{equs}
As for the usual Fourier transform on $\Omega$ and its inverse, which we define analogously to \eqref{eq:fourier-series-f} and \eqref{eq:fourier-coefficient-f} by replacing $\Omega_\epsilon$ by $\Omega$,
it can be easily seen that  $\cF^{-1} \circ \cF = \id_{X_\epsilon^\C}$.

For a translation invariant operator operator $q \colon X_\epsilon \to X_\epsilon$, we denote by $\hat q(k)$, $k\in \Omega_\epsilon^*$ its Fourier multipliers, defined by 
\begin{equs}
\widehat{q (f)}
=
\hat q(k) \hat f(k).
\end{equs}
For instance, it can be shown that the negative lattice Laplacian $-\Delta^\epsilon$ on $\Omega_\epsilon$ has Fourier multipliers
\begin{equation}
\label{eq:fourier-multipliers-discrete-laplacian}
-\hat \Delta^\epsilon(k) = \epsilon^{-2} \sum_{i=1}^2 (2-2\cos(\epsilon k_i)), \quad k \in \Omega_\epsilon^*.
\end{equation}
As $\epsilon\to 0$ these converge to the Fourier multipliers of the continuum Laplacian given by
\begin{equation}
-\hat \Delta (k) = |k|^2, \qquad k \in \Omega^*
\end{equation}
and this convergence is quantified by
\begin{equation}
\label{eq:fourier-multipliers-laplacians-difference}
0\leq -\hat \Delta(k) - (-\hat \Delta^\epsilon(k)) \leq |k|^2 h(\epsilon k), \qquad k \in \Omega_\epsilon^*,
\end{equation}
where $h(x) = \max_{i=1,2} (1 -x_i^{-2}(2-2\cos(x_i)))$ satisfies $h(x) \in [0,1-c]$ with $c = 4/\pi^2$ for $|x|\leq \pi$
and $h(x) =O(|x|^2)$ as $x\to 0$.

We extend functions on $\Omega_\epsilon$ onto $\Omega$ by using the standard trigonometric extension,
which is also an isometric embedding $I_\epsilon \colon L^2(\Omega_\epsilon) \to L^2(\Omega)$,
i.e.\ if $f\in X_\epsilon$ has Fourier series \eqref{eq:fourier-series-f}, then the extension $I_\epsilon f$
of $f$ is the unique function $\Omega\to \R$ whose Fourier coefficients agree with those of $f$ for $k\in \Omega_\epsilon^*$ and vanish for $k\in \Omega^*\setminus \Omega_\epsilon^*$.
Note that $I_\epsilon f$ coincides with $f$ on $\Omega_\epsilon$.

Conversely, we can restrict a smooth function $f\colon \Omega \to \R$ to $\Omega_\epsilon$ by restricting its Fourier series
\begin{equs}
f(x)
=
\sum_{k\in \Omega^*} \hat f(k)e^{ik\cdot x}, \qquad \hat f(k) = \int_{\Omega} f(x)e^{-ik\cdot x} dx
\end{equs}
to Fourier coefficients $k\in \Omega_\epsilon^*$, i.e.\ we write
\begin{equs}
\Pi_\epsilon f
=
\sum_{k\in \Omega_\epsilon^*}\hat f(k)e^{ik\cdot x}.
\end{equs}

\subsection{Discrete Sobolev spaces}
\label{ssec:sobolev}

The notion of Fourier series can be used to define discrete Sobolev norms in complete analogy to the continuum case.
For $\Phi \in X_\epsilon$ and $\alpha \in \R$ we define
\begin{equation}
\label{eq:sobolev-norm-def}
\| \Phi \|_{H^\alpha(\Omega_\epsilon)}^2 \equiv \| \Phi \|_{H^\alpha}^2  = \sum_{k\in \Omega_\epsilon^*} (1+|k|^2)^{\alpha} |\hat \Phi(k)|^2,
\end{equation}
where $\hat \Phi(k)$, $k\in \Omega_\epsilon^*$ denote the Fourier coefficients of $\Phi$ as defined in \eqref{eq:fourier-series-f}. Moreover, we denote by $H^\alpha(\Omega_\epsilon) = (X_\epsilon, \|\cdot \|_{H^\alpha(\Omega_\epsilon)})$ the discrete Sobolev space of regularity $\alpha$. Thus, when using the isometric embedding the discrete and the continuum Sobolev norm $\|\cdot\|_{H^\alpha(\Omega)}$,
defined as in \eqref{eq:sobolev-norm-def} except the sum being now over $k\in \Omega^*$,
coincide, i.e.\ for $\Phi \in X_\epsilon$
\begin{equation}
\|\Phi\|_{H^\alpha(\Omega_\epsilon)} = \| I_\epsilon \Phi \|_{H^\alpha(\Omega)},
\end{equation}
where now $H^\alpha(\Omega)$ denotes the Sobolev space of regularity $\alpha\in \R$.
Recall that we have the following standard embedding theorem.
Here, $\| \cdot \|_{L^p(\Omega)}$ and $\|\cdot \|_{H^s(\Omega)}$ denote the continuum norms for functions $\Omega \to \C$.

\begin{proposition}[Sobolev embeddings]
Let $s \geq 1$. Then
\begin{equs}
\|f\|_{L^p(\Omega)}
\lesssim_s
\|f\|_{H^s(\Omega)},
\end{equs}
where
\begin{equs}
\begin{cases}
2\leq p < \infty,    					&s = 1, \nnb
2\leq p \leq \infty,  				&  s>1.
\end{cases}
\end{equs}

In particular, we have for every $ \alpha > 0$ and $f\colon \Omega_\epsilon \to \C$,
\begin{equs}
 \|f\|_{L^\infty(\Omega_\epsilon)} \leq  
\|I_\epsilon f\|_{L^\infty(\Omega)}
\lesssim_\alpha 
\|I_\epsilon f \|_{H^{1+\alpha}(\Omega)} = \| f \|_{H^{1+\alpha}(\Omega_\epsilon)}  .
\end{equs}
\end{proposition}

Finally, we also have the following H\"older embedding. Let $\alpha \in (0,1)$. For a function $f\in C^\infty(\Omega,\R)$, we define its $\alpha$-H\"older norm by
\begin{equation}
\| f \|_{C^\alpha(\Omega)}  \equiv \| f \|_{C^\alpha} = |f|_{C^\alpha}  + \|f\|_{L^\infty}, \quad  |f|_{C^\alpha(\Omega)}= \sup_{x,y \in \Omega, x\neq y} \frac{|f(x)-f(y)|}{|x-y|^\alpha},
\end{equation}
where $|f|_{C^\alpha}$ denotes the H\"older seminorm.
The $\alpha$-H\"older space, denoted $C^\alpha(\Omega)$, is then defined as the completion of $C^\infty(\Omega,\R)$ with respect to $\| \cdot \|_{C^\alpha(\Omega)}$.
For discrete functions $f\in X_\epsilon$, we define 
\begin{equs}
\| f \|_{C^\alpha(\Omega_\epsilon) }
=
\| I_\epsilon f \|_{C^\alpha(\Omega)}
\end{equs}
and $C^\alpha(\Omega_\epsilon) = (X_\epsilon, \| \cdot \|_{C^\alpha(\Omega_\epsilon)})$.
Then, the following standard embedding from Sobolev spaces into H\"older spaces can be transferred to functions in $X_\epsilon$.
\begin{proposition}
\label{prop:hoelder-sobolev}
Let $\alpha \in (0,1)$ and $s-1 \geq \alpha$. Then 
\begin{equs}
\| f\|_{C^\alpha(\Omega)}
\lesssim_{\alpha, s}
\|f\|_{H^s(\Omega)}.
\end{equs}
\end{proposition}

\subsection{Discrete Besov spaces}

In this section we introduce the important class of Besov spaces and recall their most important properties. 
We state the results below for dimension $d=2$ though analogous statements hold for general $d$.
Let $A \coloneqq B_{4/3} \setminus B_{3/8}$ be the annulus of inner radius $r_1 = 3/8$ and outer radius $r_2=4/3$. 
Here, we denote by $B_r = \{ |x| \leq r \} \subset \R^2$ the centred ball of radius $r\geq 0$ in $\R^2$.
Let $\chi, \tilde \chi \in C^\infty_c(\R^2,[0,1])$,  such that 
\begin{equation}
\supp \tilde \chi \subseteq B_{4/3}, \qquad \supp \chi \subseteq A
\end{equation} 
and 
\begin{equation}
\tilde \chi(x) + \sum_{j=0}^\infty \chi(x/2^j) = 1, \qquad x\in \R^2
\end{equation}
and write
\begin{equation}
\label{eq:partition-unity}
\chi_{-1} = \tilde \chi,  \quad \chi_{j} = \chi(\cdot/2^j), \,\,\, j \geq 0.
\end{equation}
For $\epsilon>0$ define $j_\epsilon = \max \{ j \geq -1 \colon  \supp \chi_j  \subset (-\pi/\epsilon, \pi/\epsilon]^2 \}$.
Note that for $j \geq j_\epsilon$, $\supp \chi_j$ may intersect $\partial \{ [-\pi/\epsilon, \pi/\epsilon]^2 \}$.
To avoid ambiguities with the periodisation of $\chi_j$ onto $\Omega_\epsilon^*$, we modify our dyadic partition of unity in \eqref{eq:partition-unity} as follows:
for $j \in \{ -1, \ldots, j_\epsilon\}$ let $\chi_j^\epsilon \in C^\infty_c(\R^2, [0,1])$ be such that for $k \in \Omega_\epsilon^*$ we have
\begin{equation}
\chi_j^\epsilon (k)
=
\begin{cases}
\chi_j(k), \,\qquad & j < j_\epsilon,
\\ 
1- \sum_{j < j_{\epsilon}} \chi_j^\epsilon(k), \,\qquad & j = j_\epsilon,
\\
0, \,\qquad & j >j_\epsilon.
\end{cases}	
\end{equation}
Then we define for $j\geq -1$ the $j$-th Fourier projector $\Delta_j$ by
\begin{equation}
\Delta_j f = \cF^{-1} (\chi_j^\epsilon \hat f),
\end{equation}
where $\hat f \colon \Omega_\epsilon^* \to \C, \, k\mapsto \hat f(k)$ is the Fourier transform of $f$.
We can use this partition of unity to decompose a given function $f\in X_\epsilon$ into a sum of functions with almost disjoint support in Fourier space and define for any $\alpha \in \R$ and $p,q \in [1,\infty]$
\begin{equation}
\label{eq:definition-besov-norm}
\| f \|_{\besov{p}{q}{\alpha}} \coloneqq \Big[ \sum_{j=-1}^\infty \big( 2^{\alpha k} \|\Delta_j f \|_{L^p}  \big)^q  \Big]^{1/q}.
\end{equation}
One can prove that this is indeed a norm on $X_\epsilon$
and that different choice of $\tilde \chi, \chi$ yield equivalent norms uniformly in $\epsilon >0$. 

We denote by $\besov{p}{q}{\alpha}(\Omega_\epsilon) = (X_\epsilon, \| \cdot \|_{\besov{p}{q}{\alpha}})$ the discrete Besov space with parameters $p$,$ q$ and $\alpha$.
Note that $H^\alpha(\Omega_\epsilon) = \besov{2}{2}{\alpha}(\Omega_\epsilon)$ 
which holds in the sense that the norms are equivalent uniformly in the lattice spacing. 
Moreover, for any $\alpha \in \R$, we write $\cC^\alpha(\Omega_\epsilon)\coloneqq \besov{\infty}{\infty}{\alpha}(\Omega_\epsilon)$. 

We now state useful properties of these spaces. In all estimates below we use $\lesssim$ to denote less than or equal to up to a constant that may depend on the parameters of the Besov space but does not depend on $\epsilon$.
For the proofs of these result, we refer to
\cite[Remark 8, 10 and 11]{MR3693966}	   
\cite[Lemma A.2]{MR4252872},  
\cite[Theorem 2.13]{BarashkovDeVecchi2021Elliptic},	   
\cite[Lemma A.3]{MR4252872},   
and \cite[Theorem 2.14]{BarashkovDeVecchi2021Elliptic}   
respectively.

\begin{proposition}[Immediate embeddings]
\label{prop:immediate-embeddings}
Let $\alpha_1, \alpha_2 \in \R$ and $p_1,p_2,q_1, q_2 \in [1,\infty]$. Then we have
\begin{align}
\label{eq:besov-immediate-1}
\|u \|_{\besov{p_1}{q_1}{\alpha_1}} &\lesssim \| u \|_{\besov{p_2}{q_2}{\alpha_2}}  && \text{~for~} \alpha_1 \leq \alpha_2 \text{~and~} p_1 \leq p_2 \text{~and~} q_1 \geq q_2, \\
\label{eq:besov-immediate-2}
\|u \|_{\besov{p_1}{q_1}{\alpha_1}}  &\lesssim \| u \|_{B_{p_1,\infty}^{\alpha_2}} &&\text{~for~} \alpha_1< \alpha_2, \\
\label{eq:besov-immediate-3}
\|u\|_{\besov{p_1}{\infty}{0}} &\lesssim \| u \|_{L^{p_1}} \lesssim \| u \|_{\besov{p_1}{1}{0}}.
\end{align}
\end{proposition}

\begin{proposition}[Duality]
\label{prop:duality}
Let $\alpha_1,\alpha_2 \in \R$ such that $\alpha_{1}+\alpha_{2}=0$ and let $p_1,p_2,q_1,q_2 \in [1,\infty]$ such that $\frac{1}{p_{1}}+\frac{1}{p_{2}}=\frac 1{q_1} + \frac 1{q_2} = 1$.
Then
  \begin{equs}
  \label{eq:besov-duality}
   \Big| \int_{\Omega_{\epsilon}} f g dx \Big|
    \leq 
    \|f\|_{\besov{p_1}{q_1}{\alpha_1} }\|g\|_{\besov{p_2}{q_2}{\alpha_2}}.
  \end{equs}
\end{proposition}

\begin{proposition}[Multiplication inequality]
\label{prop:besov-products}
Let $p,p_{1},p_{2},q,q_1,q_2\in[1,\infty]$ be such that $\frac{1}{p}=\frac{1}{p_{1}}+\frac{1}{p_{2}}$, $\frac 1q = \frac 1{q_1} + \frac 1{q_2}$, and let $\alpha_{1},\alpha_{2}\in \mathbb{R}$ be such that $\alpha_{1}+\alpha_{2}>0$.  Denote $\alpha=\min(\alpha_{1},\alpha_{2})$. Then, for any $\epsilon > 0$ and for all $f,g \in X_\epsilon$, 
\begin{equs} 
\label{eq:besov-multiplication-inequality}
\| fg \|_{\besov{p}{q}{\alpha}}
\lesssim
\|f\|_{\besov{p_1}{q_1}{\alpha_1}} \|g\|_{\besov{p_2}{q_2}{\alpha_2}}.
\end{equs}
In particular, for $\alpha > 0$ we obtain the following iterated multiplication inequality: for every $k\in \N$ such that $k \geq 1$,
\begin{equation}
\|f^k\|_{\besov{p}{q}{\alpha}} \lesssim \| f \|_{\besov{kp}{kq}{\alpha}}^k.
\end{equation}
\end{proposition}

\begin{proposition}[Interpolation]
\label{prop:besov-interpolation}
Let $\theta \in [0,1]$, $p,p_{1},p_{2},q,q_1,q_2\in [1,\infty]$ satisfying $\frac{1}{p}=\frac{\theta}{p_{1}}+\frac{1-\theta}{p_{2}}$, $\frac 1q = \frac\theta{q_1} + \frac{1-\theta}{q_2}$,
and $\alpha,\alpha_{1},\alpha_{2}\in \R$ satisfying $\alpha=\theta \alpha_{1}+(1-\theta)\alpha_{2}$.
Then, for any $\epsilon > 0$ and $f \in X_\epsilon$,
\begin{equs}
\label{eq:besov-interpolation}
\|f\|_{\besov{p}{q}{\alpha}}
\lesssim 
\|f\|^{\theta}_{\besov{p_1}{q_1}{\alpha_1}} \|f\|_{\besov{p_2}{q_2}{\alpha_2}}^{1-\theta}.
\end{equs}
\end{proposition}

\begin{proposition}[Besov embedding]
\label{prop:besov-embedding}
Let $\alpha_{1},\alpha_{2}\in \RR$ and $p_{1},p_{2},q\in [1,\infty]$. Assume that $\alpha_{2} \leq \alpha_{1}-2\big( \frac{1}{p_{1}}-\frac{1}{p_{2}}\big)$. Then, for any $\epsilon > 0$ and $f \in X_\epsilon$,
\begin{equs}
\label{eq:besov-embedding}
\|f\|_{\besov{p_2}{q}{\alpha_2}} 
\lesssim 
\|f\|_{\besov{p_1}{q}{\alpha_1}} .
\end{equs}
\end{proposition}

\subsection{Regularity estimates on discrete Wick powers}

The relevance of Besov norms in our context comes from the fact that in dimension $d=2$ the Wick powers of log-correlated fields are distributions in $\besov{p}{p}{-\kappa}$ for any $\kappa>0$.  

\begin{lemma}
\label{lem:wick-powers-besov-expectation}
Let $Y_\infty^\epsilon$ be the discrete Gaussian free field on $\Omega_\epsilon$. Then, for any $\kappa>0$ and $ p \in [1,\infty)$,
\begin{equation}
\label{eq:wick-power-besov-expectation}
\sup_{\epsilon > 0} \E \big [\| \wick{(Y_\infty^\epsilon)^n} \|_{\besov{p}{p}{-\kappa}}^p\big] <\infty.
\end{equation}
Consequently, by Besov embedding \eqref{eq:besov-embedding}, for any $n \in \N$, $\kappa>0$, and $r > 0$,
\begin{equs}
\sup_{\epsilon>0} \E\big[\|\wick{(Y_\infty^\epsilon)^n} \|^{r}_{\cC^{-\kappa}}\big]
<
\infty.
\end{equs}
\end{lemma}

\begin{proof}
Throughout this proof we drop $\epsilon>0$ from the notation.
By the definition of the Besov norms, we have
\begin{equs} \label{eq:reg-wick-1} 
\begin{split}
\E \big[\| \wick{Y_\infty^{n}} \|_{\besov{p}{p}{-\kappa}}^{p} \big]
&=
\sum_{j=-1}^{\infty} 2^{-j \kappa p} \E \big[\|\Delta_{j}\wick{Y_\infty^{n}} \|^{p}_{L^{p}}\big]
\\
&\lesssim      
\sum_{j=-1}^{\infty}2^{-j \kappa p} \Big(\E\big[\big|\big(\Delta_{j}\wick{Y_\infty^{n}}\big)(0)\big|^{2}\big]\Big)^{p/2} 
\\
&= 
\sum_{j=-1}^{\infty}2^{-j \kappa p}\Big(\E \Big[\int_{\Omega_\epsilon\times\Omega_\epsilon} K_j(x)K_j(y)\wick{Y_\infty^{n}(x)}\wick{Y_\infty^{n}}(y) dx dy \Big] \Big)^{p/2}
\\
&=
\sum_{j=-1}^{\infty}2^{-j \kappa p} \Big(\int_{\Omega_\epsilon\times\Omega_\epsilon} K_j(x)K_j(y) \big( c(x,y)\big)^{n} dx dy \Big)^{p/2},
\end{split}
\end{equs}
where in the second line, we use stationarity together with a standard Wiener chaos estimate (which is a consequence of hypercontractivity),
in the third line, we denote the (real-space) kernel of $\Delta_j$ by $K_j=K_j^\epsilon$, and in the final line we use Wick's theorem. 
  
Recall that we have $\|K_j\|_{L^{1}(\Omega_\epsilon)}\lesssim 1$ and $\|K_j\|_{L^{\infty}(\Omega_\epsilon)}\lesssim 2^{2j}$.
This implies, by interpolation, that $\|K_j\|_{L^{q_{1}}(\Omega_\epsilon)}\lesssim 2^{j \kappa/2}$ if $q_{1}>1$ is chosen sufficiently close to $1$.
Observe that if $q_{2}<\infty$ is the H\"older conjugate of $q_{1}$, then $\sup_{\epsilon>0}\|c^n\|_{L^{q_{2}}(dxdy)}<\infty$. Thus,
  \begin{equs}\label{eq:reg-wick-2}
    \Big| \int K_{j}(x)K_{j}(y)c^{n}(x,y) dxdy \Big|
     &\leq 
     \|K_{j}(x)K_{j}(y)\|_{L^{q_{1}}(dxdy)}\|c^{n}\|_{L^{q_{2}}(dxdy)} \\
     &\leq 
     \|K_{j}\|^{2}_{L^{q_{1}}}\|c^{n}\|_{L^{q_{2}}(dxdy)}\lesssim 2^{j\kappa}.
  \end{equs}
  Plugging \eqref{eq:reg-wick-2} into the sum \eqref{eq:reg-wick-1}, we see that it converges and is bounded uniformly in $\epsilon$. This completes the proof.
\end{proof}

In the sequel we consider the Wick powers of the sum of the Gaussian free field and a regular field. 
Note that for $n \in \N$ and $\varphi \in X_\epsilon$ we have
\begin{equation}
\label{eq:y-infty-wick-binomial}
  \wick{(Y_\infty^\epsilon+\varphi)^n}
  =
  \sum_{k=0}^{n} {n \choose k} \wick{(Y_\infty^\epsilon)^{n-k}} \varphi^{k}.
\end{equation}
In practice, we apply this formula in the case where $\varphi$ admits a bound uniform in $\epsilon>0$  on the second moment of some positive regularity Sobolev norm.
In particular, it is mainly useful when $\varphi$ admits estimates in a positive regularity norm (rather than a negative regularity distribution norm) that are uniform in $\epsilon>0$. 

\begin{lemma} 
\label{lem:wick-varphi-estimate}
Let $Y_\infty^\epsilon$ be the discrete Gaussian free field on $\Omega_\epsilon$.
Then for any $n \in \N$ and $\kappa > 0$, and $\bar \kappa >0$ small enough
\begin{equation}
\label{eq:wick-gff-varphi-estimate}
\|\wick{(Y_\infty^\epsilon+\varphi)^{n}}\|_{\cC^{-\kappa}}
 \lesssim 
1 + \sum_{k=0}^{n-1}\| \wick{(Y_\infty^\epsilon)^{n-k}}\|_{\cC^{-\bar \kappa}}^{n/(n-k)} + \|\varphi\|_{H^{1}}^n
\end{equation}
uniformly in $\epsilon > 0$. 
\end{lemma}

\begin{proof}
We estimate the terms on the right-hand side of \eqref{eq:y-infty-wick-binomial} separately.
For $k\in \{1,\ldots, n-1\}$ and $p >2/\kappa$, we have for any $\delta>0$
by Besov embedding \eqref{eq:besov-embedding} and the multiplication inequality \eqref{eq:besov-multiplication-inequality} 
\begin{align}
\label{eq:beov-product-estimate-1}
\| \wick{Y_\infty^{n-k}} \varphi^{k} \|_{\cC^{-\kappa}} 
\lesssim_{\kappa}
\| \wick{Y_\infty^{n-k} } \varphi^{k} \|_{\besov{p}{\infty}{-\kappa+ 2/p}}  
\lesssim_{\kappa, \delta}
\| \wick{Y_\infty^{n-k}} \|_{\besov{\infty}{\infty}{-\kappa + 2/p}} \|\varphi^{k} \|_{\besov{p}{\infty}{\kappa -2/p + \delta}}.
\end{align}
Now, by iterating the multiplication inequality \eqref{eq:besov-multiplication-inequality} and Besov embedding \eqref{eq:besov-embedding} we have for $2/\kappa < p \leq 2(1+ 1/k) / (\kappa + \delta)$
\begin{equation}
\| \varphi^k \|_{\besov{p}{\infty}{\kappa -2/p + \delta}} 
\lesssim_{k,p}
\| \varphi \|_{\besov{kp}{\infty}{\kappa-2/p + \delta}}^k 
\lesssim_{\delta, \kappa, k}
\| \varphi \|_{H^1}^k.
\end{equation}
Thus, we can further estimate \eqref{eq:beov-product-estimate-1} by
\begin{equation}
\| \wick{Y_\infty^{n-k}} \varphi^{k} \|_{\cC^{-\kappa}} 
\lesssim_{\kappa,\delta,k}
\| \wick{Y_\infty^{n-k}} \|_{\cC^{-\kappa + 2/p}} \| \varphi \|_{H^1}^k 
\lesssim
 \| \wick{Y_\infty^{n-k}} \|_{\cC^{-\bar \kappa}}^{n/(n-k)} +\| \varphi \|_{H^1}^n,
\end{equation}
where we used Young's inequality in the last step and set $\bar \kappa = \kappa-2/p$.
Summing over $k\in \{1,\ldots ,n-1\}$ and observing that
\begin{equation}
\| \varphi^n\|_{\cC^{-\kappa}} \leq \| \varphi\|_{H^1}^n,
\end{equation}
the estimate \eqref{eq:wick-gff-varphi-estimate} follows.
\end{proof}

\section{Stochastic representations of $\cP(\phi)_2$}

In this section, we present two stochastic representations of measures $\nu^{\cP_\epsilon}$: the Polchinski re\-nor\-ma\-li\-sa\-tion group approach and a stochastic control representation via the Bou\'e-Dupuis variational formula.
The former underlies the SDE \eqref{eq:pphi-sde-construction-intro} that we use to construct the process $\Phi^{\cP_\epsilon}$.
We show that there is an exact correspondence between these two approaches. In particular, the minimiser of the variational problem is explicitly related to the difference field of $\Phi^{\Delta_\epsilon}$ of the Polchinski dynamics.
For technical reasons, we introduce a potential cut-off to guarantee the well-posedness of the SDE and to ensure existence of minimisers for the variational problem.

\subsection{Pauli-Villars decomposition of the covariance}
\label{ssec:pauli-villars}

Let $(c_t^\epsilon)_{t\in[0,\infty]}$ be a continuously differentiable decomposition of the covariance of the GFF, i.e.\
\begin{equs}
c_t^\epsilon
=
\int_0^t \dot c_s^\epsilon ds,  \qquad c_\infty^\epsilon = (-\Delta^\epsilon + m^2)^{-1},
\end{equs}
where, for every $s > 0$, $\dot c_s^\epsilon$ is a positive-semidefinite operator acting on $X_\epsilon$. 
The choice of $\dot c_t^\epsilon$ is restricted by the condition on $c_\infty^\epsilon$. In this work we use the Pauli-Villars regularisation 
\begin{equs}
\label{eq:pauli-villars}
 c_t^\epsilon
 =
 ( -\Delta^\epsilon + m^2 + 1/t)^{-1}, \qquad \dot c_t^\epsilon = \frac{d}{dt}c_t^\epsilon.
\end{equs}

\begin{remark}\label{rem:pauli-villars}
The choice of the Pauli-Villars regularisation is twofold: first, it satisfies the required regularity estimates in Section \ref{sec:sobolev-norms-int-drift}. Second, it is technically convenient with regards to convergence of the maxima in Section \ref{sec:maximum}. In particular, it allows to avoid the introduction of an additional field in the approximation of the small scale field in Section \ref{sec:approx-small-scales}. This, however, is not necessary and we expect that some other choices which are more natural from an analytic perspective (i.e.\ satisfying the estimates of Section \ref{sec:sobolev-norms-int-drift}), such as a heat kernel decomposition, should also work.
\end{remark}

We view $c_t^\epsilon$ and $\dot c_t^\epsilon$ acting on $X_\epsilon$ as Fourier multipliers.
For instance, we have for $f\in X_\epsilon$
\begin{equation}
(\dot c_t^\epsilon f )(x) = \sum_{k\in \Omega_\epsilon^*}  \widehat{\dot c_t^\epsilon} (k)\hat f(k) e^{ikx},
\end{equation}
where the Fourier multipliers of $\dot c_t^\epsilon$ are given by
\begin{equation}
\widehat{ \dot c_t^\epsilon } =
\frac{1}{\big(t(-\hat \Delta^\epsilon(k)+m^2) + 1\big)^2},
\end{equation}
where $-\hat \Delta^\epsilon(k)$ is as in \eqref{eq:fourier-multipliers-discrete-laplacian}.
We also record the Fourier multipliers for $q_t^\epsilon$ which is the unique positive semi-definite operator on $X_\epsilon$ such that $\dot c_t^\epsilon =  q_t^\epsilon * q_t^\epsilon$,
where $*$ is the discrete convolution on $\Omega_\epsilon$.  From this defining relation, we immediately deduce that
\begin{equation}
\label{eq:fourier-multipliers-q}
\hat q_t^\epsilon (k) = \frac{1}{t(-\hat \Delta^\epsilon(k)+m^2) + 1}.
\end{equation}

\begin{remark}
For $x,y \in \Omega_\epsilon$, set $c_t^\epsilon(x,y) = (-\Delta^\epsilon + m^2 + 1/t)^{-1}(x,y)$, where $(x,y)\mapsto (-\Delta^\epsilon + m^2 + 1/t)^{-1}(x,y)$ is the Green's function of $(-\Delta^\epsilon + m^2 +1/t)$.
Note that $c_t^\epsilon(x,y) = C_{t/\epsilon^2}(x/\epsilon, y/\epsilon)$, where 
\begin{equs}
C_t^\epsilon
= 
(-\Delta + \epsilon^2 m^2 + 1/t )^{-1}
\end{equs}
is the Green's function for the massive unit lattice Laplacian.
It is easy to see that $c_t^\epsilon(x,y) = c_t^\epsilon(0,x-y)$, i.e.\ $c_t^\epsilon$ is stationary. We simply write $c_t^\epsilon(x-y)= c_t^\epsilon(x,y)$.
\end{remark}

We use the Pauli-Villars decomposition of $c_\infty^\epsilon$ to construct a process associated to the discrete Gaussian free field on $\Omega_\epsilon$ as follows.
Let $W$ be a cylindrical Brownian motion in $L^2(\Omega)$ defined on a probability space $(\mathcal{O}, \cF, \P)$ and denote by $(\cF_t)_{t \geq 0}$ and $(\cF^t)_{t \geq 0}$ the forward and backward filtrations generated  by the past $\{W_s-W_0\colon s \leq t\}$ and the future $\{W_s-W_t\colon s\geq t\}$.
We assume that $\cF$ is $\P$-complete and the filtrations are augmented by $\P$-null sets.
Moreover, we write expectation with respect to $\P$ as $\E$.
The cylindric Brownian motion $W$ can be almost surely represented as a random Fourier series via the so-called Karhunen-Lo\`eve expansion.
More precisely, for $k \in 2\pi \Z^2$ and $t \geq 0$, let $\hat W_t (k) = \int_\Omega W_t(x) e^{-i k \cdot x}dx$.
Then, almost surely $\{ \hat W(k) \colon k \in 2\pi \Z^2 \}$ is a set of complex standard Brownian motions, independent up to the constraint $\hat W(k)=\overline{\hat W(-k)}$
and $\hat W(0)$ is a real standard Brownian motion and we can write
\begin{equation}
W_t(x)
=
\sum_{k \in \Omega^*} e^{i k \cdot x} \hat W_t(k),
\qquad
x \in \Omega,
\end{equation}
where the sum converges uniformly on compact sets in $ C([0,\infty),H^{-1-\kappa})$ for any $\kappa>0$, i.e.\ 
for any $T\geq 0$, we have as $r\to \infty$
\begin{equation}
\E \Big [ \sup_{t\leq T} \Big \|  \sum_{|k|>r} e^{ik\cdot} \hat W_t(k)  \Big\|_{H^{-1-\kappa}}^2  \Big] \to 0 .
\end{equation}
To obtain a Brownian motion in $\Omega_\epsilon$ we restrict the formal Fourier series of $W$ to $k\in \Omega_\epsilon^*$, i.e.\
for $x\in \Omega_\epsilon$, we set 
\begin{equation}
W^\epsilon_t(x) \equiv \Pi_\epsilon W_t(x)=  \sum_{k \in \Omega_\epsilon^*} e^{ik\cdot x} \hat W_t(k),
\qquad x\in \Omega_\epsilon.
\end{equation}
Then $(W^\epsilon(x))_{x\in\Omega_\epsilon}$ are independent Brownian motions
indexed by $\Omega_\epsilon$ with quadratic variation $t/\epsilon^{2}$, see for instance \cite[Section 3.1]{MR4399156} for more details.
As in \eqref{eq:decomposed-gff-intro} we define the decomposed Gaussian free field $\Phi_t^{\GFF_\epsilon}$ by
\begin{equation} 
\label{eq:GFFepsFourier-pphi}
  \Phi_t^{\GFF_\epsilon} =
  \int_t^\infty q_s^\epsilon dW_s^\epsilon = 
  \sum_{k\in \Omega_\epsilon^*}  e^{ik\cdot(\cdot)}\int_t^\infty \hat q_u^\epsilon(k) d\hat W_u(k),
\end{equation}
where $\dot c_t^\epsilon = q_t^\epsilon * q_t^\epsilon$ and $\hat q_t (k)$ are as in \eqref{eq:fourier-multipliers-q}.
Note that $\Phi^{\GFF_\epsilon}$ has independent increments and that $\Phi_0^{\GFF_\epsilon}\sim \nu^{\GFF_\epsilon}$,
i.e.\ at $t=0$ the we obtain the discrete Gaussian free field on $\Omega_\epsilon$.
Moreover, we emphasise that the process $\Phi^{\GFF_\epsilon} = (\Phi^{\GFF_\epsilon}_t)_{t\geq 0}$ is adapted
to the \emph{backward} filtration $(\cF^t)$.

\subsection{Polchinski renormalisation group dynamics for $\cP(\phi)_2$}
\label{ssec:polchinski}

Let $v_0^\epsilon(\phi) = \epsilon^2 \sum_{x\in \Omega_\epsilon} \wick{\cP(\phi(x))}_\epsilon$ be the interaction for the measure $\nu^{\cP_\epsilon}$ in \eqref{eq:nu-p-eps}. We also refer to this object as Hamiltonian of the $\cP(\phi)_2$ field.
For the reasons described at the beginning of this section, we consider the following energy cut-off for $v_0^\epsilon$.
For $E>0$ let $\chi_E \in C^2(\R,\R)$ be concave and increasing such that
\begin{equation}
\chi_E(x)
=
\begin{cases}
x, \,\qquad & x \in (-\infty,E/2] 
\\ 
 E, \,\qquad & x > E
\end{cases}	
\end{equation}
and moreover
$\chi_E \leq \chi_{E'}$ on $[0,\infty)$ for $E \leq E'$ and $\sup_{E > 0} \| \chi_E' \|_{\infty} < 2$.
Then we define the cut-off Hamiltonian $v_0^{\epsilon,E} =\chi_E \circ v_0^{\epsilon}$ and the associated measure $\nu^{\cP_\epsilon, E}$ on $\Omega_\epsilon$ by
\begin{equation}
\label{eq:nu-p-eps-cut-off}
\nu^{\cP_\epsilon,E}(d\phi)
\propto
 e^{-v_0^{\epsilon,E}(\phi)} \nu^{\GFF_\epsilon}(d\phi)	.
\end{equation}
Next, we define the renormalised potential $v_t^{\epsilon,E}$ at scale $t\geq 0$ and for $E\in (0,\infty]$ by 
\begin{equation}
\label{eq:pphi-renormalised-potential-cut-off}
e^{-v_t^{\epsilon,E}(\phi)} = \EE_{c_t^\epsilon} [e^{-v_0^{\epsilon,E}(\phi + \zeta)}],
\end{equation}
where we recall that $\EE_{c_t^\epsilon}$ denotes the expectation with respect to the centred Gaussian measure with covariance $c_t^\epsilon$.

We think of  $v_t^{\epsilon,E}$ as the effective potential that sees fluctuations on the scales larger than the characteristic length scale $L_t=\sqrt{t} \wedge 1/m$.
Indeed, \eqref{eq:pphi-renormalised-potential-cut-off} can be interpreted as integrating out all small scale parts of the discrete GFF which are associated to the covariance $c_t^\epsilon$. 

The scale evolution of the renormalised potential is further encoded in the Polchinski equation, as stated in the following proposition. For a proof of this result, see for instance \cite{MR914427}.

\begin{lemma} 
\label{lem:polchinski-equation}
Let $\epsilon > 0$ and $E >0 $. Then $v^{\epsilon,E}$
 is a classical solution of the Polchinski equation
\begin{equs} 
\label{eq:polchinski-pphi-cut-off}
\partial_t v_t^{\epsilon,E}
=
\frac{1}{2} \Delta_{\dot c_t^\epsilon} v_t^{\epsilon,E} - \frac{1}{2} (\nabla v_t^{\epsilon,E})_{\dot c_t^\epsilon}^2
=
\frac{1}{2} \epsilon^4\sum_{x,y \in \Omega_\epsilon} \dot c_t^\epsilon(x,y)
  \qa{\ddp{^2v_t^\epsilon}{\varphi_x\partial\varphi_y}
    -
    \ddp{v_t^\epsilon}{\varphi_x}
    \ddp{v_t^\epsilon}{\varphi_y}
    }
\end{equs}
with initial condition $v_0^{\epsilon,E}$.
\end{lemma}

We record a priori estimates on the gradient and Hessian of $v_t^{\epsilon,E}$, where we make heavy use of the potential cut-off.
We emphasise that the bounds are uniform in $\phi\in X_\epsilon$, but not in $\epsilon>0$ and $E>0$.

\begin{lemma}
\label{lem:estimates-hess-nabla}
Let $\epsilon>0$ and $E \in (0,\infty)$. There exists $C=C(\epsilon,E)>0$ such that
\begin{align}
 \label{eq:gradient-bound}
\sup_{t \geq 0} \| \nabla v_t^{\epsilon,E} \|_{L^\infty(\Omega_\epsilon)}	
&\leq
C
\\
 \label{eq:hessian-bound}
\sup_{t \geq 0} \| \He v_t^{\epsilon,E} \|_{L^\infty(\Omega_\epsilon)}
&\leq
C,
\end{align}
where there is an implicit summation over the coordinates of $\nabla v_t^{\epsilon,E}$ and $\He v_t^{\epsilon,E}$ in the norms above.
\end{lemma}

\begin{proof}
We first observe that we have by the chain rule
\begin{equation}
\nabla v_0^{\epsilon,E}(\phi)
=
\chi_E'( v_0^\epsilon(\phi)) \, \nabla v_0^{\epsilon}(\phi).	
\end{equation}
Since $v_0^\epsilon$ is continuous in $\phi\in X_\epsilon$ and since $|v_0^\epsilon(\phi)| \to \infty$ as $\| \phi \|_{L^\infty(\Omega_\epsilon)} \to \infty$, 
we have that $\nabla v_0^{\epsilon,E}$ has bounded support.
Since the components of $\nabla v_0^\epsilon$ are polynomials in $\phi$, \eqref{eq:gradient-bound} follows for $t=0$.
For $t>0$, we obtain by differentiating \eqref{eq:pphi-renormalised-potential-cut-off}
\begin{equation}
\label{eq:nabla-v-t-E}
\nabla v_t^{\epsilon, E} (\phi) 
= e^{v_t^{\epsilon, E}(\phi)} \EE_{c_t^\epsilon} \big[\nabla v_0^{\epsilon, E}(\phi+\zeta) e^{-v_0^{\epsilon, E}(\phi+\zeta)} \big]. 
\end{equation}
Now the cut-off $\chi_E$ allows to bound $|v_t^{\epsilon, E}(\phi)| \lesssim_\epsilon E$ and hence, we can estimate
\begin{equation}
\| \nabla v_t^{\epsilon, E} (\phi) \|_{L^\infty(\Omega_\epsilon)} 
\lesssim_{\epsilon, E} \EE_{c_t^\epsilon} \big[ \| \nabla v_0^{\epsilon, E} (\phi+\zeta) e^{-v_0^{\epsilon, E}(\phi+\zeta)}  \|_{L^\infty(\Omega_\epsilon)} \big]
 \lesssim_{\epsilon, E} 
 \EE_{c_t^\epsilon} \big[ \| \nabla v_0^{\epsilon,E}(\phi+\zeta) \|_{L^2}  \big],
\end{equation}
from which \eqref{eq:gradient-bound} follows.
Similarly, we have
\begin{equation}
\He v_0^{\epsilon,E}  = \chi_E'' (v_0^\epsilon) \nabla v_0^\epsilon \cdot (\nabla v_0^\epsilon)^T + \chi_E'(v_0^\epsilon) \He v_0^
\epsilon,
\end{equation}
and thus, \eqref{eq:hessian-bound} follows from the same arguments for $t=0$. For $t>0$ the statement is obtained by differentiating \eqref{eq:nabla-v-t-E} and similar arguments and the bounds for $\nabla v_t^{\epsilon, E}$.
\end{proof}

We now introduce the Polchinski dynamics: a high-dimensional SDE driven by $\Phi_t^{\GFF_\epsilon}$ and with drift given by the gradient of the renormalised potential.
We then use the Polchinski equation to show that this provides coupling between the $\cP(\phi)_2$ field with cut-off $E$ and the discrete Gaussian free field on $\Omega_\epsilon$.
One of the key properties of this dynamic that makes it useful to study global probabilistic properties of the $\cP(\phi)_2$ measure is an independence property between small and large spatial scales.
Recall that we denote by $C_0([0, \infty), \cS)$ the space of continuous sample paths with values in a metric space $(\cS, \|\cdot\|_\cS)$ that vanish at infinity.

\begin{proposition} 
\label{prop:sde-well-posed-cut-off}
  For $\epsilon>0$ and $E>0$ there is a unique $\cF^t$-adapted process $\Phi^{\cP_\epsilon,E} \in C_0([0,\infty), X_\epsilon)$ such that
  \begin{equation}
    \label{eq:SDE-polchinski-cut-off}
    \Phi_{t}^{\cP_\epsilon,E}
    = - \int_t^\infty \dot c^\epsilon_u \nabla v_{u}^{\epsilon,E}(\Phi_u^{\cP_\epsilon,E}) \, du
      + \Phi_t^{\GFF_\epsilon}.
    \end{equation}
  In particular, for any $t>0$, $\Phi^{\GFF_\epsilon}_0-\Phi_t^{\GFF_\epsilon}$ is independent of $\Phi_t^{\cP_\epsilon,E}$.
Moreover,
$\Phi_0^{\cP_\epsilon,E}$ is distributed as the measure $\nu^{\cP_\epsilon,E}$
defined in \eqref{eq:nu-p-eps-cut-off}.
\end{proposition}

In removing the cut-off, the left-hand side of \eqref{eq:SDE-polchinski-cut-off} loses meaning as an adapted solution to a backwards SDE.
However, we show that the right-hand side can be made sense of in the limit and use this to define the left-hand side.
Moreover, the independence property in Proposition \ref{prop:sde-well-posed-cut-off} is preserved. 
As such, the finite variation integral in \eqref{eq:SDE-polchinski-cut-off} is one of the main quantities of interest in the remainder of this paper, and we record this notion in the following definition.

\begin{definition}
Let $\epsilon > 0$ and $E \in (0,\infty]$. The approximate difference field $\Phi^{\Delta_\epsilon,E} \in C([0,\infty), X_\epsilon)$ is the $(\cF^t)_{t \geq 0}$ adapted process defined by
\begin{equs}
\Phi^{\Delta_\epsilon, E}_t
=
-\int_t^\infty \dot c_s \nabla v_s^{\epsilon,E}(\Phi_s^{\cP_\epsilon, E}) ds	,
\qquad
t \geq 0,
\end{equs}
where $\Phi^{\cP_\epsilon, E}$ is the solution to \eqref{eq:SDE-polchinski-cut-off}.  In particular,
\begin{equs}
\label{eq:coupling-cut-off-fields}
\Phi_t^{\cP_\epsilon,E}
=
\Phi_t^{\Delta_\epsilon,E} + \Phi_t^{\GFF_\epsilon},
\qquad
t \geq 0.
\end{equs}
\end{definition}

\begin{proof}[Proof of Proposition \ref{prop:sde-well-posed-cut-off}]
We first prove that the SDE is well-defined and has a (pathwise) unique, strong solution on $[0,\infty)$. 
Then the independence property follows by independence of increments of the underlying cylindrical Brownian motion. 
Since $\epsilon>0$ is fixed throughout the proof, we suppress it from the notation when clear.
 
We first show that the coefficients of the SDE \eqref{eq:SDE-polchinski-cut-off} are uniformly Lipschitz thanks to the global Hessian bound \eqref{eq:hessian-bound}.
The proof is then completed by arguing as in \cite[Theorem 3.1]{MR4399156}. 
For $\cD,\cE \in C([0,\infty), X_\epsilon)$, let
\begin{equs}
\label{eq:F-Picard-def-pphi}
F_t(\cD,\cE) = -\int_t^\infty \dot c_s\nabla v_s^E(\cE_s+\cD_s) \, ds.
\end{equs}
By the mean value theorem and the bound \eqref{eq:hessian-bound}, we have, for $s\geq 0$ and $\Phi, \tilde \Phi \in X_\epsilon$, 
\begin{equs}
\label{eq:nabla-v-lipschitz}
\| \dot c_s \nabla v_s^E (\Phi) - \dot c_s \nabla v_s^E (\tilde \Phi)\|_{L^2} 
\leq 
\| \He v_s^E \|_{L^\infty} \| \dot c_s\| \|\Phi - \tilde \Phi \|_{L^2}
\lesssim_{\epsilon,E }
\| \dot c_s \| \|\Phi - \tilde \Phi \|_{L^2},
\end{equs}
where $\|\dot c_s\|$ denotes the spectral norm of the operator $\dot c_s \colon X_\epsilon \to X_\epsilon$. Thus, for $\cE \in C([0,\infty),X_\epsilon)$, we have that
\begin{equs} 
\label{eq:F-Lip-pphi}
\norm{F_t(\cD,\cE)-F_t(\tilde \cD,\cE)}_{L^2}
\leq 
  \int_t^\infty O_{\epsilon,E} (\| \dot c_s \| ) \norm{\cD_s-\tilde \cD_s}_{L^2} \, ds.
\end{equs}

Suppose first that there are two solutions $\Phi^{\cP}$ and $\tilde\Phi^{\cP}$ to \eqref{eq:SDE-polchinski-cut-off} satisfying $\cD\coloneqq\Phi^{\cP}-\Phi^{\GFF} \in C_0([0,\infty),X_\epsilon)$
and $\tilde \cD\coloneqq\tilde\Phi^{\cP}-\Phi^{\GFF} \in C_0([0,\infty),X_\epsilon)$.
Then, by \eqref{eq:F-Lip-pphi},
\begin{equs}
  \norm{\cD_t-\tilde \cD_t}_{L^2}
  &= 
  \norm{F_t (\cD,\Phi^{\GFF})-F_t(\tilde \cD,\Phi^{\GFF})}_{L^2}
    \nnb
  &\leq
   \int_t^\infty O_{\epsilon,E}(\|\dot c_s\|) \norm{\cD_s-\tilde \cD_s}_{L^2} \, ds.
\end{equs}
Thus, $f(t) = \norm{\cD_t-\tilde \cD_t}_{L^2}$ is bounded with $f(t)\to 0$ as $t\to\infty$ and additionally satisfies
\begin{equs}
  f(t) \leq a + \int_{t}^\infty O_{\epsilon, E} (\|\dot c_s\|) f(s) \, ds, \qquad a=0.
\end{equs}
Since $\| \dot c_s \| \lesssim_m \frac{1}{1+ s^2}$, we have $\int_{0}^\infty O(\| \dot c_s \|) \, ds < \infty$ and thus,
a version of Gronwall's inequality 
implies that for $t\geq 0$
\begin{equation}
f(t) \leq a \exp\pa{\int_t^\infty O_{\epsilon,E}( \|\dot c_s\|) \, ds} = 0,
\end{equation}
and hence, $\cD=\tilde \cD$ on $[0,\infty)$. 

That a solution to \eqref{eq:SDE-polchinski-cut-off} on $[0,\infty)$ exists follows from Picard iteration.
For $\cD\in C([0,\infty), X_\epsilon)$ and $t\geq 0$ let $\|\cD\|_t = \sup_{s\geq t} \norm{\cD_s}_{L^2}$.
Fix $\cE\in C_0([0,\infty), X_\epsilon)$ and
set $\cD^0=0$ and $\cD^{n+1}=F(\cD^n,\cE)$.
Then,
\begin{equs}
  \|\cD^1\|_t 
  = 
  \|F(0,\cE)\|_t &\leq \|\cE\|_t \int_t^\infty  O_{\epsilon,E}(\|\dot c_s\|)\,ds
  \\
  \| \cD^{n+1}-\cD^n \|_t
  &\leq 
  \int_t^\infty  O_{\epsilon,E}(\|\dot c_s\|) \| \cD^{n}-\cD^{n-1}\|_{s}\,ds ,
\end{equs}
and from the elementary identity 
\begin{equs}
  \int_{t}^\infty ds \, g(s) \pa{\int_s^\infty ds' \, g(s')}^{k-1} 
  = 
   \frac{1}{k} \pa{\int_{t}^\infty ds\, g(s)}^{k}
\end{equs}
applied with $g(s) = O( \|\dot c_s \|)$, we conclude that
\begin{equs}
\label{eq:picard-iterates-difference}
  \| \cD^{n+1}-\cD^n\|_t
  \lesssim_{\epsilon,E}
   \| \cE\|_{t} \frac{1}{n!} \pa{O(\| \dot c_t \|)}^n.
\end{equs}

Since the right-hand side of \eqref{eq:picard-iterates-difference} is summable, we have that $\cD^n \to \cD^*$ for some $\cD^* = \cD^*(\cE) \in C_0([0,\infty),X_\epsilon)$ in $\|\cdot\|_{0}$,
and the limit satisfies $F_s(\cD^*(\cE),\cE) = \cD^*(\cE)_s$ for $s\geq {0}$.
Now, we apply this result for $\cE=\Phi^\GFF$ noting that, from the representation \eqref{eq:GFFepsFourier-pphi},
\begin{equs}
 \E \big[ \|\Phi_t^\GFF\|_{L^2}^2 \big]  
 = 
 \sum_{k\in \Omega_\epsilon^*} \int_t^\infty \frac{1}{\big(s(-\hat \Delta^\epsilon(k) + m^2) +1)  \big)^2} ds \lesssim  1/t,
\end{equs}
which implies that $\Phi^\GFF \in C_0([0,\infty), X_\epsilon)$ a.s.
In summary, $\Phi^{\cP, E}_t = D^*(\Phi^\GFF)_t + \Phi^{\GFF}_t$ is the desired solution on $[0,\infty)$.

We now prove that $\Phi_0^{\cP,E}$ is distributed as $\nu^{\cP,E}$ as defined in \eqref{eq:nu-p-eps-cut-off}, for which we proceed similarly as in the proof of \cite[Theorem 3.2]{MR4399156}.
Let $\nu_t^{\cP,E}$ be the renormalised measure defined by 
\begin{equs}
\label{eq:ren-measure-cut-off}
\E_{\nu_t^{\cP,E}} [F] 
= 
e^{\vE_\infty(0)} \EE_{c_\infty -c_t}[e^{-\vE_t(\zeta)}F(\zeta)],
\end{equs}
where $\EE_{c_t}$ denotes the expectation of the Gaussian measure with covariance $c_t$.
Then, let $\tilde\Phi^T$ be the unique strong solution to the (forward) SDE
\begin{equation}
  d\tilde\Phi_t^T = -\dot c_{T-t} \nabla \vE_{T-t}(\tilde\Phi_t^T) \, dt + q_{T-t} \, d\tilde W_t^T,
  \qquad 0 \leq t \leq T
\end{equation}
with initial condition $\tilde \Phi_0\sim \nu_T^{\cP,E}$ and $\tilde W^{T}_t = W_T-W_{T-t}$.
Existence and uniqueness can be seen by the exact same arguments as above for the case $T=\infty$.
Using the Polchinski semigroup $\PP_{s,t}$, $s\leq t$ as defined in \cite[(1.3--1.4)]{MR4303014} to evolve the measure  $\nu_T^{\cP,E}$, it follows that $\tilde \Phi_T^T$ is distributed as the measure \eqref{eq:nu-p-eps-cut-off}.
Reversing the direction of $t$ and setting  $\Phi_t^T = \tilde\Phi_{T-t}^T$ we obtain 
\begin{align} 
 \label{e:varphiT-SDE}
  \Phi_t^T
  &= \Phi_T^T - \int_0^{T-t} \dot c_{T-s} \nabla v_{T-s}^E(\tilde\Phi_s^T) \, ds + \int_0^{T-t} q_{T-s} \, d\tilde W_s^T
    \nnb
  &= \Phi_T^T - \int_t^T \dot c_{s} \nabla v_{s}^E(\Phi_s^T) \, ds + \int_t^T q_{s} \, dW_s
    .
\end{align}
Note that this yields a coupling of all solutions $\PhiE$, $\Phi^T$, $T>0$.
Therefore, we have with $\PhiE_\infty=0$ that
\begin{align}
  \PhiE_t - \Phi_t^T
  &= (\PhiE_\infty-\Phi_T^T)
  - \int_t^T \qa{\dot c_{s} \nabla \vE_{s}(\PhiE_s)-\dot c_{s} \nabla \vE_{s}(\Phi_s^T)} \, ds  \nnb
  & - \int_T^\infty \dot c_{s} \nabla \vE_{s}(\PhiE_s) \, ds + \int_T^\infty q_{s} \, dW_s.
\end{align}
We will show that as $T\to \infty$, we have $\| \PhiE_0 - \Phi_0^T\|_{L^2}\to 0$ in probability, from which we deduce that $\PhiE_0\sim \nu^{\cP,E}$. 
In what follows, we denote by $\| \dot c_t\|$ the operator norm of $\dot c_t$ when seen as an operator $L^2(\Omega_\epsilon) \to L^2(\Omega_\epsilon)$. 

The first, third, and fourth terms on the right-hand side above are independent of $t$ and
they converge to $0$ in probability as $T \to\infty$.
Indeed, for the first term this follows from the weak convergence of the
measure $\nu_T^{\cP,E}$ to $\delta_0$, e.g.,
in the sense of \cite[(1.6)]{MR4303014}.
The third term is bounded by
\begin{equation}
  \int_T^\infty \|\dot c_s \nabla v_s^E(\PhiE_s) \|_{L^2} \, ds
\end{equation}
which converges to $0$ in $L^1$ as $T\to \infty$ by \eqref{eq:gradient-bound} and the fact that $\|\dot c_t \| = O(1/t^2)$.
The fourth term is a Gaussian field on $\Omega_\epsilon$ with covariance matrix
$c_\infty - c_T \to 0$
as $T\to\infty$. Since $\Omega_\epsilon$ is finite,
it is a trivial consequence that 
this Gaussian field convergences to $0$.

In summary, we have shown that there is $\cR_T$ such that $\norm{\cR_T} \to 0$ in probability, and
\begin{equation}
  \PhiE_t - \Phi_t^T
  = - \int_t^T \qa{\dot c_{s} \nabla \vE_{s}(\PhiE_s)-\dot c_{s} \nabla \vE_{s}(\Phi_s^T)} \, ds  + \cR_T.
\end{equation}
For any  $t\geq 0$, we have by \eqref{eq:nabla-v-lipschitz} that
\begin{equation}
\int_t^T \| \dot c_s \nabla v_s^E (\PhiE_s) - \dot c_s \nabla v_s^E ( \Phi_s^T)\|_{L^2} ds
\lesssim_{\epsilon,E} \int_t^T \| \dot c_t \| \|\PhiE_s - \Phi_s^T \|_{L^2} ds
\end{equation}
so that we have with $M_t= \| \dot c_t \|$ 
\begin{equation}
  \int_t^T \|\dot c_{s} \nabla v_{s}^E(\Phi_s)-\dot c_{s} \nabla v_{s}^E(\Phi_s^T) \|_{L^2} \, ds
  \lesssim_{\epsilon,E} \int_t^T  M_s \norm{\PhiE_s-\Phi_s^T}_{L^2}  \, ds.
\end{equation}
Thus, we have shown that $\cD_t = \PhiE_{t}-\Phi_t^T$ satisfies
\begin{align}
  \norm{\cD_t}_{L^2}
  \lesssim_{\epsilon,E} \norm{\cR_T}_{L^2} +  \int_{t}^T  M_s \norm{\cD_s}_{L^2}  \, ds.
\end{align}
Since $\int_{0}^\infty M_s \, ds < \infty$, the same version of Gronwall's inequality as above implies that
\begin{equation}
  \norm{\cD_t}_{L^2}
  \lesssim_{\epsilon,E} \norm{\cR_T}_{L^2} \exp\pa{\int_t^T M_s \, ds}
  \leq \norm{\cR_T}_{L^2} \exp\pa{\int_{t_0}^\infty M_s \, ds}
  \lesssim \norm{\cR_T}_{L^2}.
\end{equation}
Since the right-hand side is uniform in $t\geq 0$, we conclude that $\sup_{t\geq 0} \| \cD _t\|_{L^2} \to 0$ in probability as $T\to \infty$. 
\end{proof}

\subsection{A stochastic control representation via the Bou\'e-Dupuis formula}
\label{ssec:boue-dupuis}

We now turn to a stochastic control representation of the renormalised potential  $v_t^{\epsilon,E}$ defined in \eqref{eq:pphi-renormalised-potential-cut-off} based on the Bou\'e-Dupuis variational formula,
which allows to express the moment generating function for functionals of Brownian motion as an expectation of the shifted functional and a drift term. 
For this, we interpret the drift part of our SDE as a minimiser of the control problem.
This correspondence is known as the verification principle in stochastic control theory, see \cite[Section 4]{MR2976505}. 
To connect to the Polchinski renormalisation group approach of Section \ref{ssec:polchinski} and also to the notation in \cite{MR4173157} we write
\begin{equation}
\label{eq:Y-definition}
Y_t^\epsilon
=
\int_0^t q_s^\epsilon d W_s^\epsilon = \Phi_0^{\GFF_\epsilon} - \Phi_t^{\GFF_\epsilon}
\end{equation}
for the small scale of the Gaussian free field process $\Phi^{\GFF_\epsilon}$. Note that $Y_t^\epsilon$ is a Gaussian field with covariance $c_t^\epsilon$, which follows from standard stochastic analysis results.
Thus, for $\phi\in X_\epsilon$, the renormalised potential $v_t^{\epsilon,E}$ can be likewise expressed as
\begin{equation}
e^{-v_t^{\epsilon,E}(\phi)}
=
\E [e^{-v_0^{\epsilon,E}(Y_t^\epsilon + \phi)}].
\end{equation}
The right-hand side is now expressed as a measurable functional of Brownian motions.
One can then exploit continuous-time martingale techniques, in particular Girsanov's theorem, to analyse this expectation.
This underlies the stochastic control representation that we now present.

Let $\HH_a$ be the space of progressively measurable (with respect to the backward filtration $\cF^t$) processes that are a.s.\ in $L^2(\R_0^+ \times \Omega_\epsilon)$, i.e.\ 
$u \in \HH_a$ if and only if $u|_{[t,\infty)}$ is $\cB([t,\infty)) \otimes \cF^t$-measurable for every $t\geq 0$ and 
\begin{equation}
\int_{0}^\infty \|u_\tau\|_{L^2}^2 d\tau < \infty \quad \text{a.s.}
\end{equation}
Here $\cB([t,\infty))$ denotes the Borel $\sigma$-algebra on $[t,\infty)$. 
The restriction of $\mathbb{H}_a$ to a finite interval $[0,t]$ is denoted $\mathbb{H}_a[0,t]$ with the convention $\mathbb{H}_a[0,\infty]= \mathbb{H}_a$. 
We will refer to elements $u\in \HH_a$ as drifts. 
For $u\in \HH_a$ and $0\leq s \leq t \leq \infty$ define the integrated drift
\begin{align}
I_{s,t}^\epsilon(u) &= \int_s^t q_\tau^\epsilon u_\tau d\tau
\end{align}
with the convention $I_{0,t}^\epsilon(u) \equiv I_{t}^\epsilon(u)$.
The following proposition is the Bou\'e-Dupuis formula for the renormalised potential.
We state it in the conditional form in order to be able to draw the correct comparison to the Polchinski renormalisation group approach afterwards. 

\begin{proposition}[Bou\'e-Dupuis formula] 
\label{prop:boue-dupuis-formula}
Let $\epsilon>0$. Then the conditional Bou\'e-Dupuis formula holds for the renormalised potential $v_t^{\epsilon,E}$, i.e.\ for 
$t\in [0, \infty]$
\begin{equation}
\label{eq:boue-dupuis}
 -\log \E \big[e^{-v_0^{\epsilon,E}(Y_t^\epsilon +  \Phi_t^{\cP_\epsilon,E})} \bigm | \cF^t\big ]	 
=
 \inf_{u \in \mathbb{H}_a[0,t]} \E \Big[v_0^{\epsilon,E} \big (Y_t^\epsilon + \Phi_t^{\cP_\epsilon,E} +  I_t^\epsilon(u)\big) + \frac{1}{2} \int_0^t \|u_s\|_{L^2}^2 ds \bigm | \cF^t\Big].
\end{equation}
\end{proposition}

\begin{proof}
Since $Y_t^\epsilon$ is independent of $\Phi_t^{\cP_\epsilon,E}$ we have by standard properties of conditional expectation 
\begin{align}
-\log \E\big[ e^{-v_0^{\epsilon,E}(Y_t^{\epsilon}+\Phi_t^{\cP_\epsilon,E})} \bigm | \cF^t \big] = -\log \E\big[e^{-v_0^{\epsilon,E}(Y_t^{\epsilon}+\phi)}\big]_{\phi=\Phi_t^{\cP_\epsilon,E}}.
\end{align}
From \cite[Theorem 8.3]{MR1675051}, we have that for a deterministic $\phi\in X_\epsilon$ the unconditional expectation on right-hand side of this display is equal to
\begin{align} 
\label{bd1}
-\log \E\big[ e^{-v_0^{\epsilon,E}(Y_t^\epsilon+\phi)} \big] 
&=
\inf_{u \in \mathbb{H}_{a}[0,t]} \E\Big[ v_0^{\epsilon,E}(Y_t^\epsilon+\phi+I_t^\epsilon(u))+\frac{1}{2}\int_{0}^{t}\|u_s \|^{2}_{L^{2}(\Omega_{\epsilon})} ds\Big],
\end{align}
from which the claim follows.
\end{proof}

\subsection{Correspondence between the Polchinski and the Bou\'e-Dupuis representations}
\label{ssec:correspondence}

The following proposition describes the equivalence between the Polchinski renormalisation group dynamics and the stochastic control representation via the Bou\'e-Dupuis formula in the presence of a potential cut-off $E > 0$. 


%

\begin{proposition}
\label{prop:bd-polchinski-correspondence}
Let $\epsilon > 0$, $E\in (0,\infty)$ and
let $\Phi^{\cP_\epsilon,E} \in C_0([0,\infty), X_\epsilon)$ be the unique strong solution to 
\eqref{eq:SDE-polchinski-cut-off}. 
Let $\minimiserE\colon [0,t]\times\Omega_\epsilon \to \R$ denote the process defined by
\begin{equation}
\minimiserE_s
=
-q_s^\epsilon \nabla v_s^{\epsilon,E} (\Phi_s^{\cP_\epsilon,E}),
\qquad s\in [0,\infty).	
\end{equation}
Then $u^E|_{[0,s]}$ is a minimiser of the conditional Bou\'e-Dupuis variational formula \eqref{eq:boue-dupuis}.
In particular, the relation between $u^E$ and the difference field $\Phi^{\Delta_\epsilon,E}$ is given by
\begin{equation}
\Phi_t^{\Delta_\epsilon,E} = \int_t^\infty q_s^\epsilon \minimiserE_s ds.
\end{equation}
\end{proposition}

\begin{proof}
Since $(\Phi_t^{\cP_\epsilon,E})_{t\in [0,\infty]}$ is $\cF^t$-measurable and continuous, it follows that $u^{\epsilon,E} \in \HH_a^\epsilon[0,t]$ for any $t\in [0,\infty]$.
To ease the notation, we drop $\epsilon$ throughout this proof.
Applying  It\^o's formula to $\vE_{T-\tau}(\PhiE_{T-\tau})$ for $ \tau \leq T<\infty$ and substituting $t=T-\tau$ thereafter, we obtain
\begin{align}
d \vE_t(\PhiE_t) &= - \nabla \vE_t(\PhiE_t) \dot c_t \nabla \vE_t(\PhiE_t) dt - \partial_t \vE_t(\PhiE_t) dt \nnb 
& +\frac{1}{2} \tr\big(\He \vE_t \dot c_t\big) dt
+\nabla \vE_t(\PhiE_t) \dot c_t^{1/2} dW_t \nnb
&= -\big(\nabla \vE_t (\PhiE_t)\big)_{\dot c_t}^2 - \partial_t \vE_t(\PhiE_t) dt + \frac{1}{2} \Delta_{\dot c_t} \vE_t(\PhiE_t) dt  + \nabla \vE_t(\PhiE_t) \dot c_t^{1/2} dW_t\nnb
&= -\frac{1}{2} \big(\nabla \vE_t (\PhiE_t)\big)_{\dot c_t}^2 dt + \nabla \vE_t(\PhiE_t) \dot c_t^{1/2} dW_t,
\end{align}
where we used the Polchinski equation \eqref{eq:polchinski-pphi-cut-off} in the last step. Integrating from $0$ to $t$ and taking conditional expectation yields
\begin{equation}
\label{eq:ito-conditional-expectation}
\E \Big[ \vE_0(\PhiE_0) - \vE_t (\PhiE_t) \bigm | \cF^t \Big] = - \frac{1}{2}\E \Big[ \int_0^t \big(\nabla \vE_s (\PhiE_s)\big)_{\dot c_s}^2  ds \bigm | \cF^t \Big],
\end{equation}
where we used that for $t\in [0,\infty]$
\begin{equation}
\E \Big[ \int_0^t \nabla \vE_s(\PhiE_s) \dot c_s^{1/2} d W_s \bigm | \cF^t \Big] = 0
\end{equation}
which holds by \eqref{eq:gradient-bound} 
and the fact that $\|\dot c_t \| = O(1/t^2)$.
Since $Y_t$ is independent of $\cF^t$ and $\PhiE_t$ is $\cF^t$-measurable, we have by standard properties of conditional expectation
\begin{equation}
e^{-\vE_t(\PhiE_t)}  = \E\Big [e^{-\vE_0(Y_t + \PhiE_t)} \bigm | \cF^t \Big].
\end{equation}
Therefore, we obtain from \eqref{eq:ito-conditional-expectation}
\begin{align}
\E &\Big[ \frac{1}{2}\int_0^t \big(\nabla \vE_s (\PhiE_s)\big)_{\dot c_s}^2  ds \bigm | \cF^t \Big] =  \vE_t(\PhiE_t) - \E \Big[   \vE_0(\PhiE_0) \bigm | \cF^t\Big] \nnb
&=  -  \log \E\Big [e^{-\vE_0(Y_t^\epsilon + \PhiE_t)} \bigm | \cF^t \Big] - \E \Big[  \vE_0\Big(\int_0^\infty q_t dW_t - \int_0^\infty \dot c_t \nabla \vE_t(\PhiE_t)dt \Big) \bigm |\cF^t \Big].
\end{align}
Rearranging this and using \eqref{eq:Y-definition}  show that
\begin{align}
-  \log \E\Big [e^{-\vE_0(Y_t^\epsilon + \PhiE_t)} \bigm | \cF^t \Big] 
&= \E \Big[   \vE_0\Big(\int_0^\infty q_t dW_t - \int_0^\infty \dot c_t \nabla \vE_t(\PhiE_t)dt \Big) + \frac{1}{2}\int_0^t \big(\nabla \vE_s (\PhiE_s)\big)_{\dot c_s}^2  ds \bigm | \cF^t \Big] \nnb
& = \E \Big[   \vE_0\Big( Y_t^\epsilon + \PhiE_t - \int_0^t \dot c_s \nabla \vE_s(\PhiE_s)ds   \Big)+ \frac{1}{2}\int_0^t \big(\nabla \vE_s (\PhiE_s)\big)_{\dot c_s}^2  ds \bigm | \cF^t \Big],
\end{align}
which proves that $\minimiserE_s = -q_s \nabla \vE_s(\PhiE_s)$, $s\in [0,t]$ is a minimiser for  \eqref{eq:boue-dupuis}.
\end{proof}

\section{Fractional moment estimate on the renormalised potential}

In this section we prove a fractional moment estimate on the small-scale behaviour of the renormalised potential.
We exploit the connection between the Polchinski dynamics and a minimiser of the Bou\'e-Dupuis variational problem \eqref{eq:boue-dupuis} as in Proposition \ref{prop:bd-polchinski-correspondence} to transfer estimates on minimisers onto the renormalised potential.
First, we prove some a priori bounds on Sobolev norms of the integrated drift. 

\subsection{Sobolev norms of integrated drifts}
\label{sec:sobolev-norms-int-drift}

Recall that, for $0\leq s\leq t \leq \infty$ and $u\in \HH_a$, the integrated drift is defined by
\begin{equs}
I_{s,t}^\epsilon(u) 
= 
\int_s^t q_\tau^\epsilon u_\tau d\tau, \qquad q_\tau^\epsilon =\big (\tau (-\Delta^\epsilon + m^2) + 1\big)^{-1}
\end{equs}
and that  $u$ is  a.s.\ $L^2$ integrable on $[0,\infty]$, i.e.\ $\int_0^\infty \|u_\tau \|_{L^2}^2 d\tau<\infty$ a.s. In Fourier space this condition reads as 
\begin{equs}
\int_0^\infty \|u_\tau \|_{L^2}^2 d\tau = \sum_{k\in \Omega_\epsilon^*}\int_0^\infty |\hat u_\tau(k)|^2 d\tau <\infty \quad \text{a.s.,}
\end{equs}
where $\hat u_\tau(k)$, $k\in \Omega_\epsilon^*$ denote the Fourier coefficients of $u_\tau$. 

In what follows we will discuss the regularity of $I_{s,t}^\epsilon(u)$ for different choices of $s,t$, for which we use the Sobolev norm defined in \eqref{eq:sobolev-norm-def}. Note that we have
\begin{equs}
\|I_{s,t}^\epsilon(u) \|_{H^\alpha}^2= \sum_{k\in \Omega_\epsilon^*}    (1+|k|^2)^\alpha | \widehat{ I_{s,t}^\epsilon(u) } |^2 =\sum_{k\in \Omega_\epsilon^*}    (1+|k|^2)^\alpha \Big| \int_s^t \hat q_\tau^\epsilon(k) \hat u_\tau(k) d\tau\Big|^2.
\end{equs}
In addition, recall that the Fourier multipliers of $q_\tau^\epsilon$ are given by
\begin{equs}
\hat q_\tau^\epsilon(k) 
= 
\frac{1}{\tau (- \hat \Delta^\epsilon (k) + m^2) + 1}, \qquad k\in \Omega_\epsilon^*,
\end{equs}
where $-\hat \Delta^\epsilon(k)$ are as in \eqref{eq:fourier-multipliers-discrete-laplacian}. Using \eqref{eq:fourier-multipliers-laplacians-difference}, we have that
\begin{equs}
-\hat \Delta^\epsilon (k) 
\geq
 c |k|^2 
\end{equs}
for $k\in \Omega_\epsilon^*$ with $c=\frac{4}{\pi^2}$. Hence,  we have for $k\in \Omega_\epsilon^*$
\begin{equs}
\label{eq:q-fourier-multipliers}
\hat q_\tau^\epsilon(k) 
\leq
 \frac{1}{\tau (c|k|^2 + m^2) + 1}.
\end{equs}

The following results establish bounds for Sobolev norms of  the large  and small scales of  $I^\epsilon(u)$ for a given drift $u\in \HH_a$. For the rest of this subsection we drop $\epsilon>0$ from the notation.

\begin{lemma}
\label{lem:integrated-drift-large-scales-h2}
Let $\alpha \in [1,2]$. Then,  for {$0<t<1$},
\begin{equs}
\| I_{t,\infty} (u) \|_{H^\alpha}^2 
\lesssim_{\alpha} 
\frac{1}{t^{\alpha -1}} \int_t^\infty \| u_\tau\|_{L^2}^2 d\tau.
\end{equs}
\end{lemma}

\begin{proof}
By the Cauchy-Schwarz inequality and \eqref{eq:q-fourier-multipliers} we have
\begin{equs}
\| I_{t,\infty} (u) \|_{H^\alpha}^2 
&\leq
 \sum_{k\in \Omega_\epsilon^*}    (1+|k|^2)^\alpha  \big( \int_t^\infty |\hat q_\tau(k) \hat u_\tau(k)| d\tau \big)^2 
 \nnb
&\leq 
\sum_{k\in \Omega_\epsilon^*}(1+|k|^2)^\alpha \big( \int_t^\infty |\hat q_\tau(k) |^2 d\tau \big) \big ( \int_t^\infty |\hat u_\tau(k)|^2 d\tau \big) 
\nnb
& \leq 
\sum_{k\in \Omega_\epsilon^*}(1+|k|^2)^\alpha \frac{1}{(c|k|^2+ m^2)}\frac{1}{t(c|k|^2+ m^2) +1} \int_t^\infty |\hat u_\tau(k)|^2 d\tau 
\nnb
& \leq 
\sum_{k\in \Omega_\epsilon^*} \frac{(1+|k|^2)}{(c|k|^2+ m^2)}\frac{(1+|k|^2)^{\alpha-1}}{t(c|k|^2+ m^2) +1} \int_t^\infty |\hat u_\tau(k)|^2 d\tau 
\nnb
& \lesssim_{\alpha}
 \sum_{k\in \Omega_\epsilon^*} \frac{1}{t^{ \alpha-1}} \int_t^\infty |\hat u_\tau(k)|^2 d \tau = \frac{1}{t^{\alpha-1}} \int_t^\infty \|u_\tau\|_{L^2}^2 d\tau,
\end{equs}
where we used that $\sup_{x\geq 0}\frac{(1+x)^\beta}{1+tx} {\lesssim_\beta} \frac{1}{t^\beta}$ for $t<1$ and $\beta<1$.
\end{proof}

\begin{lemma}
\label{lem:integrated-drift-large-scales-h1}
Let $\alpha \in [0,1]$. Then,
\begin{equs}
\sup_{t\geq 0} \| I_{t,\infty} (u) \|_{H^\alpha}^2 \lesssim  \int_0^\infty \| u_\tau\|_{L^2}^2 d\tau.
\end{equs}
\end{lemma}

\begin{proof}
Using the same calculations as above, we see that
\begin{equs}
\| I_{t,\infty} (u) \|_{H^\alpha}^2 &\leq \sum_{k\in \Omega_\epsilon^*}    (1+|k|^2)^\alpha  \big( \int_t^\infty |\hat q_\tau(k) \hat u_\tau(k)| d\tau \big)^2 \nnb
&\leq \sum_{k\in \Omega_\epsilon^*}(1+|k|^2)^\alpha \big( \int_t^\infty |\hat q_\tau(k) |^2 d\tau \big) \big ( \int_t^\infty |\hat u_\tau(k)|^2 d\tau \big) \nnb
& \leq \sum_{k\in \Omega_\epsilon^*}(1+|k|^2)^\alpha \frac{1}{(c|k|^2+ m^2)}\frac{1}{t(c|k|^2+ m^2) +1} \int_t^\infty |\hat u_\tau(k)|^2 d\tau \nnb
& \lesssim \sum_{k\in \Omega_\epsilon^*} \frac{(1+|k|^2)^\alpha}{(c|k|^2+ m^2)}\int_t^\infty |\hat u_\tau(k)|^2 d\tau \leq \int_t^\infty \|u_\tau\|_{L^2}^2 d\tau.
\end{equs}
Taking the supremum over $t\geq 0$ establishes the estimate.
\end{proof}

\begin{lemma}
\label{lem:integrated-drift-small-scales-h1}
For $0\leq \alpha \leq 1$ and $0\leq s\leq t$, we have
\begin{equs}
\| I_{s,t} (u) \|_{H^\alpha}^2
 \lesssim_{\alpha} 
 (t-s)^{1-\alpha} \int_s^t \|u_\tau\|_{L^2}^2 d\tau.
\end{equs}
\end{lemma}

\begin{proof}
By the same arguments that were used in the proof of the previous statements, we obtain
\begin{equs}
\| I_{s,t} (u) \|_{H^\alpha}^2 
&\leq 
\sum_{k\in \Omega_\epsilon^*} (1+|k|^2)^{\alpha}  \Big ( \int_s^t |\hat q_\tau(k) \hat u_\tau(k)| d\tau \Big)^2  
\nnb
&\leq 
\sum_{k\in \Omega_\epsilon^*} (1+|k|^2)^{\alpha} \int_s^t \frac{1}{\tau(c|k|^2 + m^2) + 1  )^2} d\tau \int_s^t | \hat u_\tau(k)|^2 d\tau
 \nnb
& = 
\sum_{k\in \Omega_\epsilon^*} (1+|k|^2)^{\alpha} \Big [ \frac{1}{(c|k|^2 +m^2)} \Big( \frac{1}{s(c|k|^2 +m^2)+1} - \frac{1}{t(c|k|^2 +m^2)+1}\Big)\Big] \int_s^t |\hat u_\tau(k)|^2 d\tau 
\nnb
&= 
\sum_{k\in \Omega_\epsilon^*} (1+|k|^2)^{\alpha} \Big [ \frac{t-s}{\big(s(c|k|^2 +m^2) + 1\big)\big( t(c|k|^2 +m^2) + 1 \big) } \Big] \int_s^t |\hat u_\tau(k)|^2 d\tau 
\nnb
&  \leq 
( t -s) \sum_{k \in \Omega_\epsilon^*} \frac{(1 + |k|^2)^\alpha}{t(c|k|^2 +m^2)+1} \int_s^t | \hat u_\tau(k)|^2 d\tau 
\nnb
&\lesssim_{\alpha} 
(t-s) \frac{1}{t^\alpha} \int_s^t \|u_\tau\|_{L^2}^2 d\tau \leq (t-s)^{1-\alpha} \int_s^t \|u_\tau\|_{L^2}^2 d\tau.
\end{equs}
\end{proof}

\begin{lemma}
\label{lem:integrated-drift-small-scales-h2}
For $\alpha \in(1, 2]$ and $0<s\leq t$, we have
\begin{equs}
\| I_{s,t} (u) \|_{H^\alpha}^2 
\lesssim_{\alpha} 
\frac{t-s}{s^\alpha} \int_s^t \|u_\tau\|_{L^2}^2 d\tau .
\end{equs}
\end{lemma}

\begin{proof}
The following calculations are similar to before. We include them here for completeness.
\begin{equs}
\| I_{s,t}(u) \|_{H^\alpha}^2 
&\leq 
\sum_{k\in \Omega_\epsilon^*} (1+ |k|^2)^\alpha  \int_s^t \frac{1}{(\tau(c|k|^2 +m^2) + 1)^2} d\tau \int_s^t |\hat u_\tau(k)|^2 d \tau 
\nnb
&\lesssim_{\alpha}
 \sum_{k\in \Omega_\epsilon^*} \int_s^t \frac{1}{\tau^\alpha} d\tau \int_s^t |\hat u_\tau(k)|^2 d \tau 
 =
  \int_s^t \frac{1}{\tau^\alpha} d\tau \sum_{k\in \Omega_\epsilon^*} \int_s^t |\hat u_\tau(k)|^2 d \tau 
 \nnb
& = 
\frac{1}{s^\alpha} (t-s) \int_s^t \|u_\tau\|_{L^2}^2 d\tau .
\end{equs}
\end{proof}

\subsection{Uniform $L^2$ estimate for minimisers} 
\label{ssec:l2-estimates}

As an application of the estimates established in the previous section we now show that if a drift minimises the functional in \eqref{eq:boue-dupuis} for $t=\infty$,
then the expectation of its $L^2$ norm is finite as stated in the following proposition. 

\begin{proposition}
\label{prop:minimiser-l2}
Assume that $\bar u$ minimises the functional 
\begin{equation}
F^{\infty,E} (u) =  \E\big [  v_0^{\epsilon,E}(Y_\infty^\epsilon+ I_\infty^\epsilon (u) ) + \frac{1}{2} \int_0^\infty \| u_t\|_{L^2}^2 dt   \big] 
\end{equation}
in $\mathbb{H}_a$. Then
\begin{equation}
\label{eq:minimiser-l2}
\sup_{\epsilon>0}\sup_{E>0} \E\Big[\frac{1}{2} \int_0^\infty \|\bar u_t\|_{L^2}^2 dt\Big] < \infty.
\end{equation}
\end{proposition}

Before we prove Proposition \ref{prop:minimiser-l2}, we establish the following auxiliary result to fix the numerology in the analysis when using interpolation inequalities,
which will allow to treat all polynomials $\cP$ as in \eqref{eq:wick-polynomials} simultaneously.

\begin{lemma}
\label{lem:numerology}
Let $l,N \in \N $ such that $1\leq l \leq N-1$. There exist $\theta_{l,N}\in (0,1)$ and $\alpha_{l,N},\beta_{l,N},\gamma_{l,N} \in (0,\infty)$ satisfying
\begin{equs}
\label{eq:numerology-definitions}
\alpha_{l,N} = \frac{2}{\theta_{l,N} l},
\qquad 
\beta_{l,N}=\frac{N}{(1-\theta_{l,N})l},
\qquad 
\frac{1}{\gamma_{l,N}}=1-\frac{1}{\alpha_{l,N}}-\frac{1}{\beta_{l,N}}
\end{equs}
such that the following inequalities hold:
\begin{equs}
\label{eq:conditions-interpolation}
  \frac{1}{\alpha_{l,N}}+\frac{1}{\beta_{l,N}}<1
  \qquad
  \text{and}
  \qquad
  \gamma_{l,N} \geq \alpha_{l,N}.
\end{equs}
\end{lemma}
 
 \begin{proof}
The first inequality can be rewritten as
\begin{equs}
\frac{\theta_{l,N} l}{2}+\frac{(1-\theta_{l,N})l}{N} < 1
\quad \iff \quad
\theta_{l,N} < \frac{2(N-1)}{l(N-2)}.
\end{equs}
To reformulate the second condition, we compute
\begin{equs}
\frac{1}{\gamma_{l,N}} = 1 - \frac{1}{\alpha_{l,N}} - \frac{1}{\beta_{l,N}}
=
1 - \frac{\theta_{l,N}l}2 - \frac{(1-\theta_{l,N} )l}{N}
 =
\frac{2N - \theta_{l,N} Nl - 2(1-\theta_{l,N})l}{2N}.
\end{equs}        
Thus, to satisfy the second inequality in \eqref{eq:conditions-interpolation}, we need to choose $\theta_{l,N}$ such that
\begin{equs}
\frac{N}{2N - \theta_{l,N} Nl - 2(1-\theta_{l,N})l}
\geq
\frac{1}{\theta_{l,N}l},
\end{equs}
which is equivalent to 
\begin{equs}
2 N \theta_{l,N}l \geq 2N -2l + 2l \theta_{l,N}
\quad \iff \quad 
\theta_{l,N} \geq \frac{N-l}{(N+1)l}.
\end{equs}
Hence, both conditions in \eqref{eq:conditions-interpolation} are satisfied, if we choose $\theta_{l,N} \in (0,1)$ satisfying
\begin{equs}
\frac{2(N-1)}{ l(N-2)} > \theta_{l,N} \geq \frac{N-l}{ l(N+1)},
\end{equs}
which is possible for all specified values of $l,N$.
\end{proof}

\begin{lemma}
\label{lem:duality-power-integrated-drift}
For every $l,N \in \N$ such that $1 \leq l \leq N-1$ and for every $\kappa>0$, there exists $\tilde\kappa_{l,N} >0$ such that the following estimate holds.
For any $\delta > 0$ sufficiently small there exists $C>0$ such that, for $f\in \cC^{- \tilde\kappa_{l,N}} $ and $u\in L^{2}(\R_0^+ \times \Omega_{\epsilon})$, we have, for $t \in (0,\infty]$,
\begin{equs}
\label{eq:duality-power-integrated-drift}
\Big|\int_{\Omega_\epsilon} f I_t^l(u)dx \Big| \leq C (t\wedge 1)^{1-\kappa}\|f\|^{\gamma_{l,N}}_{\cC^{-\tilde \kappa_{l,N}} } + \delta a_N \|I_{t}(u)\|^{N}_{L^N}+ \delta \int_{0}^{t} \| u_s \|^{2}_{L^{2}} ds.
\end{equs}
\end{lemma}

\begin{proof}
We only prove the statement for $t \in (0,1]$ - the case $t=\infty$ follows similarly.
Let $\theta_{l,N}$ and $\alpha_{l,N},\beta_{l,N}, \gamma_{l,N}$ be as in Lemma \ref{lem:numerology}.
Multiplying the first inequality in \eqref{eq:conditions-interpolation} by $1/l$ we have
\begin{equation}
\frac{1}{\tilde{l}} \equiv \frac{\theta_{l,N}}{2}+\frac{1-\theta_{l,N}}{N}<\frac{1}{l}.
\end{equation}
Now, the standard\footnote{Note that this embedding is \textit{not} true if $s' = s$, since the condition $q' \geq q$ must be replaced by $q' \leq q$. However, in allowing for the strict inequality in $s$, we are allowed to trade for any $q$.}
Besov embedding $\besov{p'}{q'}{s'} \hookrightarrow \besov{p}{q}{s}$ if $s' > s, p' \geq p$, and $q' \geq q$, together with interpolation \eqref{eq:besov-interpolation} and Lemma \ref{lem:integrated-drift-small-scales-h2}, implies that
\begin{equs}
\|I_t(u)\|^{l}_{\besov{l}{l}{\tilde \kappa}} 
&\lesssim 
\|I_{t}(u)\|^{l}_{\besov{\tilde l }{\tilde l }{\tilde \kappa}}
\lesssim
\|I_{t}(u)\|^{l}_{\besov{\tilde l}{\infty}{2\tilde \kappa}}
\lesssim
\|I_t(u)\|_{\besov{N}{\infty}{0}}^{(1-\theta_{l,N}) l }\|I_t(u)\|_{\besov{2}{\infty}{2\tilde \kappa/\theta_{l,N}}}^{\theta_{l,N}l}
\\
&\lesssim
\|I_{t}(u)\|^{(1-\theta_{l,N})l}_{L^{N}} \|I_{t}(u)\|^{\theta_{l,N} l}_{H^{2\tilde \kappa/\theta_{l,N}}}
\\
&\lesssim 
t^{(1-\kappa)\theta_{l,N}l/2} \| I_{t}(u) \|_{L^{N}}^{(1-\theta_{l,N})l}\big( \int_0^t \|u_s \|_L^2 ds\big)^{\theta_{l,N}l/2}
\end{equs}
 where, for a given $\kappa$, we chose $\tilde \kappa_{l,N} \equiv \tilde \kappa = 2\theta_{l,N} \kappa$.
    
 Then, by duality \eqref{eq:besov-duality}, the iterated multiplication inequality \eqref{eq:besov-multiplication-inequality}
 and the definitions of $\alpha_{l,N}, \beta_{l,N}$ in \eqref{eq:numerology-definitions}, there exists $C>0$ such that
\begin{equs}
\Big| \int_{\Omega_\epsilon} f I_t^l(u) dx \Big|
&\leq C \| f \|_{\cC^{\tilde \kappa}} \| I_t^l(u) \|_{\besov{1}{1}{\tilde \kappa}} 
\leq  C \| f \|_{\cC^{\tilde \kappa}} \| I_t(u) \|_{\besov{l}{l}{\tilde \kappa}}^l \\
&\leq 
\bar C t^{(1-\kappa)/\alpha_{l,N}}\|f\|_{\cC^{-\tilde\kappa}} \|I_{t}(u)\|_{L^N}^{N/{\beta_{l,N}}} \big( \int_0^t \|u_s\| ds\big)^{2/\alpha_{l,N}}.
\end{equs}
Now, applying Young's inequality with $\alpha_{l,N},\beta_{l,N},\gamma_{l,N}$ and using the fact that $\gamma_{l,N}/ \alpha_{l,N}>1$, the estimate in \eqref{eq:duality-power-integrated-drift} follows.
\end{proof}

\begin{proof}[Proof of Proposition \ref{prop:minimiser-l2}]
Throughout this proof $\epsilon > 0$ is fixed and therefore dropped from the notation. The goal is to prove that for any $\delta > 0$ sufficiently small, there exists $\bar C > 0$ such that, for all $u \in \HH$,
\begin{equs}
\label{eq:bg-poly-7}
F^{\infty,E}(u)
\geq
- \bar C + \frac{1-2\delta}2 \E \Big[ \int_0^\infty \| u_s \|_{L^2}^2 \Big].
\end{equs}
This allows us to compare the cost function evaluated at a minimiser $\minimiser$ against the cost function evaluated at a competitor drift.
For our purposes it suffices to choose this competitor as the zero drift.
Indeed, since $\minimiser$ is a minimiser and since $\chi_E$ is concave and satisfies $\chi_E(0)=0$, we have
\begin{equs}
F^{\infty,E}(\minimiser)
\leq
F^{\infty,E}(0)	
= 
\E [v_0^{\epsilon,E}(Y_\infty)]
\leq
\chi_E \big( \E [ v_0(Y_\infty^\epsilon) ] \big)
\leq
0.
\end{equs}
Hence, by \eqref{eq:bg-poly-7},
\begin{equs}
\E \Big[ \int_0^\infty \| \bar u_s \|_{L^2}^2 ds \Big]
\leq
\frac{2 \bar C}{1-2\delta}.
\end{equs}
Taking the supremum over $E >0$ and $\epsilon > 0$ establishes \eqref{eq:minimiser-l2}.

It remains to establish \eqref{eq:bg-poly-7}. Fix $u \in \HH_a$.
For $\cP$ as in \eqref{eq:wick-polynomials}, by repeated use of the triangle inequality, we have that
\begin{equs} \label{eq:bg-poly-0}
\begin{split}
\Big|\int_{\Omega_\epsilon} \wick{\cP(Y_\infty + I_{\infty}(u))} dx\Big|
&\geq
-\Big|\int_{\Omega_\epsilon} \wick{\cP(Y_\infty + I_{\infty}(u))} dx
- \int_{\Omega_\epsilon} \wick{\cP(Y_\infty)} dx - \int_{\Omega_\epsilon} \cP(I_{\infty}(u)) dx \Big|
\\
&\qquad
- \Big|  \int_{\Omega_\epsilon} \wick{\cP(Y_\infty)} dx \Big| - \Big| \int_{\Omega_\epsilon} [\cP(I_{\infty}(u)) - a_N I_\infty^N(u)] dx	\Big|
\\
&\qquad 
 + a_N \|I_\infty(u)\|_{L^N}^N.
 \end{split}
\end{equs}
First, observe that there exists $c_1 > 0$, depending on the coefficients of $\cP$, such that
\begin{equs} \label{eq:bg-poly-3.5}
\Big|  \int_{\Omega_\epsilon} \wick{\cP(Y_\infty)} dx \Big| 
\leq
c_1 \sum_{k=1}^N \| \wick{Y_\infty^k}   \|_{\cC^{-\kappa}}.
\end{equs}
Second, observe that, for any $\delta > 0$, there exists $c_2=c_2(\delta) > 0$,  also depending on the coefficients of $\cP$, such that
\begin{equs}
\label{eq:bg-poly-4}
\Big|\int_{\Omega_\epsilon} \big[ \cP \big(I_{\infty}(u)\big ) - a_N I_{\infty}^N(u)\big] dx \Big|
\leq
c_2
+ \delta \int_{\Omega_\epsilon } a_N I_{\infty}^N(u) dx.
\end{equs}

Let us estimate the first term on the right-hand side of \eqref{eq:bg-poly-0}. By the triangle inequality and binomial theorem for Wick powers \eqref{eq:y-infty-wick-binomial},
\begin{equs}
\label{eq:bg-poly-1}
\begin{split}
\Big|\int_{\Omega_\epsilon} \wick{\cP(Y_\infty + I_{\infty}(u))} dx
&
- \int_{\Omega_\epsilon} \wick{\cP(Y_\infty)} dx - \int_{\Omega_\epsilon} \cP(I_{\infty}(u)) dx \Big|
\\
&\leq
 \sum_{k = 1}^{N} |a_{k}| \Big| \int_{\Omega_\epsilon}\big[ \wick{\big(Y_\infty + I_{\infty}(u) \big)^k}  - \wick{Y_\infty^k} - I_{\infty}^k(u)  \big] dx  \Big| 
\\
&\leq
\sum_{k = 1}^{N} \sum_{l = 1}^{k-1} |a_{k}| {k \choose l} \Big| \int_{\Omega_\epsilon} \wick{Y_\infty^{k-l}} I_{\infty}^{l} (u) dx \Big|.
\end{split}
\end{equs}

An application of Lemma \ref{lem:duality-power-integrated-drift} and setting $\tilde  \kappa= \min_{l\leq N} \tilde \kappa_{l,N}$ and $\tilde \kappa_{l,N}$ as in \eqref{eq:duality-power-integrated-drift} now yields
\begin{equation}
\Big| \int_{\Omega_{\epsilon}}\wick{Y^{k-l}_{\infty}} I_{\infty}^{l} dx \Big|
\leq C_{3} \|\wick{Y^{k-l}_{\infty} }\|_{\cC^{-\tilde \kappa}}^{\gamma_{l,N}}+\delta a_N \|I_{t}(u)\|^{N}_{L^{N}} +\delta \int_{0}^{\infty}\|u_{s}\|^{2}_{L^{2}}ds
\end{equation}
Set $\Gamma=\max_{l \leq N}\gamma_{l,N}$ and define
$Q =  \sum_{k=1}^N \sum_{l=1}^{k} \big(\| \wick{Y_\infty^l} \|_{\cC^{-\tilde \kappa}} + 1 \big)^\Gamma + \sum_{k=1}^N \| \wick{Y_\infty^k}   \|_{\cC^{-\tilde \kappa}}$.
Inserting this estimate into \eqref{eq:bg-poly-1} and combining it with \eqref{eq:bg-poly-3.5} and \eqref{eq:bg-poly-4}, we thus obtain the estimate:
there exists $c> 0$ depending on the mass $m$ only such that there exists $c_4,c_5 > 0$ such that
\begin{equs}
\label{eq:bg-poly-9}
\begin{split}
\int_{\Omega_\epsilon} \wick{\cP(Y_\infty + I_{\infty}(u))} dx
&\geq
-c_4 \big({ \sum_{k=1}^N }\sum_{l=1}^{k-1} (\| \wick{ Y_\infty^l } \|^{\gamma_{l,N}}_{\cC^{-\tilde \kappa}} + 1)  - \sum_{k=1}^N \| \wick{Y_\infty^k} \|_{\cC^{-\tilde \kappa}} \big)
\\
&\qquad 
+ a_N(1 - \delta) \|I_{\infty}(u)\|^{N}_{L^{N}}-\tilde \delta \|I_\infty(u)\|_{H^1}^2
 \\
&\geq -c_5 Q -\delta c \int_{0}^\infty \|u_s\|^{2}_{L^{2}} ds,
\end{split}
\end{equs}
where the last inequality follows by Lemma \ref{lem:integrated-drift-large-scales-h1}. 

Thus, by the monotonicity of the cut-off, and the fact that $ \chi_E(x) = x$ for $x \leq 0$, \eqref{eq:bg-poly-9} yields
\begin{equs} 
\label{eq:bg-poly-5}
\begin{split}
  \chi_{E}\Big( &\int_{\Omega_{\epsilon}}\wick{\cP\big(Y_{\infty}+I_{\infty}(u)\big)} dx \Big)
 &\geq - c_5 Q -\delta c \int_{0}^\infty \|u_s\|^{2}_{L^{2}} ds.
  \end{split}
\end{equs}
Hence, by Lemma \ref{lem:wick-varphi-estimate}, and the a priori estimate on the drift Lemma \ref{lem:integrated-drift-large-scales-h1}, we have that there exists $\bar C > 0$ such that,
\begin{equs} 
\label{eq:bg-poly-6}
\begin{split}
F^{\infty,E}(u)
&\geq
\E \Big[ -C Q + \frac{1-2\delta c}{2} \int_0^\infty \| u_s \|_{L^2}^2 ds  \Big]
\geq
- \bar C + \frac{1-2\delta c }2 \E \Big[ \int_0^\infty \| u_s \|_{L^2}^2 \Big],
\end{split}
\end{equs}
which, after a redefinition of $\delta$, establishes  \eqref{eq:bg-poly-7}.
\end{proof}

\begin{remark} \label{rem: 0 competitor}
It may seem surprising that the naive choice of comparing the minimiser $\overline{u}$ with the $0$ drift yields useful information, as opposed to a more carefully chosen competitor.
This is a manifestation of the mild ultraviolet divergences in dimension $2$ for these field theories.
Indeed, a trivial lower bound on the partition functions uniform in the cutoff can be obtained via Jensen's inequality and using that Wick polynomials are martingales - note that Jensen's inequality coincides \textit{exactly} with choosing the $0$ drift.
To extend these techniques to more singular field theories, such as the sine-Gordon model beyond a certain threshold, a more informed choice of competitor would need to be made. 
\end{remark}

\begin{remark}
\label{rem:lp-integrability}
The above argument also establishes uniform $L^{p}$ bounds for the density of the $\cP(\phi)_{2}$ measure with respect to the Gaussian.
Indeed, observe that we have proved that $F^{\infty,E}$ is bounded from below uniformly in $E,\epsilon$.
Recall that by \eqref{eq:boue-dupuis} we have
\begin{equs}
\inf_{u\in \HH_{a}} F^{\infty,E} (u)=
    -\log \E[\exp(-v_{0}^{\epsilon,E}(Y^{\epsilon}_{\infty}))].
\end{equs}
  So, a uniform lower bound on $F^{\infty,E}$ establishes a uniform bound on the expectation of the density of $\nu^{\cP_\epsilon,E}$. 
  The argument in the proof of Proposition \ref{prop:minimiser-l2} is also valid when $\cP$ is replaced by $p\cP$,
since Wick renormalisations are linear in the coefficients, which implies a uniform $L^{p}$ bound.
  \end{remark}

\subsection{A fractional moment estimate for small scales}

We now turn to a more refined estimate on the conditional second moment of minimisers restricted to finite time-horizons. 
Here and henceforth, we denote
\begin{equation}
F^{t,E}(u,\varphi) =  \E\Big [  v_0^{\epsilon,E}(Y_\infty^\epsilon +  I_t^\epsilon (u)+\varphi ) + \frac{1}{2} \int_0^t \| u_s\|_{L^2}^2 ds  \bigm | \cF^t \Big], \qquad \cF^t = \sigma( W_s\colon s\geq t  ) 
\end{equation}
for $t\in [0,\infty]$ and $E\in (0,\infty]$ with the convention $F^{t,E}(u,0) = F^{t,E}(u)$.

\begin{proposition}
\label{prop:cond-exp-small-scales-min}
Assume that $\minimiser$ minimises the functional $F^{t,E}(u,\varphi)$
in $\mathbb{H}_a[0,t]$ and that $\varphi\in X_\epsilon$ is $\cF^t$-measurable.
Then, as $t\to 0$ and for $\kappa$ sufficiently small,
\begin{equation}
\sup_{\epsilon>0}\sup_{E>0}\E \Big [ \int_0^t \| \minimiser_s \|_{L^2}^2 ds \bigm | \cF^t \Big ] \lesssim t^{1 -\kappa} ( 1 + \cW_t^\epsilon + \| \varphi \|_{H^1}^2 )^{L},
\end{equation}
where $\cW_t^\epsilon \geq 0$ a.s.\ for all $t\geq 0$ and  $\sup_{\epsilon>0}\sup_{t \geq 0} \E[\cW_t^\epsilon] <\infty$, and $L>0$ only depends on the maximal degree of $\cP$ and the interpolation parameters chosen in Lemma \ref{lem:numerology}.
\end{proposition}

\begin{proof}
From the definition of $F^{t,E}(u,\varphi)$ we see that we need to bound the conditional expectation of $v_0^{\epsilon,E} (Y_\infty^\epsilon + I_t^\epsilon (u) + \varphi) $.
Similarly to the proof of Proposition \ref{prop:minimiser-l2}, we compare the cost function evaluated at the minimiser against the cost function evaluated at a competitor,
which is again the zero drift because of similar reasons outlined in Remark \ref{rem: 0 competitor}.
To simplify the notation, we will drop $\epsilon$ from here on.
Since $\minimiser$ is a minimiser for $F^{t,E}(u,\varphi)$, we have $ F^{t,E}(\minimiser, \varphi) - F^{t,E}(0,\varphi) \leq 0$, and therefore
\begin{equs}
\label{eq:A-est-0}
\begin{split}
         0
          &\geq 
          F^{t,E}(\minimiser, \varphi)-F^{t,E}(0, \varphi)
          \\
         &= \E \Big[\chi_{E}\big(v_0(Y_{\infty}+I_{t}(\minimiser) +\varphi)\big)  +\frac{1}{2}\int_0^t \| \minimiser_{s}\|_{L^2}^{2}ds -\chi_{E}\big(v_0(Y_{\infty}+\varphi)\big) \Bigm | \cF^t\Big]
         \\
         &= 
              \E \Big[\chi_{E}\Big(v_0\big(Y_{\infty}+\varphi\big)  +\cA(Y_\infty,\minimiser,\varphi) + a_N \int_{\Omega_\epsilon} ( I_{t}(\minimiser)^N dx \Big)
				\\ & \qquad \qquad \qquad \qquad   \qquad + \frac{1}{2}\int_{0}^{t}\| \bar u_{s}\|^{2}_{L^2}ds
   		-\chi_{E}\Big(v_0\big(Y_{\infty}+\varphi\big)\Big) \Bigm | \cF^t\Big]
\end{split}
\end{equs}
where, for a generic drift $u \in \HH_a[0,t]$, we set
\begin{equs}
\label{eq:def-A}
\begin{split}
\cA(Y_\infty,u,\varphi)
&=
\int_{\Omega_\epsilon }\wick{\cP\big(Y_\infty + \varphi + I_{t}(u)\big)} dx - \int_{\Omega_\epsilon} \wick{\cP(Y_\infty + \varphi)}  \,  dx - a_N \int_{\Omega_\epsilon} I_t(u)^N dx
\\
&=
\sum_{k=1}^N \sum_{l=1}^{k-1} a_k {k \choose l} \int_{\Omega_\epsilon} \wick{(Y_\infty + \varphi)^{k-l}} I_{t}^{l}(u) dx + \sum_{l=1}^{N-1} a_k \int_{\Omega_{\epsilon}} I_t(u)^l dx.
\end{split}
\end{equs}
Note that, above, we tacitly use that the polynomial $\cP$ has no constant term.

The heart of the proof is to show the following estimate on $\cA$: there exists $\kappa > 0$ sufficiently small and $L>0$ such that, for any $\delta > 0$, there exists $C>0$ such that, for any $u \in \HH_a[0,t]$,

\begin{equs} \label{eq:A-est-goal}
  \begin{split}
    \cA(Y_\infty,u,\phi) &+ \delta a_N \int_{\Omega_\epsilon} I_t(u)^{N}  dx
    \\
    &\geq 
    -C t^{1-\kappa}  \big(  1+  \sum_{k=1}^N\| \wick{Y_{\infty}^j}\|_{\cC^{-\kappa}}^{2L} + \|\varphi\|_{H^{1}}^{2L}\big) -\delta \int_{0}^{t}\|u_{s}\|_{L^{2}}^{2}ds.
\end{split}
\end{equs}
Then, inserting \eqref{eq:A-est-goal} into \eqref{eq:A-est-0} and using that $\chi_E(a-b) \geq \chi_E(a) -b$ for $a \in \R$ and $b\geq 0$, we obtain
\begin{equs}
\label{eq:A-est-4}
\begin{split}
         0
          &\geq 
          \E \Big[\chi_{E}\Big(v_0\big(Y_{\infty}+\varphi\big) +\cA(Y_\infty,\minimiser,\varphi) + a_N \int_{\Omega_\epsilon}\ I_{t}(\minimiser)^N  dx \Big)
				\\ & \qquad \qquad \qquad \qquad   \qquad + \frac{1}{2}\int_{0}^{t}\| \bar u_{s}\|^{2}_{L^2}ds
   		-\chi_{E}\Big(v_0\big(Y_{\infty}+\varphi\big)\Big) \Bigm | \cF^t\Big]
        \\
         &\geq 
        \E\Big [-Ct^{1-\kappa}\Big(1 + \sum_{k=1}^N \| \wick{Y_{\infty}^k}\|_{\cC^{- \kappa}}^{2L} +\|\varphi\|_{H^{1}}^{2L} \Big) + \big(\frac{1}{2}-\delta\big)\int_{0}^{t}\|\minimiser_{s}\|^{2}_{L^2} ds \Bigm | \cF^t \Big]
        \\
        &\geq 
        -Ct^{1- \kappa} \Big( 1 + \cW_t + \|\varphi\|^2_{H^1} \Big)^{L} + \E \Big[ \big(\frac{1}{2}-\delta\big) \int_{0}^{t} \|\minimiser_{s}\|^{2}_{L^2} ds \Bigm | \cF^t\Big],
\end{split}
\end{equs}
where 
\begin{equs}
  \cW_t
 =
  \E \Big[ \sum_{k=1}^N \| \wick{Y_{\infty}^k} \|_{\cC^{-\bar \kappa}}^{2L}  \bigm | \cF^t \Big]^\frac 1L.
\end{equs}        
Finally, observe that by Jensen's inequality, the tower property of conditional expectation, and Lemma \ref{lem:wick-varphi-estimate}, we have
\begin{equs}
\begin{split}
\sup_{\epsilon > 0} \sup_{t \geq 0} \E[\cW_t]
&=
\sup_{\epsilon > 0} \sup_{t \geq 0} \E \Big[ \Big(\E \Big[ \sum_{k=1}^N \| \wick{Y_{\infty}^k} \|_{\cC^{-\bar \kappa}}^{2L}  \bigm | \cF^t \big] \Big)^\frac 1L \Big]
\\
&\leq
\sup_{\epsilon > 0} \sup_{t \geq 0} \Big( \sum_{k=1}^N \E \big[ \| \wick{Y_{\infty}^k}\|_{\cC^{-\bar \kappa}}^{2L} \big] \Big)^\frac 1L
<
\infty.
\end{split}
\end{equs}
Thus, rearranging \eqref{eq:A-est-4} completes the proof conditional on \eqref{eq:A-est-goal}.

We now focus on \eqref{eq:A-est-goal}. 
Applying Lemma \ref{lem:duality-power-integrated-drift} with  $f=\wick{(Y_{\infty}+\varphi)^{k-l}}$ and $f=1$, we know that there exists $\kappa, \tilde \kappa > 0$ sufficiently small and $L>0$ such that, for any $\delta > 0$, there exists $C>0$ such that, for any $u \in \HH_a[0,t]$,
\begin{equs} 
\label{eq:A-est-2}
\Big| \int_{\Omega_\epsilon} \wick{(Y_\infty &+ \varphi)^{k-l}} I_{t}^{l} (u)dx \Big|
\\
&\leq
C t^{1-\kappa}   \| \wick{(Y_\infty + \varphi)^{k-l}} \|_{\cC^{-\tilde\kappa}}^{\gamma_{l,N}} + \delta \Big(a_N \int_{\Omega_\epsilon} I_t(u)^{N} dx  + \frac{1}{2} \int_0^t \| u_s \|_{L^2}^2 ds \Big)
\end{equs}
and
\begin{equs}
  \label{bound:I}
  \int I_t(u)^l dx \leq C t^{1-\kappa}+ \delta a_N\|I_t(u)\|^N_{L^N} + \delta \int_0^t \|u_s\|_L^2 ds.
  \end{equs}

In addition, by Lemma \ref{lem:wick-varphi-estimate}, there exists $\bar \kappa>0$ sufficiently small, such that 
 \begin{equs} 
 \label{eq:A-est-3}
 \|\wick{(Y_{\infty}+\varphi)^{k-l}} \|_{\cC^{-\tilde\kappa}}
 \lesssim 
 1 + \sum^{k-l}_{j=1} \|\wick{Y_{\infty}^{k-l-j} }\|_{\cC^{-\bar \kappa}}^{ (k-l)/(k-l-j)}+ \|\varphi\|_{H^{1}}^{k-l}.
 \end{equs}
Therefore, inserting \eqref{eq:A-est-3} into \eqref{eq:A-est-2}, and summing this with \eqref{bound:I}, there exists $L>0$ such that
  \begin{equs}
  \begin{split}
    &\cA(Y_\infty,u,\phi) + \delta a_N \|I_{t} (u)\|_{L^{N}}^{N}
    \\
    &\geq
    -C t^{1-\kappa}  \big(  1+  \sum_{k=1}^N \sum_{l=1}^{k-1} \sum_{j=1}^{k-l} \big( 1 + \|\wick{Y_{\infty}^{k-l-j} }\|_{\cC^{-\bar \kappa}}^{ (k-l)/(k-l-j)}+ \|\varphi\|_{H^{1}}^{k-l} \big)^{\gamma_{l,N}}
    -\delta \int_{0}^{t}\|u_{s}\|_{L^{2}}^{2}ds
    \\
    & 
    \geq 
    -C t^{1-\kappa}  \big(  1+  \sum_{j=1}^N\| \wick{Y_{\infty}^j}\|_{\cC^{-\bar \kappa}}^{2L} + \|\varphi\|_{H^{1}}^{2L}\big) -\delta \int_{0}^{t}\|u_{s}\|_{L^{2}}^{2}ds
    \end{split}
  \end{equs}
  thereby establishing \eqref{eq:A-est-goal} up to a redefinition of $\kappa$.
\end{proof}

We now prove the main estimate of this section.
\begin{proposition}
\label{prop:minimiser-small-scales}
Let $\minimiserE_t=-q_t^\epsilon \nabla v_t^{\epsilon,E}(\Phi_t^{\Delta_\epsilon,E})$ and let $L$ be as in Proposition \ref{prop:minimiser-small-scales}.
Then for any $ \kappa >0$ sufficiently small, we have
\begin{equation}
\sup_{E>0}  \sup_{\epsilon>0} \sup_{t\leq 1} \E\Big[ \big( t^{ -1 + \kappa} \int_0^t \| \minimiserE_s \|_{L^2}^2 ds\big)^{1/L} \Big] < \infty.
\end{equation}
\end{proposition}

\begin{proof}
From Proposition \ref{prop:bd-polchinski-correspondence} we know that $u^E$ is a minimiser of the Bou\'e-Dupuis variational formula \eqref{eq:boue-dupuis}.
Since $\Phi_t^{\cP_\epsilon,E} = \Phi_t^{\Delta_\epsilon,E} + \Phi_t^{\GFF_\epsilon}$ and $Y_t^\epsilon+ \Phi_t^{\GFF_\epsilon} = Y_\infty^\epsilon$, we also have that $u^E$ minimises
\begin{equation}
F^{t,E}(u, I_{t,\infty}^\epsilon(u^E))
= \E \Big[ v_0^{\epsilon,E}\big(Y_\infty^\epsilon + I_t^\epsilon(u) + I_{t,\infty}^\epsilon(\minimiserE) \big) + \frac{1}{2}\int_0^t \|u\|_{L^2}^2 \bigm | \cF^t \Big]
\end{equation}
in $\HH_a[0,t]$, where $I_{t,\infty}^\epsilon(\minimiserE) = \Phi_t^{\Delta_\epsilon,E}$ is $\cF^t$-measurable. Hence, by Proposition \ref{prop:cond-exp-small-scales-min}  we have for any $\kappa$ sufficiently small
\begin{equation}
\E \Big[ \int_0^t \|\minimiserE_s \|_{L^2}^2 \bigm | \cF^t \Big] 
\lesssim
 t^{1-\kappa} \big(  1+   \cW_t^\epsilon + \|I_{t,\infty}^\epsilon(\minimiserE)\|_{H^1}^2 \big)^L.
\end{equation}
Thus, for every such $\kappa$
\begin{align}
\E\big[ \big(   t^{-1+\kappa } \int_0^t \| \minimiserE_s\|_{L^2}^2 ds \big)^{1/L}   \big] &= \E\Big[ \E\big[ \big(   t^{-1+\kappa} \int_0^t \| \minimiserE\|_{L^2}^2 ds \big)^{1/L} \bigm | \cF^t \big]   \Big] \nnb
&\leq \E \Big[ \E\big[ t^{-1+\kappa } \int_0^t \| \minimiserE_s \|_{L^2}^2 \bigm |\cF^t\big]^{1/L}\Big] \leq \E\Big[ 1+  \cW_t^\epsilon + \|I_{t,\infty}^\epsilon(\minimiserE) \|_{H^1}^2  \Big] \leq C
\end{align}
for some universal constant $C>0$. 
Taking the supremum over $t\leq 1$ and $\epsilon>0$, the claim follows.
\end{proof}

\section{Proof of the coupling to the GFF}

This section is devoted to the proof of the main results Theorem \ref{thm:coupling-pphi-to-gff-eps} and Corollary \ref{cor:pphi-coupling-continuum}.
Recall that we introduced the notion $\minimiserE_t = -q_t^\epsilon \nabla v_t^{\epsilon,E} (\Phi_t^{\cP_\epsilon,E})$.
By Proposition \ref{prop:bd-polchinski-correspondence} we know that this is a minimiser of the variational problem \eqref{eq:boue-dupuis}.
Hence, it  satisfies the estimate in Proposition \ref{prop:minimiser-l2} and also the estimates in Proposition \ref{prop:minimiser-small-scales}.

\subsection{Estimates on the difference field $\Phi^{\Delta_\epsilon,E}$}

First, we show how the bound for the minimiser $\minimiserE$ in Proposition \ref{prop:minimiser-l2}  implies bounds on the fractional moments of Sobolev norms of the difference field $\Phi_t^{\Delta_\epsilon,E} = \int_t^\infty q_\tau^\epsilon \minimiserE_\tau d\tau$ in \eqref{eq:SDE-polchinski-cut-off}.
The main result is the following proposition.

\begin{proposition}
\label{prop:Sobolev-norm-phi-delta}
Let $\alpha \in (0,2)$. Then
\begin{equation}
\label{eq:Sobolev-norm-phi-delta}
\sup_{E>0} \sup_{\epsilon>0} \sup_{t\geq 0}\E \big[ \| \Phi_0^{\Delta_\epsilon,E} \|^{2/L}_{H^{\alpha}} \big ] < \infty.
\end{equation}
\end{proposition}

The first step towards \eqref{eq:Sobolev-norm-phi-delta} is the following estimate for  the large scale of $\Phi^{\Delta_\epsilon,E}$.

\begin{lemma}
\label{lem:phi-delta-large-scales}
Let $t_0>0$ and $\alpha \in [1,2)$. Then, for $t\geq t_0$, 
\begin{equation}
\label{eq:sobolev-difference-field-expectation}
\sup_{E>0} \sup_{\epsilon >0} \E[\|\Phi_t^{\Delta,E} \|_{H^{\alpha}} ] \lesssim_{t_0, \alpha} \frac{1}{t^{(\alpha-1)/2}}.
\end{equation}
\end{lemma}

\begin{proof}
Since $\Phi_t^{\Delta_\epsilon,E} = I_{t,\infty}^\epsilon(u^E)$ and since $\E[\int_0^\infty \|\minimiserE_\tau\|_{L^2}^2 d\tau] <\infty$ uniformly in $\epsilon>0$ and $E>0$, the claim follows from Lemma \ref{lem:integrated-drift-large-scales-h2} .
\end{proof}

\begin{lemma}
\label{lem:phi-delta-small-scales}
 Set $\Phi_{s,t}^{\Delta_\epsilon,E} = \int_s^t \dot c_\tau^\epsilon \nabla v_\tau^{\epsilon,E}(\Phi_\tau^{\cP_\epsilon,E}) d\tau$. Then for $\alpha \in (1,2)$
 \begin{equation}
\sup_{E>0}\sup_{\epsilon >0} \E  \Big[ \| \Phi_{s,t}^{\Delta_\epsilon,E}  \|_{H^{\alpha}}^{2/L} \Big] \lesssim \E\Big[ \big( \frac{t-s}{s^\alpha} \int_s^t \| \minimiserE_\tau\|_{L^2}^2 d\tau  \big)^{1/L}\Big].
 \end{equation}
\end{lemma}

\begin{proof}
Again, the proof is an application of the regularity estimates at the end of Section \ref{ssec:boue-dupuis}. 
Noting that $\Phi_{s,t}^{\Delta_\epsilon,E} = I_{s,t}^\epsilon(\minimiserE)$, by  Lemma \ref{lem:integrated-drift-small-scales-h2} that
\begin{equation}
\| \Phi_{s,t}^{\Delta_\epsilon,E}  \|_{H^\alpha}^{2/L} \lesssim \Big(\frac{t-s}{s^\alpha} \int_s^t \|\minimiserE_\tau\|_{L^2}^2 d\tau\Big)^{1/L},
\end{equation}
from which the claim follows by taking expectation. 
\end{proof}

Combining Lemma \ref{lem:phi-delta-large-scales} and Lemma \ref{lem:phi-delta-small-scales} we can now give the proof of Proposition \ref{prop:Sobolev-norm-phi-delta}.

\begin{proof}[Proof of Proposition \ref{prop:Sobolev-norm-phi-delta}]
Let $\alpha \in (0,2)$ and let $\kappa = (2-\alpha)/r$ for $r$ large enough.
Further, let $(t_n)_n$ be a decreasing sequence of numbers with $t_n\to 0$ as $n\to \infty$ that will be determined later.
Then, using Lemma \ref{lem:phi-delta-small-scales} and Proposition \ref{prop:minimiser-small-scales},
\begin{align}
\E[\| \Phi_0^{\Delta_\epsilon,E} \|_{H^{\alpha}}^{2/L} ] 
&\leq \sum_{n \geq 1} \E \| \Phi_{t_{n+1},t_n}^{\Delta_\epsilon,E} \|_{H^{\alpha}}^{2/L}  + \E \| \Phi_{t_0}^{\Delta_\epsilon,E} \|_{H^{\alpha}}^{2/L} \nnb
& \leq \sum_{n\geq 1} \E\Big[\Big( \frac{t_n-t_{n+1}}{t_{n+1}^\alpha}\int_{t_{n+1}}^{t_{n}} \|\minimiserE_\tau \|_{L^2}^2 d\tau \Big)^{1/L}\Big]  + C  \nnb
&\leq \sum_{n\geq 1} \Big( \frac{t_n-t_{n+1}}{t_{n+1}^\alpha} t_n^{1-\kappa}\Big)^{1/L} \E\Big[\Big( t_n^{-1+\kappa} \int_0^{t_n} \|\minimiserE_\tau \|_{L^2}^2 d\tau \Big)^{1/L}\Big] + C \nnb
&\lesssim \sum_{n \geq 1} \Big( \frac{t_n-t_{n+1}}{t_{n+1}^\alpha} t_n^{1-\kappa}\Big)^{1/L}  + C.
\end{align}
Now, choosing $t_n=2^{-n}$ we see that
\begin{equation}
\frac{t_n-t_{n+1}}{t_{n+1}^\alpha} t_n^{1-\kappa} = \frac{2^{-(n+1)}}{2^{-\alpha(n+1)}} 2^{-n(1-\kappa)} = 2^{-(2-\alpha-\kappa)n + \alpha} 
\end{equation}
and thus, the sum in the last display is finite for the specified values of $\alpha$ and $\kappa$.
\end{proof}

\subsection{Removal of the cut-off: proof of Theorem \ref{thm:coupling-pphi-to-gff-eps}}
\label{ssec:removal-cut-off}

We have collected all results to give the proof of Theorem \ref{thm:coupling-pphi-to-gff-eps} and Corollary \ref{cor:pphi-coupling-continuum}.
The main task in this section is to remove the cut-off by taking the limit $E\to \infty$.

\begin{proof}[Proof of Theorem \ref{thm:coupling-pphi-to-gff-eps}]
We first prove that the sequence of processes $(\Phi^{\Delta_\epsilon,E})_E$ is tight. 
For $R>0$ and $\alpha<1$, let
\begin{align}
\cX_R
=
\Big \{ \Phi \in C_0([0,\infty),X_\epsilon) \colon \sup_{t \in [0,\infty)} \|\Phi\|_{H^\alpha}^2 \leq R \text{ and } \sup_{s< t} \frac{\|\Phi_t - \Phi_s\|_{H^\alpha}^2}{(t-s)^{1-\alpha}} \leq R \Big\}.	
\end{align}
Note that $\cX_R$ is totally bounded and equicontinuous.
Therefore, by the Arz\`ela-Ascoli theorem, $\cX_R \subset C_0([0,\infty),X_\epsilon)$ is precompact, and the closure $\overline{\cX_R}$ is compact.
Moreover, we have by Lemma \ref{lem:integrated-drift-large-scales-h1} and Lemma \ref{lem:integrated-drift-small-scales-h1} that 
\begin{align}
\sup_{t\geq 0}\|\Phi_t^{\Delta_\epsilon,E} \|_{H^\alpha}^2 &\lesssim \int_0^\infty \|\minimiserE_\tau\|_{L^2}^2 d\tau   \\
 \|\Phi_t^{\Delta_\epsilon,E} - \Phi_s^{\Delta_\epsilon,E}\|_{H^\alpha}^2 &\lesssim (t-s)^{1-\alpha} \int_{0}^\infty \| \minimiserE_\tau\|_{L^2}^2 d\tau. 
\end{align}
Therefore, we have for some constant $C>0$, which is independent of $\epsilon$ and $E$,
\begin{align}
\P ( \Phi^{\Delta_\epsilon,E}	\in \overline{\cX_R}^c ) &\leq 
\P ( \Phi^{\Delta_\epsilon,E}	\in \cX_R^c )
\leq 
\P \Big( \sup_{t \in [0,\infty)} \|\Phi_t^{\Delta_\epsilon,E}\|_{H^\alpha}^2 +
 \sup_{s< t} \frac{\|\Phi_t^{\Delta_\epsilon,E} - \Phi_s^{\Delta_\epsilon,E} \|_{H^\alpha}^2}{(t-s)^{1-\alpha}} > R \Big )
\nnb
&\leq \P \Big( \int_0^\infty \|\minimiserE_\tau\|_{L^2}^2 d\tau > R/C \Big) \leq
\frac{C}{R} \E \Big[ \int_0^\infty \|\minimiserE_\tau\|_{L^2}^2 d\tau \Big].
\end{align}
So, for a given $\kappa>0$, we can choose $R$ large enough such that
\begin{equation}
\sup_{E>0} \P \big( \Phi^{\Delta_\epsilon,E}	\in \overline{\cX_R}^c \big) \leq \frac{C}{R} \sup_{E>0} \E \Big[ \int_0^\infty \|\minimiserE_\tau\|_{L^2}^2 d\tau \Big] <\kappa,
\end{equation}
which establishes tightness for the sequence $(\Phi^{\Delta_\epsilon,E})_E \subseteq C_0([0,\infty), X_\epsilon)$.

By Prohorov's theorem there is a process  $ \Phi^{\Delta_\epsilon}$ and a subsequence $(E_k)_k$ such that $\Phi^{\Delta,E_k} \to \Phi^{\Delta_\epsilon}$ in distribution as $k\to \infty$.
By \eqref{eq:coupling-cut-off-fields} there exists a process $\Phi^{\cP_\epsilon} \equiv \Phi^{\Delta_\epsilon} + \Phi^{\GFF_\epsilon}$,
such that $\Phi^{\cP_\epsilon,E_k}\to \Phi^{\cP_\epsilon}$ in distribution as $k\to \infty$. 

Since $e^{-v_0^{\epsilon,E}(\phi)} \to e^{-v_0^\epsilon(\phi)}$ as $E\to \infty$ for every $\phi \in X_\epsilon$ and since $e^{-v_0^{\epsilon,E}} \leq e^{C_\epsilon}$ for some constant $C_\epsilon>0$, we have by dominated convergence for every bounded and continuous $f\colon X_\epsilon \to \R$
\begin{equation}
\label{eq:weak-convergence-nu-P}
\E_{\nu^{\cP_\epsilon,E}}[f] \to \E_{\nu^{\cP_\epsilon}}[f],
\end{equation}
so that $\nu^{\cP_\epsilon,E}\to \nu^{\cP_\epsilon}$ in distribution. 
Since $\Phi_0^{\cP_\epsilon,E} \sim \nu^{\cP_\epsilon,E}$ we conclude that $\Phi_0^{\cP_\epsilon} \sim \nu^{\cP_\epsilon}$ by uniqueness of weak limits.
Moreover, since independence is preserved under weak limits, we have that for every $t>0$, $\Phi_t^{\cP_\epsilon}$ is independent of $\Phi_0^{\GFF_\epsilon} - \Phi_t^{\GFF_\epsilon}$.

Finally, the bounds (\ref{eq:phi-delta-h1}--\ref{eq:phi-delta-to-0}) follow from the respective estimates on $\Phi^{\Delta_\epsilon,E}$ and the fact that the norms are continuous maps from $C_0([0,\infty), X_\epsilon)$ to $\R$.
For instance, \eqref{eq:phi-delta-1-2} is proved by
\begin{align}
\sup_{\epsilon>0} \sup_{t\geq 0}  \E\Big[ \| \Phi_t^{\Delta_\epsilon} \|_{H^\alpha}^{2/L} \Big]&=
\sup_{\epsilon>0} \sup_{t\geq 0} \lim_{C\to \infty} \E \Big[ \| \Phi_t^{\Delta_\epsilon} \|_{H^\alpha}^{2/L} \wedge C \Big]
= \sup_{\epsilon>0} \sup_{t\geq 0}  \lim_{C\to \infty} \lim_{E\to \infty} \E \Big[ \| \Phi_t^{\Delta_\epsilon,E} \|_{H^\alpha}^{2/L} \wedge C \Big] \nnb
&\leq  \sup_{\epsilon>0} \sup_{t \geq 0} \sup_{E >0} \E \Big[ \| \Phi_t^{\Delta_\epsilon,E} \|_{H^\alpha}^{2/L} \Big] <\infty,
\end{align}
where the last display is finite by Proposition \ref{prop:Sobolev-norm-phi-delta}.
\end{proof}

\subsection{Lattice convergence: proof of Corollary \ref{cor:pphi-coupling-continuum}}
\label{ssec:lattice-convergence}

In this section we prove that as $\epsilon \to 0$ the processes $(\Phi^{\Delta_\epsilon})_\epsilon$ converge along a subsequence $(\epsilon_k)_k$ to a continuum process $\Phi^{\Delta_0}$.
In order to obtain a continuum process $\Phi^{\cP_0}$ from the sequence  $(\Phi^{\cP_\epsilon})_\epsilon$,
we also need the convergence of the decomposed Gaussian free field $\Phi^{\GFF_\epsilon}$,
which is the content of Lemma \ref{lem:PhiGFF-limit-pauli-villars} below.
Define for $t\geq0$
\begin{equation} 
\label{eq:GFF-fourier-pauli-villars}
  \Phi_t^{\GFF_0} =
  \int_t^\infty q_s^0 dW_s= 
  \sum_{k\in \Omega^*}  e^{ik\cdot(\cdot)}\int_t^\infty \hat q_u^0(k) d\hat W_u(k), \quad \hat q_u^0(k) = \frac{1}{t(-|k|^2 +m^2) + 1}.
\end{equation}
Note that for $t>0$ the convergence of the sum is understood in $H^\alpha$, $\alpha <1$, while for $t=0$ it is in $H^\alpha$ for $\alpha<0$. 

\begin{lemma}
\label{lem:PhiGFF-limit-pauli-villars}
Let $I_\epsilon\colon L^2(\Omega_\epsilon) \to L^2(\Omega)$ be the isometric embedding  defined in Section \ref{ssec:fourier}.
Then, for any $t_0 > 0$ and $\alpha < 1$, we have
\begin{align}
\label{eq:gff-limit-eps-to-positive}
\E\qa{ \sup_{t\geq t_0}\norm{I_\epsilon\Phi^{\GFF_\epsilon}_t -\Phi^{\GFF_{0}}_t}_{H^\alpha}^2} &\to 0.
\end{align}
Moreover, if $t_0=0$ we have for $\alpha<0$
\begin{align}
\label{eq:gff-limit-eps-to-0}
    \E\qa{ \sup_{t\geq 0}\| I_\epsilon \Phi_t^{\GFF_\epsilon}-\Phi_t^{\GFF_{0}}\|_{H^{\alpha}}^2 } \to 0.
  \end{align}
\end{lemma}

\begin{proof}
Note that the process $I_\epsilon \Phi^{\GFF_\epsilon} - \Phi^{\GFF_0}$ is a backward martingale adapted to the filtration $\cF^t$ with values in $H^\alpha$ with $\alpha<1$ for $t_0>0$ and $\alpha<0$ for $t_0=0$.
Thus, the $H^\alpha$ norm of this process is a real valued submartingale. Therefore, we have by Doob's $L^2$ inequality, 
\begin{align}
\E\qa{ \sup_{t\geq t_0}\norm{I_\epsilon\Phi^{\GFF_\epsilon}_t -\Phi^{\GFF_{0}}_t}_{H^\alpha}^2} \lesssim \E\qa{ \norm{I_\epsilon\Phi^{\GFF_\epsilon}_{t_0} -\Phi^{\GFF_{0}}_{t_0}}_{H^\alpha}^2}, 
\end{align}
so it suffices to prove that the right-hand side converges to $0$ for the specified values of $t_0$ and $\alpha$.

To this end, we consider the difference of the Fourier coefficients $\hat q^\epsilon_t(k)$ and $\hat q_t^0(k)$.
By \eqref{eq:fourier-multipliers-laplacians-difference}, we have that
\begin{equs}
\begin{split}
\label{eq:fourier-multipliers-q-difference}
0\leq \hat q_t^\epsilon(k) - \hat q_t^0(k) &= \frac{1}{t(-\hat \Delta^\epsilon(k) + m^2) + 1} - \frac{1}{t(-\hat \Delta^0(k) + m^2) + 1} 
\\
&\leq \frac{t |k|^2 h(\epsilon k)}{ \big[t(c |k|^2+ m^2) + 1\big] \big[ t(|k|^2 + m^2) + 1\big] }
\leq \frac{h(\epsilon k)}{t(c |k|^2 + m^2) + 1},
\end{split}
\end{equs}
where we recall that $c =4/\pi^2$ as below \eqref{eq:fourier-multipliers-laplacians-difference}. By the definition of the Sobolev norm, we have
\begin{multline}
\| I_\epsilon \Phi_{t_0}^{\GFF_\epsilon} -\Phi_{t_0}^{\GFF_{0}}\|_{H^{\alpha}(\Omega)}^2 
= \sum_{k\in \Omega_\epsilon^*} (1+ |k|^2)^{\alpha} \Big|\int_{t_0}^\infty (\hat q_u^\epsilon(k) - \hat q_u^0(k)) d\hat W_u(k)   \Big|^2 
\\
+ 
\sum_{k \in \Omega^*\setminus \Omega_\epsilon^*}  (1+ |k|^2)^{\alpha} \Big|\int_{t_0}^\infty \hat q_u^0(k) d\hat W_u(k)   \Big|^2.
\end{multline}
Taking expectation and using the estimate \eqref{eq:fourier-multipliers-q-difference} for the first sum yields
\begin{equs}
\label{eq:expectation-sobolev-norm}
\begin{split}
\E \Big[ \| I_\epsilon \Phi_{t_0}^{\GFF_\epsilon} -\Phi_{t_0}^{\GFF_{0}}\|_{H^{\alpha}(\Omega)}^2 \Big]
&\leq \sum_{k\in \Omega_\epsilon^*} (1+ |k|^2)^{\alpha} \int_{t_0}^\infty \frac{h^2(\epsilon k)}{\big(u(c |k|^2 + m^2) + 1\big)^2} du
\nnb
&+ 
\sum_{k \in \Omega^*\setminus \Omega_\epsilon^*}  (1+ |k|^2)^{\alpha} \int_{t_0}^\infty  \frac{1}{\big(u(|k|^2+ m^2) + 1\big)^2}  du 
\nnb
&\leq \sum_{k\in \Omega_\epsilon^*} (1+ |k|^2)^{\alpha} \frac{h^2(\epsilon k)}{c |k|^2 + m^2} \frac{1}{t_0(c |k|^2 + m^2)  + 1}
\nnb
&+ 
\sum_{k \in \Omega^*\setminus \Omega_\epsilon^*} (1+ |k|^2)^{\alpha} \frac{1}{|k|^2 + m^2} \frac{1}{t_0( |k|^2 + m^2)  + 1}.
\end{split}
\end{equs}
Further note that  $h^2(\epsilon k) \leq O(\epsilon |k|)^\delta$ for any $\delta <4$. 
We can now discuss the convergence to $0$ as $\epsilon\to 0$ for $\alpha$ and $t_0$ as above. For $t_0>0$ the first sum on the right-hand side can be bounded by
\begin{equs}
\label{eq:gff-convergence-expectation-first-sum}
\begin{split}
&\sum_{k\in \Omega_\epsilon^*} (1+ |k|^2)^{\alpha} \frac{h(\epsilon k)}{c |k|^2 + m^2} \frac{1}{t_0(c |k|^2 + m^2)  + 1} \nnb
& \lesssim \epsilon^\delta \sum_{k\in \Omega_\epsilon^*} (1+ |k|^2)^{\alpha} \frac{|k|^\delta}{c |k|^2 + m^2} \frac{1}{t_0(c |k|^2 + m^2)  + 1} 
\lesssim_{t_0} \epsilon^\delta \sum_{k\in \Omega^*} \frac{1}{\big(1+|k|^2\big)^{2-\alpha-\delta}}.
\end{split}
\end{equs}
The last sum is finite for $\alpha<1$ and $\delta$ small enough (depending on $\alpha$), and hence vanishes as $\epsilon \to 0$.
For the second sum on the right-hand side of \eqref{eq:expectation-sobolev-norm} we similarly have
\begin{align}
 \sum_{k \in \Omega^*\setminus \Omega_\epsilon^*}  (1+ |k|^2)^{\alpha}  \frac{1}{|k|^2+m^2} \frac{1}{t_0 (|k|^2+m^2)+1} 
\lesssim_{t_0} \sum_{k \in \Omega^*\setminus \Omega_\epsilon^*}  \frac{1}{\big ( 1+|k|^2\big)^{2-\alpha} },
\end{align}
which is finite uniformly in $\epsilon$ for $\alpha<1$, and hence converges to $0$ as $\epsilon \to 0$.
Together with \eqref{eq:gff-convergence-expectation-first-sum} this shows the convergence in \eqref{eq:gff-limit-eps-to-positive}.

For the proof of \eqref{eq:gff-limit-eps-to-0}, we note that both sums on the right-hand side lose a term of order $O(|k|^{-2})$ when $t_0=0$.
Using the same arguments as for the case $t_0>0$, we find that the corresponding sums are finite for $\alpha<0$ in this case.
\end{proof}

The next result is the convergence of the difference field $\Phi^{\Delta_\epsilon}$ along a subsequence $(\epsilon_k)_k$, and its proof is almost identical to the removal of the cut-off in the proof of Theorem \ref{thm:coupling-pphi-to-gff-eps}.
Recall that we denote by $C_0([0,\infty), \cS)$ the set of all continuous processes with values in a metric space $(\cS, \|\cdot\|_{\cS})$ that vanish at $\infty$
and that $I_\epsilon\colon L^2(\Omega_\epsilon) \to L^2(\Omega)$ is the isometric embedding.

\begin{proposition}
Let $\alpha<1$. Then $(I_\epsilon \Phi^{\Delta_\epsilon})_\epsilon$ is a tight sequence of processes in $C_0([0,\infty), H^{\alpha})$.
In particular, there is a process $\Phi^{\Delta_0} \in C_0([0,\infty), H^{\alpha})$ and a subsequence $(\epsilon_k)_k$, $\epsilon_k \to 0$ as $k\to\infty$
such that the laws of $\Phi^{\Delta_{\epsilon_k}}$ on $C_0([0,\infty), H^{\alpha})$ converge weakly to the law of $\Phi^{\Delta_0}$.
\end{proposition}

\begin{proof}
For $R>0$ and $\alpha<1$, let
\begin{align}
\cX_R
=
\Big \{ \Phi \in C_0([0,\infty),H^\alpha) \colon \sup_{t \in [0,\infty)} \|\Phi\|_{H^\alpha}^2 \leq R \text{ and } \sup_{s< t} \frac{\|\Phi_t - \Phi_s\|_{H^\alpha}^2}{(t-s)^{1-\alpha}} \leq R \Big\}.	
\end{align}
As in the proof of Theorem \ref{thm:coupling-pphi-to-gff-eps} the closure $\overline{\cX_R}$ is compact by the Arz\`ela-Ascoli theorem.
Moreover, we have
\begin{align}
\sup_{t\geq 0}\|I_\epsilon \Phi_t^{\Delta_\epsilon,E} \|_{H^\alpha}^2 & \lesssim \int_0^\infty \|\minimiserE_\tau\|_{L^2}^2 d\tau , \\
 \|I_\epsilon \Phi_t^{\Delta_\epsilon,E} - I_\epsilon \Phi_s^{\Delta_\epsilon,E} \|_{H^\alpha}^2 &\lesssim (t-s)^{1-\alpha} \int_{0}^\infty \| \minimiserE_\tau\|_{L^2}^2 d\tau ,
\end{align}
and thus, by the weak convergence of $(\Phi^{\Delta_\epsilon,E})_E$ as $E\to \infty$ (along a subsequence $(E_k)_k$),
we have for some constant $C>0$, which is independent of $\epsilon$ and $E$,
\begin{align}
\P (I_\epsilon \Phi^{\Delta_\epsilon}	\in \overline{\cX_R}^c ) 
&\leq \liminf_{k\to \infty} \P(I_\epsilon \Phi^{\Delta_\epsilon,E_k}	\in \overline{\cX_R}^c ) \leq 
\liminf_{k\to \infty} \P ( I_\epsilon \Phi^{\Delta_\epsilon,E_k}	\in \cX_R^c )
 \nnb
&\leq \liminf_{k\to \infty} \P \Big(  \int_0^\infty \|u^{E_k}_\tau\|_{L^2}^2 d\tau > R/C \Big) 
\leq \sup_{E>0}\P \Big( \int_0^\infty \|\minimiserE_\tau\|_{L^2}^2 d\tau > R/C \Big)  \nnb
& \leq \sup_{E>0}\frac{C}{R} \E \Big[ \int_0^\infty \|\minimiserE_\tau\|_{L^2}^2 d\tau \Big].
\end{align}
So, for a given $\kappa>0$, we can choose $R$ large enough such that
\begin{equation}
\sup_{\epsilon>0} \P ( I_\epsilon \Phi^{\Delta_\epsilon}	\in \overline{\cX_R}^c ) \leq \frac{2}{R} \sup_{\epsilon>0} \sup_{E>0} \E \Big[ \int_0^\infty \|\minimiserE_\tau\|_{L^2}^2 d\tau \Big] <\kappa,
\end{equation}
which establishes tightness for the sequence $(I_\epsilon\Phi^{\Delta_\epsilon})_\epsilon \subseteq C_0([0,\infty], H^{\alpha})$. The existence of a weak limit $\Phi^{\Delta_0}$ then follows by Prohorov's theorem.
\end{proof}

\begin{proof}[Proof of Corollary \ref{cor:pphi-coupling-continuum}]
Since $\Phi^{\Delta_{\epsilon_k}} \to \Phi^{\Delta_{0}}$ in distribution as $k\to \infty$, we also have that there exists a process $\Phi_t^{\cP_0} \equiv \Phi^{\Delta_0} + \Phi^{\GFF_0}$,
such that $\Phi^{\cP_{\epsilon_k}} \to \Phi^{\cP_{0}}$ in distribution as $k\to \infty$. 
Moreover, as $\epsilon\to 0$, we have that $\nu^{\cP_\epsilon} \to \nu^{\cP}$, where $\nu^\cP$ is the continuum $\cP(\phi)_2$ measure, see also Proposition \ref{prop:weak-convergence-nu-t} below.

Finally, the estimates on the norms of $ \Phi^{\Delta_0}$ and the independence of $\Phi_t^{\cP_0}$ and $\Phi_0^{\GFF_0} - \Phi_t^{\GFF_0}$ follow from the convergence in distribution similarly as in the proof of Theorem \ref{thm:coupling-pphi-to-gff-eps}.
\end{proof}

\subsection{Uniqueness of the limiting law for fixed $t > 0$}

For the discussion of the maximum, we need a refined statement on the convergence of $(\Phi_t^{\cP_\epsilon})_\epsilon$ for $t>0$ as $\epsilon \to 0$.
More precisely, we prove that the law of the limiting field $\Phi_t^{\cP_{0}}$ for a fixed $t>0$ does not depend on the subsequence.
By the same arguments that led to \eqref{eq:weak-convergence-nu-P}, we have that $\Phi_t^{\cP_\epsilon}$ is distributed as the renormalised measure $\nu_t^{\cP_\epsilon}$ defined by
\begin{equs}
\label{eq:ren-measure}
\E_{\nu_t^{\cP_\epsilon}} [F] 
= 
e^{v_\infty^\epsilon(0)} \EE_{c_\infty^\epsilon -c_t^\epsilon} \big[ F(\zeta) e^{-v_t^\epsilon(\zeta)} \big],
\end{equs}
where $F\colon X_\epsilon \to \R$, and $v_t^\epsilon$ is the renormalised potential defined in \eqref{eq:pphi-renormalised-potential-cut-off} for $E=\infty$.
Let $(I_\epsilon)_*\nu_t^{\cP_\epsilon}$ denote the pushforward measure of $\nu_t^{\cP_\epsilon}$ under the isometric embedding $I_\epsilon$. Then we have the following convergence to a unique limit as $\epsilon \to 0$. 

\begin{proposition}
\label{prop:weak-convergence-nu-t}
As $\epsilon \to 0$, we have for $t>0$ that, as measures on $H^\alpha(\Omega)$ for every $\alpha <1$,  $(I_{\epsilon})_*\nu_t^{\cP_\epsilon}$  converges weakly to $\nu_t^\cP$ given by
\begin{equs}
\label{eq:nu-t-continuum}
\E_{\nu_{t}^\cP}[F]
=
e^{v_\infty^0(0)}  \E \big[ F(Y_\infty - Y_t) e^{-v_t^0(Y_\infty-Y_t)} \big]
\end{equs}
for $F\colon H^\alpha (\Omega) \to \R$ bounded and measurable,
and where $e^{-v_\infty^0(0)} = \E\big[ e^{-\int_{\Omega} \wick{\cP(Y_{\infty})}dx}\big]$.

Moreover, for $t=0$, the weak convergence $(I_{\epsilon})_*\nu_0^{\cP_\epsilon} \to \nu_0^\cP$ holds as measures on $H^\alpha(\Omega)$ for any $\alpha <0$ and with $\nu_0^\cP$ defined by \eqref{eq:nu-t-continuum} with $t=0$ and $F\colon H^\alpha(\Omega) \to \R$ for $\alpha <0$.
\end{proposition}

\begin{remark}
Note that, although $\wick{\cP(Y_\infty)}$ is almost surely a distribution of negative regularity and not a function,
the expression $\int_\Omega \wick{\cP(Y_\infty)}dx$ makes sense as an abuse of notation to denote the duality pairing between a distribution and a constant function (which is smooth on the torus).
Furthermore, we sometimes include the spatial argument of the distribution as a further abuse of notation, e.g.\ $\int_\Omega Y_\infty(x) dx$, not to be confused with a pointwise evaluation (which may not exist).
\end{remark}

We first prove that the discrete potential, suitably extended, converges to the continuum potential in $L^2$. 

\begin{lemma}
\label{lem:wick-convergence}
As $\epsilon \to 0$, we have
  \begin{equs}
  \label{eq:wick-convergence}
    \E\Big[ \Big(\int_{\Omega_{\epsilon}} \wick{\cP(Y^\epsilon_{\infty})} dx - \int_{\Omega} \wick{\cP(Y_{\infty})} dx \Big)^{2} \Big] 
    \to
     0.
  \end{equs}
\end{lemma}
 
\begin{proof}
It is convenient to introduce an alternative extension operator to the trigonometric extension that behaves well under commutation with products.
We choose the piecewise constant extension of a function defined on $\Omega_\epsilon$ to a function defined on $\Omega$ that is piecewise constant on $x + \epsilon(-1/2, 1/2]^2$ for all $x \in \Omega_\epsilon \subseteq \Omega$.
Indeed, for any $f \in X_\epsilon$, define $E_\epsilon f \colon \Omega \rightarrow \R$ for $x \in \Omega$ by
\begin{equs}
\label{eq:v-0-convergence-l2}
E_\epsilon f (x)
=
\sum_{z \in \Omega_\epsilon} f(z) \mathbf{1}_{(-\epsilon/2,\epsilon/2]^2}(x-z).
\end{equs}

We identify $Y^{\epsilon}_{\infty}$ with $E_\epsilon Y^{\epsilon}_{\infty}$ and regard its covariance operator $c^\epsilon=c_\infty^\epsilon$ as an operator acting on such piecewise constant functions.
In the following computation, we abuse notation and evaluate the distribution $Y_{\infty}$ pointwise.
Although such a pointwise evaluation alone is a priori ill-defined, the covariance $E[Y^{\epsilon}_{\infty}(x) Y_{\infty}(y)]$ for $x,y \in \Omega_\epsilon$ yields a well defined expression,
as our computation below shows.
This can be made rigorous by first approximating $Y_{\infty}$ with smooth fields and then removing the approximation. We omit this to lighten the notation.
Thus, for  $x,y \in \Omega$,
  \begin{equs}
    \E[Y^{\epsilon}_{\infty}(x)Y_{\infty}(y)]
    &=
    \E \Big[ \sum_{z \in \Omega_\epsilon} Y^{\epsilon}_{\infty}(z) \mathbf{1}_{(-\epsilon/2,\epsilon/2]^2}(x-z)Y_{\infty}(y) \Big]
    \\
    &=
    \sum_{z \in \Omega_\epsilon} \mathbf{1}_{(-\epsilon/2,\epsilon/2]^2}(x-z) \sum_{k_1 \in \Omega_\epsilon^*, k_2 \in \Omega^*} \E \Big[ e^{ik_1 \cdot z+ik_2 \cdot y} \int_0^\infty \hat{q}_t^\epsilon d\hat W_t(k_1)\int_0^\infty \hat{q}_t^0 d \hat W_t(k_2) \Big]
    \\ 
    &= \sum_{z \in \Omega_\epsilon} \mathbf{1}_{(-\epsilon/2,\epsilon/2]^2}(x-z) \sum_{k\in \Omega_\epsilon^*} e^{ik\cdot (z-y) } \int_0^\infty \hat q_t^\epsilon(k) \hat q_t^0(k) dt 
    \\
    & \to \sum_{k\in \Omega^*} e^{ik\cdot (x-y) } \int_0^\infty | \hat q_t^0(k)|^2 dt = c(x-y),
  \end{equs}
  as $\epsilon \to 0$ and where we recall that $c=c_\infty^0$ is the kernel of $(-\Delta+ m^2 )^{-1}$.
  
To prove \eqref{eq:wick-convergence},  it is sufficient to show that as $\epsilon \to 0$
  \begin{equs}
    \E \Big[\Big(\int_{\Omega_{\epsilon}}\wick{(Y^{\epsilon}_{\infty})^{n}}dx -\int_{\Omega}\wick{Y_{\infty}^{n}} dx \Big)^{2}\Big] 
    \to 
    0.
  \end{equs}
  Note that, by Wick's theorem and an abuse of notation regarding evaluating distributions pointwise as above, we have
  \begin{equs}
  \E[\wick{(Y_\infty^\epsilon)^n(x)} \wick{Y_\infty^n(y)}]
  &=
  \sum_{z_1,\dots,z_n} \prod_{i=1}^n \mathbf{1}_{(-\epsilon/2,\epsilon/2]^2}(x-z_i) \, \E [ \wick{ \prod_{i=1}^n Y_\infty^\epsilon(z_i) } \wick{Y_\infty^n (y)} ]
  \\
  &=
  n!\sum_{z_1,\dots,z_n} \prod_{i=1}^n \mathbf{1}_{(-\epsilon/2,\epsilon/2]^2}(x-z_i)   \, \prod_{i=1}^n \E [   Y_\infty^\epsilon(z_i)  Y_\infty (y)]
  \\
  &=
  n!\E[ (Y_\infty^\epsilon)(x) Y_\infty(y) ]^n.
  \end{equs}
  Hence, by expanding, we have that the expression on the left-hand side of \eqref{eq:v-0-convergence-l2} is equal to
  \begin{equs}
    &\int_{\Omega_{\epsilon}}\int_{\Omega_{\epsilon}}\E[ \wick{(Y^{\epsilon}_{\infty})^{n} (x)} \wick{(Y^{\epsilon}_{\infty})^{n}(y)}] dx dy 
    + \int_{\Omega} \int_{\Omega} \E[  \wick{(Y_{\infty})^{n}(x)} \wick{(Y_{\infty})^{n}(y)}] dxdy 
    \\
     &-2\int_{\Omega}\int_{\Omega_{\epsilon}} \E [\wick{(Y_{\infty})^{n}(x)} \wick{(Y^{\epsilon}_{\infty})^{n}(y)}] dx dy 
     \\
    = \,&\, n!\int_{\Omega_{\epsilon}} \int_{\Omega_{\epsilon}} (c^{\epsilon}(x-y))^{n} dx dy + n!\int_{\Omega} \int_{\Omega} c^n(x-y) dx dy - 2n! \int_{\Omega}\int_{\Omega}  \E[Y^{\epsilon}_{\infty}(x) Y_{\infty}(y)]^n dxdy,
  \end{equs}
  which converges to $0$ by dominated convergence.
  \end{proof}
  
We need the following exponential integrability lemma, which follows from a by-now standard argument due to Nelson, see \cite[Chapter 9.6]{MR887102}.
Note that one may also prove this using the Bou\'e-Dupuis formula and estimates similar to those in Section \ref{ssec:l2-estimates}, see Remark \ref{rem:lp-integrability}.
The convergence statement below then follows by Vitali's theorem as stated in \cite[Theorem 4.5.4]{Bogachev2007MeasureTheory} and the convergence \eqref{eq:wick-convergence}.

\begin{lemma}
\label{lem:renormalised-potential-convergence}
    For any $0\leq p<\infty$, we have 
  \begin{equs}
  \label{eq:renormalised-potential-exponential-moments}
    \sup_{\epsilon>0}\E\Big[\exp\Big(-p\int_{\Omega_{\epsilon}}\wick{\cP(Y_\infty^\epsilon)} dx\Big)\Big]
    <
    \infty.
  \end{equs}
In particular, as $\epsilon \to \infty$
  \begin{equs}
  \label{eq:v-infty-convergence}
  \E\Big[\exp\Big(-\int_{\Omega_{\epsilon}}\wick{\cP(Y_{\infty}^\epsilon)} dx\Big)\Big] 
  \to 
  \E\Big[\exp\Big(-\int_{\Omega}\wick{\cP(Y_{\infty})} dx\Big)\Big].
\end{equs}
\end{lemma}

Combining Lemma \ref{lem:wick-convergence} and Lemma \ref{lem:renormalised-potential-convergence}, we can now give a proof of Proposition \ref{prop:weak-convergence-nu-t}.
\begin{proof}[Proof of Proposition \ref{prop:weak-convergence-nu-t}]
Recall from \eqref{eq:ren-measure} that the renormalised measure $\nu_t^{\cP_\epsilon}$ is defined by
\begin{equs}
    \E_{\nu_t^{\cP_\epsilon}}[F]
    =
    { e^{v_\infty^\epsilon(0)}  \EE_{c_\infty^\epsilon - c_t^\epsilon} [F(\zeta) e^{-v_t^\epsilon(\zeta)}  ] }
    =
      { e^{v_\infty^\epsilon(0)} } \E[F(Y^{\epsilon}_{\infty}-Y^{\epsilon}_{t})   e^{-v^{\epsilon}_{t}(Y^{\epsilon}_{\infty}-Y^{\epsilon}_{t})}],
\end{equs}
where $F\colon X_\epsilon \to \R$ is  bounded and continuous, and $v^\epsilon_t$ is the renormalised potential.
By the definition of the pushforward measure and the renormalised potential we obtain that, for $F\colon H^\alpha(\Omega)  \to \R$, bounded and continuous,
\begin{equs}
\E_{(I_{\epsilon})_*\nu^{\cP_\epsilon}_{t}}[F]
&= \E_{\nu_t^{\cP_\epsilon}} [F \circ I_\epsilon] 
=  e^{v_\infty^\epsilon (0)} \E [F \big( I_\epsilon (Y_\infty^\epsilon - Y_t^\epsilon) \big)  e^{-v_t^\epsilon(Y_\infty^\epsilon - Y_t^\epsilon)} ] =  
\\
&= { e^{v_\infty^\epsilon (0)} } \E\Big[ F\big(I_{\epsilon}(Y_\infty^\epsilon -Y_t^\epsilon)\big) 
\E \big[e^{- v_0^\epsilon(Y_\infty^\epsilon - Y_t^\epsilon + Y_t^\epsilon) } \bigm| \cF^t \big] \Big]
\\
&= e^{v_\infty^\epsilon (0)}  \E\Big[F\big(I_\epsilon(Y_\infty^\epsilon - Y_t^\epsilon)\big) e^{-v_0^\epsilon (Y_\infty^\epsilon)} \Big].
\end{equs}
Now, we have by \eqref{eq:gff-limit-eps-to-positive} that $I_{\epsilon}(Y^{\epsilon}_{\infty}-Y^{\epsilon}_{t})$ converges to $Y_{\infty}-Y_{t}$ in $L^2$ with respect to the norm of $H^\alpha(\Omega)$ for any $\alpha <1$.
Moreover, we have by \eqref{eq:wick-convergence} that $v_0^\epsilon (Y_\infty^\epsilon) \to v_0^0(Y_\infty)$ in $L^2$.
Take any subsequence, which we continue to denote as $\epsilon$. Then, there is a further subsequence $(\epsilon_k)_k$, along which we have 
\begin{equation}
\label{eq:vitali-convergence-as}
F\big(I_{\epsilon_k}(Y_\infty^{\epsilon_k} - Y_t^{\epsilon_k})\big) e^{-v_0^{\epsilon_k} (Y_\infty^{\epsilon_k})}
\to
F( Y_\infty - Y_t) e^{-v_0^0 (Y_\infty)}
\end{equation}
almost surely, where we also used that $F$ is continuous with respect to the norm on $H^\alpha$.
Since $F$ is bounded, we have by \eqref{eq:renormalised-potential-exponential-moments} that 
\begin{equation}
\Big (F\big(I_\epsilon(Y_\infty^{\epsilon} - Y_t^{\epsilon})\big) e^{-v_0^{\epsilon} (Y_\infty^{\epsilon})} \Big)_\epsilon
\end{equation}
is uniformly integrable. Hence, by Vitali's theorem, the convergence in \eqref{eq:vitali-convergence-as} holds in $L^1$, i.e.\
\begin{equation}
\label{eq:vitali-convergence-l1}
\E \Big [F\big(I_{\epsilon_k}(Y_\infty^{\epsilon_k} - Y_t^{\epsilon_k})\big) e^{-v_0^{\epsilon_k} (Y_\infty^{\epsilon_k})} \Big]
\to
\E\Big[ F( Y_\infty - Y_t) e^{-v_0^0 (Y_\infty)} \Big].
\end{equation}
In summary, we have shown that for every subsequence of $\epsilon$ there is a further subsequence $(\epsilon_k)_k$ such that \eqref{eq:vitali-convergence-l1} holds,
thus showing full convergence of \eqref{eq:vitali-convergence-l1}.

For the case $t=0$ we follow the same arguments as for $t>0$, but now we take $F\colon H^\alpha (\Omega) \to \R$ for $\alpha <0$ and use \eqref{eq:gff-limit-eps-to-0}. 
\end{proof}

\section{Convergence in law for the maximum of $\cP(\phi)_2$}
\label{sec:maximum}

In this section we use the results on the difference field $\Phi^\Delta$ and prove that the maximum of the $\cP(\phi)_2$ field converges in distribution to a randomly shifted Gumbel distribution.
The analogous result was recently established for the sine-Gordon field in \cite[Section 4]{MR4399156}.
The main difficulty in this reference is to deal with the non-Gaussian and non-independent term $\Phi_0^\Delta$, which requires generalising and extending several key results of \cite{MR3433630}. 
In the present case the combination of the Polchinski approach and the Bou\'e-Dupuis variational approach produce a similar situation with the essential difference to the sine-Gordon case being different regularity estimates for the difference field.
Thus, the main goal in this section is to argue that all results in \cite[Section 4]{MR4399156} also hold under the modified assumptions on $\Phi^\Delta$.

From now on, when no confusion can arise, we will drop $\epsilon$ from the notation. Moreover, to ease notation,
we will use the notation $\| \varphi \|_{L^\infty(\Omega_\epsilon)} \equiv \| \varphi \|_{\infty}$ for fields $\varphi \colon \Omega_\epsilon \to \R$.
Recall that by Theorem \ref{thm:coupling-pphi-to-gff-eps} we have that
\begin{equation}
\label{eq:scale-coupling-max}
\Phi_0^{\cP} = \Phi_0^\GFF + \Phi_0^\Delta = \Phi_0^\GFF - \Phi_s^\GFF + \Phi_s^\cP +  R_s,
\end{equation}
where $R_s = \Phi_0^\Delta -\Phi_s^\Delta$ satisfies
$\sup_{\epsilon>0} \E [\|R_s\|_\infty^{2/L} ] \to 0$ as $s\to 0$ by the Sobolev embedding and \eqref{eq:phi-delta-to-0}.
In the analysis of the maximum of $\Phi_0^\cP$  we also need the following continuity result for the field $\Phi^\Delta$.

\begin{lemma}
Let $\alpha \in (0,1)$. Then
\begin{equation}
\label{eq:phi-delta-hoelder}
\sup_{\epsilon>0} \sup_{t\geq 0}\E \big[\| \Phi_t^{\Delta} \|^{2/L}_{C^{\alpha}(\Omega)} \big] < \infty.
\end{equation}
\end{lemma}

\begin{proof}
The statement follows from the Sobolev-H\"older embedding in Proposition \ref{prop:hoelder-sobolev} and \eqref{eq:phi-delta-1-2}.
\end{proof}

To express convergence in distribution we will use the L\'evy distance $d$ on the set of probability measures on $\R$, which is a metric for the topology of weak convergence.
It is defined for any two probability measures $\nu_1,\nu_2$ on $\R$ by
\begin{equation} \label{e:Levy}
d(\nu_1,\nu_2)= \min \{\kappa >0 \colon \nu_1 (B) \leq \nu_2(B^\kappa)+\kappa \text{~for all open sets~} B \},
\end{equation}
where $B^\kappa=\{y\in \R \colon \dist (y,B) < \kappa\}$.
We will use the convention that when a random variable appears in the argument of $d$, we refer to its distribution on $\R$.
Note that if two random variables $X$ and $Y$ can be coupled with $|X-Y|\leq \kappa$ with probability $1-\kappa$ then $d(X,Y) \leq \kappa$.

\subsection{Reduction to an independent decomposition}
The first important step is the introduction of a scale cut-off $s>0$ to obtain an independent decomposition from \eqref{eq:scale-coupling-max}. More precisely, we write
\begin{equation}
\label{eq:approximate-independent-decomposition}
\Phi_0^{\cP} = \tPhisP +  R_s, \qquad \tPhisP =  (\Phi_0^\GFF - \Phi_s^\GFF) + \Phi_s^\cP
\end{equation}
and argue that we may from now on focus on the auxiliary field $\tPhisP$. The following statement plays the same role as Lemma 4.1 in \cite{MR4399156}.

\begin{lemma}[Analog to {\cite[Lemma 4.1]{MR4399156}}]
\label{lem:reduction-to-phi-tilde-phi4}
  Assume that the limiting law $\tilde \mu_s$ of $\max_{\Omega_\epsilon} \sqrtpi \tPhisP -\mathfrak{m}_\epsilon$ as $\epsilon\to 0$ exists for every $s>0$,
  and that there are positive random variables $\ZDM_s$ (on the above common probability space) such that
  \begin{equation}   
  \label{eq:mu-s-gumbel}
    \tilde \mu_s((-\infty,x])=\E[e^{-\alpha^* \ZDM_s e^{-\scaling x}}]
  \end{equation}
  for some constant $\alpha^*>0$.
  Then the law of $\max_{\Omega_\epsilon} \sqrtpi \Phi_0^\cP - \mathfrak{m}_\epsilon$ converges weakly to
  some probability measure $\mu_0$ as $\epsilon \to 0$ and $\tilde \mu_s \rightharpoonup \mu_0$ weakly as $s\to 0$.
  Moreover,  there is a positive random variable
  $\ZDM^\cP$ such that
  \begin{equation}  
    \mu_0((-\infty,x])=\E[e^{-\alpha^* \ZDM^\cP e^{-\scaling x}}].
  \end{equation}
\end{lemma}

\begin{proof}
We follow identical steps as in the proof of Lemma 4.1 of \cite{MR4399156}. We first argue that the sequence $(\Max \Phi_0^\cP - \mathfrak{m}_\epsilon)_\epsilon$ is tight.
Indeed, since
\begin{equation}
\Max \Phi_0^\cP - \mathfrak{m}_\epsilon = \Max \Phi_0^\GFF -\mathfrak{m}_\epsilon + O( \| \Phi_0^\Delta \|_\infty)
\end{equation}
and since $(\Max \Phi_0^\GFF - \mathfrak{m}_\epsilon)_\epsilon$ is tight by \cite{MR2846636}, 
we see that $(\Max \Phi_0^\cP - \mathfrak{m}_\epsilon)_\epsilon$ differs from a tight sequence by a sequence $Y_\epsilon$ with $\E[|Y_\epsilon|^{2/L}] <\infty$.
Elementary arguments such as Markov's inequality then imply that also sequence $(\Max \Phi_0^\cP - \mathfrak{m}_\epsilon)_\epsilon$ is tight.

Thus, there is a probability distribution $\mu_0$ such that the law of $\Max \Phi^\cP - \mathfrak{m}_\epsilon$ converges to $\mu_0$ weakly along a subsequence $(\epsilon_k)_k$.
Considering the L\'evy distance to $\tilde \mu_s$, we have
\begin{align}
d(\mu_0, \tilde \mu_s) & \leq 
\limsup_{\epsilon = \epsilon_k \to 0} [d(\Max \Phi_0^\cP - \mathfrak{m}_\epsilon, \Max \tPhisP - \mathfrak{m}_\epsilon ) + d(\Max \tPhisP - \mathfrak{m}_\epsilon, \tilde \mu_s) ]
\nnb
&=\limsup_{\epsilon=\epsilon_k \to 0} d(\Max \Phi_0^\cP - \mathfrak{m}_\epsilon, \Max \tPhisP - \mathfrak{m}_\epsilon).
\end{align}
The last display can be estimated as follows: for any open set $B \subseteq \R$ we have
\begin{align}
\P(\Max \Phi_0^\cP - \mathfrak{m}_\epsilon \in B)
& \leq
\P(\Max \tilde \Phi_s - \mathfrak{m}_\epsilon \in B \pm \|R_s\|_\infty) \nnb
&\leq \P (\Max \tPhisP - \mathfrak{m}_\epsilon \in B^\kappa ) + \P(\Max \|R_s\|_\infty \geq \kappa).
\end{align}
Markov's inequality implies that
\begin{equation}
\P(\|R_s\|_\infty \geq \kappa ) \leq \frac{\E[\|R_s\|_\infty^{2/L} ] }{\kappa^{2/L}},
\end{equation}
and thus, choosing
\begin{equation}
\kappa = \frac{\E[\|R_s\|_\infty^{2/L} ] }{\kappa^{2/L}} \iff \kappa = \big(\E[ \|R_s\|_\infty^{2/L}]\big)^{L/(L+2)},
\end{equation}
we get by the definition of the Levy distance
\begin{equation}
\label{eq:levy-distance-mu0-mus}
d(\mu_0, \tilde \mu_s) \lesssim \big(\sup_{\epsilon >0} \E[\|R_s\|_{\infty}^{2/L} ] \big)^{L/(L+2)}.
\end{equation}
Taking $s\to 0$ shows that  the subsequential limit $\mu_0$ is unique, as it is the unique weak limit of $\tilde \mu_s$.
Thus, it follows that $\Max \Phi_0^\cP - \mathfrak{m}_\epsilon \to \mu_0$ in distribution as $\epsilon \to 0$.
Moreover, by \eqref{eq:levy-distance-mu0-mus} we have $\tilde \mu_s \rightharpoonup \mu_0$ weakly, and thus
\begin{equation}
\label{eq:mus-to-mu0-distribution}
\tilde \mu_s((-\infty, x]) \to \mu_0((-\infty, x])
\end{equation}
for the distribution functions. 

It remains to show that $(\ZDM_s)_s$ is tight. Using \eqref{eq:approximate-independent-decomposition} we get for any $C>0$
\begin{equs}
\begin{split}
\P(\Max \tPhisP -\mathfrak{m}_\epsilon \leq x) & 
\geq \P(\Max \Phi_0^\GFF - \mathfrak{m}_\epsilon  \leq x - \|\Phi_s^\Delta \|_\infty) \nnb
&\geq \P(\Max \Phi_0^\GFF - \mathfrak{m}_\epsilon  \leq x - C, \|\Phi_s^\Delta \|_\infty \leq C) \nnb
\label{eq:tightness-zs-probabilities}
& = \P( \Max \Phi_0^\GFF - \mathfrak{m}_\epsilon \leq x- C)
- 
\P(\Max \Phi_0^\GFF - \mathfrak{m}_\epsilon \leq x-C, \| \Phi_s^\Delta\|_\infty > C).
\end{split}
\end{equs}
Now, for a given $\kappa >0$ we use Markov's inequality and choose $C$ such that
\begin{equation}
\sup_{s\geq 0} \P(\|\Phi_s^\Delta\|_\infty > C) \leq \kappa/2.
\end{equation}
Then we obtain from \eqref{eq:tightness-zs-probabilities}
\begin{equation}
\label{eq:tightness-zs-inequality-probabilities}
-\kappa/2 + \P(\Max \Phi_0^\GFF - \mathfrak{m}_\epsilon \leq x-C)   \leq \P(\Max \tPhisP -\mathfrak{m}_\epsilon \leq x).
\end{equation}
Let $\mu_0^\GFF$ be the limiting law of the centred maximum of the discrete Gaussian free field which exists by \cite{MR3433630}. Then taking $\epsilon \to 0$ in \eqref{eq:tightness-zs-inequality-probabilities} together with the assumption \eqref{eq:mu-s-gumbel} yields
\begin{equation}
\label{eq:tightness-inequality-exp}
-\kappa/2 + \mu_0^\GFF((-\infty, x - C] ) \leq   \E[e^{-\alpha^* \ZDM_s e^{-\scaling x}}].
\end{equation}
From here the argument is analogous as in the proof of \cite[Lemma 4.1]{MR4399156}:
assume that the sequence $(\ZDM_s)_s$ is not tight.
Then we have
\begin{equation}
\exists \kappa >0 \colon \forall M>0 \colon \exists s_M \colon \P(\ZDM_{s_M} >M) >\kappa.
\end{equation}
It follows that 
\begin{equation}
\E[e^{-\alpha^* \ZDM_{s_M} e^{-\scaling x}}]
\leq e^{-\alpha^* M e^{-\scaling x}} + \P(\ZDM_{s_M} \leq M) \leq  e^{-\alpha^* M e^{-\scaling x}} + (1-\kappa).
\end{equation}
Sending $M\to \infty$, \eqref{eq:tightness-inequality-exp} implies that 
\begin{equation}
- \kappa/2 + \mu_0^\GFF((-\infty, x-C]) \leq 1-\kappa,
\end{equation}
which is a contradiction when sending $x\to \infty$.
\end{proof}

\subsection{Approximation of small scale field}
\label{sec:approx-small-scales}

Thanks to Lemma \ref{lem:reduction-to-phi-tilde-phi4} we may focus from now on the centred maximum of $\tPhisP$,
which has a Gaussian small scale field $\Phi_0^\GFF - \Phi_s^\GFF$ and a non-Gaussian but independent large scale field $\Phi_s^\cP$. 
Similar to \cite[Section 4.2]{MR4399156} we replace the small scale field by a collection of massless discrete Gaussian free fields, so that the results in \cite{MR3433630} apply.
The only difference is the regularisation of the Gaussian free field covariance:
in \cite{MR4399156} the heat-kernel regularisation is used, i.e.\
\begin{equation}
\frac{d}{dt} \tilde c_t^\epsilon = e^{t(-\Delta^\epsilon + m^2)/2}, \qquad \tilde c_t^\epsilon = \int_0^t \frac{d}{ds}\tilde c_s^\epsilon \, ds,
\end{equation}
while here, we use the Pauli-Villars regularisation \eqref{eq:pauli-villars}, which implies
\begin{equation}
\label{eq:covariance-small-scales}
\cov(\Phi_0^\GFF - \Phi_s^\GFF) = (-\Delta + m^2 + 1/s)^{-1}.
\end{equation}
Therefore, we do not need the additional decomposition (4.17) in \cite{MR4399156} involving the function $g_s$. The following paragraph is completely analogous to \cite[Section 4.2]{MR4399156}, but we include it here to set up the notation and improve readability.

We introduce a macroscopic subdivision of the torus $\Omega$ as follows:
let $\Gamma$ be the union of horizontal and vertical lines intersecting at the vertices $\frac{1}{K}\Z^2 \cap \Omega$ which subdivides $\Omega$
into boxes $V_i \subset \Omega$, $i=1,\dots, K^2$ of side length $1/K$.
We use the notation $V_i$ for both the subset of $\Omega$ and the corresponding lattice version as subset of $\Omega_\epsilon$. 

Let $\Delta_\Gamma$ be the Laplacian on $\Omega$ with Dirichlet boundary conditions on $\Gamma$,
and let $\Delta$ be the Laplacian with periodic boundary conditions on $\Omega$.
The domain of $\Delta$ is the space of $1$-periodic functions,
and that of $\Delta_\Gamma$ is the smaller space of $1$-periodic functions vanishing on $\Gamma$.
This implies that $-\Delta_\Gamma \geq -\Delta$ and thus,
\begin{equation}
\label{eq:covariances-quadratic-form-inequality}
(-\Delta+m^2 + 1/s)^{-1} \geq (-\Delta_\Gamma+m^2 + 1/s)^{-1}
\end{equation}
as quadratic form inequalities.

Hence, using \eqref{eq:covariance-small-scales}, we can decompose the small scale Gaussian field $\Phi_0^\GFF- \Phi_s^\GFF$ as
\begin{equation}
  \Phi^{\GFF}_0-\Phi^{\GFF}_s  \stackrel{d}{=} \tilde X_{s,K}^f + \coarsesKGauss,
  \label{e:decomp-gaussian-part}
\end{equation}
where the two fields on the right-hand side are independent Gaussian fields with covariances
\begin{align}
  \label{e:Xf}
  \cov(\tilde X_{s,K}^f) &= (-\Delta_\Gamma + m^2 + 1/s)^{-1}  \\
    \label{e:Xc}
  \cov(\coarsesKGauss) &=  (-\Delta+m^2 + 1/s)^{-1} - (-\Delta_\Gamma+m^2 + 1/s)^{-1}.
\end{align}
Note that for this decomposition, which is analogous to the Gibbs-Markov decomposition of the massless GFF with Dirichlet boundary condition,
the Pauli-Villars decomposition is particularly convenient due to \eqref{eq:covariances-quadratic-form-inequality}.
Using \cite[Lemma 4.2]{MR4399156} in the exact same form, we see that the maximum of $\Phi_0^\GFF- \Phi_s^\GFF$ can be replaced by the maximum of $X_K^f + \coarsesKGauss$, where
\begin{equation}
\cov(X_K^f) = (-\Delta_\Gamma)^{-1}.
\end{equation}
This yields a new auxiliary field denoted $\Phi_s$ with independent decomposition
\begin{equation}
\label{eq:independent-decomposition-phi-s}
\Phi_s = X_K^f + \coarsesKGauss + \Phi_s^\cP,
\end{equation}
which is completely analogous to (4.43) in \cite{MR4399156}, except that there is no field $X_s^h$ for the different choice of the covariance regularisation.
In fact the two fields $X_K^f$ and $\coarsesKGauss$ are exactly the same 
and thus, the covariance estimates for $\coarsesKGauss$ in \cite[Lemma 4.4]{MR4399156} can be used verbatim. 
The only essential difference is that the field $\Phi_s^\cP$ is different, but the following statement establish the same regularity estimates as in \cite[Lemma 4.5]{MR4399156}.

\begin{lemma}
For any $s>0$ and $\epsilon\geq 0$ the fields $\Phi_s^\GFF$ and $\Phi_s^\cP$ are a.s.\ H\"older continuous. Moreover, there is $\alpha \in {(0,1)}$ such that for $\#\in \{\GFF, \cP\}$
\begin{equation}
\label{eq:hoelder-continuity-small-scales}
\sup_{\epsilon>0} \E\Big [  \Max |\Phi_s^\#|  + \max_{x,y \in \Omega_\epsilon} \frac{|\Phi_s^\#(x) - \Phi_s^\# (y)|}{|x-y|^\alpha} \Big] < \infty.
\end{equation}
\end{lemma}

\begin{proof}
Note that the Fourier coefficients of $\Phi_t^\GFF$ satisfy
\begin{equation}
\E \big[ |\hat \Phi_s^\GFF (k)|^2 \big] = O_{s} \Big(\frac{1}{1+|k|^4}\Big).
\end{equation}	
Hence, \eqref{eq:hoelder-continuity-small-scales}  for $\Phi_s^\GFF$ follows from standard results as stated in \cite[Proposition B.2 (i)]{MR3339158} (for H\"older continuity) and \cite[Lemma 3.5]{MR3433630} (for the maximimum). For $\Phi_s^\cP$ the results follow from the properties of the difference term $\Phi_s^\Delta$.
\end{proof}

The only remaining results, where the properties of $\Phi^\Delta$ enter, are Proposition 4.8 and Pro\-po\-si\-tion 4.9 in \cite{MR4399156}.
To state these results we introduce for technical reasons a small neighbourhood of the grid $\Gamma$ as follows.
For $\delta \in (0,1)$, define 
\begin{equation}
\label{e:Omega-delta}
V_i^{\delta}= \{x \in V_i \colon \dist(x, \Gamma) \geq \delta/K \},
\qquad
\Omega^\delta= \bigcup_{i=1}^{K^2} V_i^{\delta},
\qquad
\DiscTorusGrid= \Omega^\delta \cap (\epsilon \Z^2).
\end{equation}

\begin{proposition}[Version of {\cite[Proposition~5.1]{MR3433630}}]
\label{prop:maximiser-not-on-grid-phi4}
Let $\DiscTorusGrid$ be as in \eqref{e:Omega-delta}. Then,
\begin{equation}
\lim_{\delta \to 0}\limsup_{K\to \infty} \limsup_{\epsilon \to 0} \P (\max_{\DiscTorusGrid}\PhisKP\neq \max_{\Omega_\epsilon}\PhisKP)=0.
\end{equation}
\end{proposition}

\begin{proposition}[Version of {\cite[Proposition~5.2]{MR3433630}}]
\label{prop:fine-field-maximisers-phi4}
Let $\PhisKP$ be as in \eqref{eq:independent-decomposition-phi-s}. Let $z_i \in V_i^{\delta}$ be such that 
\begin{equation}
\max_{V_i^{\delta}} \fineGFF=  \fineGFF(z_i)
\end{equation}
and let $\bar z$  be such that 
\begin{equation}
\max_i \PhisKP(z_i) = \PhisKP(\bar z).
\end{equation}
Then for any fixed $\kappa>0$ and small enough $\delta >0$,
\begin{equation}
\lim_{K\to \infty} \limsup_{\epsilon \to 0} \P (\max_{\DiscTorusGrid} \PhisKP \geq \PhisKP(\bar z)+\kappa)=0.
\label{e:Psi-fine-field-maximisers}
\end{equation}
Moreover, there is a function $g\colon \N \to \R_0^+$ with $g(K) \to \infty$ as $K\to \infty$, such that 
\begin{equation}
\lim_{K\to \infty} \limsup_{\epsilon \to 0} \P ( \sqrtpi \fineGFF(\bar z)\leq \mathfrak{m}_{\epsilon K}+g(K))=0.
\label{e:GFF-fine-field-maximisers}
\end{equation}
\end{proposition}

\begin{proof}[Proof of Proposition \ref{prop:maximiser-not-on-grid-phi4} and Proposition  \ref{prop:fine-field-maximisers-phi4}]

In the case of the sine-Gordon field it was used that $\|\Phi_s^\Delta\|_\infty$ is bounded by a deterministic constant and that $\Phi_t^\Delta$ is H\"older continuous.
Thus, the generalisation of these results are immediate when restricting to the event
\begin{equation}
E = \{ \|\Phi_s^\Delta\|_{C^\alpha(\Omega)}  < C \},
\end{equation}
whose probability is arbitrarily close to $1$ if $C$ is large enough.
\end{proof}

\subsection{Approximation by $\epsilon$-independent random variables}

Following \cite[Section 2.3]{MR3433630} we approximate $\Max \PhisKP -\mathfrak{m}_\epsilon$ by
$G^*_{s,K}= \max_{i}G_{s,K}^i$ where 
\begin{equation}
G_{s,K}^i= \rho_{K}^i (Y_{K}^i+g(K)) + \limcoarsePhiP (\bu_\delta^i) - \scalinglog \log K,
\label{e:approximation-maximum}
\end{equation}
and also define
\begin{equation} \label{e:cZdef}
 \ZDM_{s,K}=m_\delta \frac{1}{K^2} \sum_{i=1}^{K^2} (\scalinglog \log K - \limcoarsePhiP(\bu_\delta^i)) e^{ -2\log K +\scaling \limcoarsePhiP(\bu_\delta^i) }.
\end{equation}
Here, the sequence $g(K)$ is as in Proposition~\ref{prop:fine-field-maximisers-phi4} 
and the random variables in \eqref{e:approximation-maximum} and \eqref{e:cZdef}
are all independent and defined as follows:
\begin{itemize}
\item
  The random variables $\rho_{K}^i \in \{0,1\}$, $i=1,\dots, K^2$, are independent Bernoulli random variables
with $\bbP(\rho_{K}^i=1)=\alpha^*  m_\delta g(K) e^{-\scaling g(K)}$
with $\alpha^*$ and $m_\delta$ as in \cite[Proposition 4.6]{MR4399156}.
\item
  The random variables $Y_{K}^i \geq 0$, $i=1,\dots, K^2$,
  are independent and characterised by $\bbP(Y_{K}^i\geq x)= \frac{g(K)+x}{g(K)}e^{- \scaling x}$ for $x\geq 0$.
\item
  The random field $\limcoarsePhiP(x)$, $x\in\Omega$, is a weak limit of the overall coarse field
  $\coarsePhiP \equiv \coarsesKGauss+\Phi_s^\cP$ as $\epsilon \to 0$.
  The existence of this limit is guaranteed by Corollary \ref{cor:pphi-coupling-continuum}
  and \cite[Lemma 4.4]{MR4399156}.
\item
  The random variables $\bu_\delta^i \in V_i^\delta$, $i=1,\dots,K^2$,
  have the limiting distribution of the maximisers $z_i$ of $\fineGFF$ in $V_i^{\delta}$
  as $\epsilon\to 0$.
  Thus, $\bu_\delta^i$ takes values in the $i$-th subbox of $\Omega=\T^2$ and,
  scaled to the unit square, its
  density is $\psi^\delta$ as in 
  \cite[Proposition 4.6]{MR4399156}.
\end{itemize}
Note that the correction in \eqref{e:approximation-maximum} can be understood from
\begin{equation}
\label{e:correction-approximation}
\mathfrak{m}_{\epsilon  K} - \mathfrak{m}_{\epsilon}= - \scalinglog \log K + O_K(\epsilon).
\end{equation}

From here the rest of \cite[Section 4]{MR4399156} can be used without adjustment, as no other properties of the difference field are used. 
We omit further details.

\section*{Acknowledgements}

TSG would like to thank Ajay Chandra for interesting discussions on the Polchinski equation and Romain Panis for useful comments.
MH thanks Roland Bauerschmidt for discussing the project at early stages, as well as Beno\^it Dagallier for the helpful comments.  NB, TSG, and MH would like to thank Hugo Duminil-Copin for hosting us at the Universit\'e de Gen\`eve in April 2022 to work on this project. 

TSG was supported by the Simons Foundation, Grant 898948, HDC. MH was partially supported by the UK EPSRC grant EP/L016516/1 
for the Cambridge Centre for Analysis. NB is supported by the ERC Advanced Grant 741487 (Quantum Fields and Probability).

\section*{Declarations}

\textit{Conflict of interest:} The authors have no relevant financial or non-financial interests to disclose.

\noindent
\textit{Data availability statement:} Not applicable to this article, as no datasets were generated or analysed.

\noindent
\textit{Author contribution:} All authors wrote the manuscript and contributed equally.

\bibliographystyle{abbrv}
\bibliography{all}
\end{document}